\documentclass[12pt]{article}
\usepackage{latexsym, amscd, amsfonts, eucal, mathrsfs, amsmath, mathabx, amssymb, amsthm, xypic,xr, stmaryrd, color, enumerate, tikz}
\usepackage{appendix}
\usepackage{multicol}
\usepackage[all]{xy}
\usepackage{hyperref}
\usepackage{fullpage}
\usepackage{tikz-cd}

\newtheorem{thm}{Theorem}
\newtheorem*{thm*}{Theorem}

\newtheorem{theorem}{Theorem}[subsection]
\newtheorem{lemma}[theorem]{Lemma}
\newtheorem{proposition}[theorem]{Proposition}
\newtheorem{corollary}[theorem]{Corollary}

\theoremstyle{definition}
\newtheorem{question}[theorem]{Question}
\newtheorem{definition}[theorem]{Definition}
\newtheorem{construction}[theorem]{Construction}
\newtheorem{recollection}[theorem]{Recollection}
\newtheorem{convention}[theorem]{Convention}
\newtheorem{example}[theorem]{Example}

\newtheorem*{conj}{Conjecture}
\newtheorem*{cor}{Corollary}

\newtheorem{notation}[theorem]{Notation}
\newtheorem{warning}[theorem]{Warning}
\newtheorem{remark}[theorem]{Remark}
\newtheorem{variant}[theorem]{Variant}

\DeclareMathOperator*{\holim}{holim}

\DeclareMathOperator*{\colim}{colim}

\DeclareMathOperator*{\K}{\mathrm{K}}

\usepackage{cleveref}
\usepackage{scrextend}

\begin{document}

\title{Redshift and multiplication for truncated Brown-Peterson spectra}
\author{Jeremy Hahn and Dylan Wilson}
\date{}
\maketitle

\begin{abstract}
We equip $\mathrm{BP} \langle n \rangle$ with an $\mathbb{E}_3$-$\mathrm{BP}$-algebra structure, for each prime $p$ and height $n$.  The algebraic $K$-theory of this ring is of chromatic
height exactly $n+1$, and the map $\mathrm{K}(\mathrm{BP}\langle n \rangle)_{(p)} \to \mathrm{L}_{n+1}^{f} \mathrm{K}(\mathrm{BP}\langle n\rangle)_{(p)}$ has bounded above fiber.
\end{abstract}

\tableofcontents

\section{Introduction}

Our main aim here is to prove the following:
\begin{thm*}
For each prime $p$ and height $n$, there exists an $\mathbb{E}_3$-$\mathrm{BP}$-algebra structure on $\mathrm{BP}\langle n \rangle$.  
The
algebraic $K$-theory of the $p$-completion of
this ring has finitely presented cohomology over the mod $p$ Steenrod algebra, and is of fp-type $n+1$ after $p$-completion.
\end{thm*}
The principal connective theories in the chromatic approach to stable homotopy theory are thus more structured than previously known, and they satisfy higher height analogs of the Lichtenbaum--Quillen conjecture.   The $\mathbb{E}_3$ forms of $\mathrm{BP}\langle n\rangle$ constructed here give the first known examples, for $n>1$, of chromatic height $n$ theories with algebraic $K$-theory provably of height $n+1$.

\subsection*{The Redshift Philosophy}

In his 1974 ICM address, Quillen \cite{quillen} stated
as a `hope' the now proven
Lichtenbaum-Quillen conjecture \cite{voeI,voeII}.
His hope was that the algebraic
$K$-theory of regular noetherian rings could be well approximated by
\'etale cohomology, at least in large degrees. Ten years later,
Waldhausen
\cite{waldhausen} investigated interactions between
his $K$-theory of spaces and the chromatic filtration.
He observed
that, in the presence of
a descent theorem of Thomason \cite{thomason}, the Lichtenbaum-Quillen conjecture could be restated in terms of
localization at complex $K$-theory. Let $L_1^f$ denote the localization
that annihilates those finite
spectra with vanishing $p$-adic complex $K$-theory; for suitable rings $R$,
the Lichtenbaum-Quillen conjecture is equivalent to the
statement that
	\[
	\pi_*\mathrm{K}(R)_{(p)} \to \pi_*L_1^f\mathrm{K}(R)_{(p)}
	\]
is an isomorphism for $\ast\gg 0$.

Algebraic $K$-theory is defined not only on rings, but (crucially for applications to smooth manifold theory) on ring spectra.  One of the deepest computations of the algebraic $K$-theory of ring spectra to date is by Ausoni and Rognes \cite{ausoni-rognes-redshift}, who for primes $p \ge 5$ computed the mod $(p,v_1)$ $K$-theory of the $p$-completed Adams summand $\ell^{\wedge}_p$.  
Their computations imply that
\[\mathrm{K}(\ell^{\wedge}_p)_{(p)} \to L_2^{f} \mathrm{K}(\ell^{\wedge}_p)_{(p)}\]
is a $\pi_*$-isomorphism for $* \gg 0$.  Here $L_2^f$
is the next localization in a hierarchy of chromatic localizations
$L_n^f$ for each $n\ge 0$ (at an implicit prime $p$).
This of course suggests a higher height analog of the Lichtenbaum--Quillen conjecture.  In the Oberwolfach lecture \cite{rognes-oberwolfach}, Rognes laid out a far-reaching vision of how this higher height analog might go, which is now known as the \emph{chromatic redshift} philosophy.  The name \emph{redshift} refers to the hypothesis that algebraic $K$-theory should raise the chromatic height of ring spectra by exactly $1$.

To give a more
precise statement, we will need the notion of fp-type, due
to Mahowald--Rezk \cite{mahowald-rezk}:
A $p$-complete, bounded below
spectrum $X$ is of fp-type $n$ if the thick subcategory
of $p$-local 
finite complexes $F$ such that $|\pi_*(F\otimes X)|<\infty$
is generated by a type $(n+1)$ complex (i.e., a complex
with a $v_{n+1}$ self-map).

With this definition, Ausoni-Rognes conjecture that:

\begin{conj} For suitable $\mathbb{E}_1$-rings $R$ of
fp-type $n$, $\mathrm{K}(R)^{\wedge}_p$ is of fp-type $n+1$.
\end{conj}

As we review below (see Theorem \ref{fp-implies-LQ}),
this statement also implies that $\mathrm{K}(R) \to L_{n+1}^f\mathrm{K}(R)$
is a $p$-local equivalence in large degrees, so we can think of
it as a higher height analog of the Lichtenbaum-Quillen conjecture.

In the years since the Ausoni--Rognes computations, redshift has been verified for additional height $1$ ring spectra, including $\mathrm{ku}^{\wedge}_{p}$, $\mathrm{KU}^{\wedge}_{p}$, and $\mathrm{ku}/p$ at primes $p \ge 5$ \cite{blumberg-mandell, ausoni-topk, ausoni-rognes-height1}, and evidence for redshift has accumulated in general \cite{2-vect, rognesmsri, westerland, veen, angeliniknollbeta, angeliniknoll-quigley,carmelischlankyanovski}.  Recent conceptual advances show that the algebraic $K$-theories of many height $n$ rings are of height \emph{at most} $n+1$ \cite{lmmt,cmnn}.  Here, we give the first arbitrary height examples of ring spectra for which redshift provably occurs.

\subsection*{Main Results}

The truncated Brown--Peterson spectra, $\mathrm{BP} \langle n \rangle$,  are among the simplest and most important cohomology theories in algebraic topology.
There is one such spectrum for every prime $p$ and height $n \ge 0$,
 though we will follow tradition by localizing at the prime and omitting it from notation.\footnote{At each prime $p$ and height $n \ge 0$, $\mathrm{BP} \langle n \rangle$ is conjectured to be unique as a $p$-local spectrum.  For $n>1$, uniqueness is only proved up to $p$-completion, by work of Angeltveit and Lind  \cite{angeltveit-lind}.}
The height $1$ spectrum $\mathrm{BP} \langle 1 \rangle$ is the Adams summand $\ell$, while $\mathrm{BP} \langle  2 \rangle$ is a summand of either topological modular forms (at $p \ge 5$), or topological modular
forms with level structure (at $p=2,3$).

%The redshift philosophy of Ausoni and Rognes, first laid out in the Oberwolfach lecture \cite{rognes-oberwolfach}, predicts that the algebraic $K$-theory of a chromatic height $n$ ring spectrum is of chromatic height exactly $n+1$.
%Motivated by the Quillen--Lichtenbaum conjecture \cite{voeI,voeII}, this was verified for low values of $n$ \cite{suslin, ausoni-rognes-redshift, blumberg-mandell, ausoni-rognes-height1}, and evidence for redshift has accumulated in general \cite{rognesmsri, westerland, veen, angeliniknollbeta, carmelischlankyanovski}.  Recent conceptual advances show that the algebraic $K$-theories of many height $n$ rings are of height \emph{at most} $n+1$ \cite{lmmt,cmnn}.  Here, we give the first arbitrary height examples of ring spectra for which redshift provably occurs, along with an analogue
%of the Lichtenbaum--Quillen conjectures.

Both $\ell$ and $\mathrm{tmf}$ are extraordinarily structured: they are $\mathbb{E}_\infty$-ring spectra, inducing power operations on the cohomology of spaces. Our first main result, proven in \S\ref{sec:mult},
is a construction of part of this structure at an arbitrary height $n$. To make sense of the statement, we remind the
reader that $\mathrm{BP}$ admits the structure of
an $\mathbb{E}_4$-ring by \cite{bm}.

\begin{thm}[Multiplication]  \label{thm:intro-main}
For an appropriate choice of indecomposable generators 
\[v_{n+1},v_{n+2},\cdots \in \pi_* \mathrm{BP},\] 
the quotient map
\[\mathrm{BP} \to \mathrm{BP}/(v_{n+1},\cdots) = \mathrm{BP} \langle n \rangle\]
is the unit of an $\mathbb{E}_3$-$\mathrm{BP}$-algebra structure on $\mathrm{BP} \langle n \rangle$.
\end{thm}

Our second main theorem establishes the above conjecture
for $R=\mathrm{BP}\langle n\rangle^{\wedge}_p$.

\begin{thm}[Redshift]\label{thm:intro-fpshift} Let $\mathrm{BP}\langle n \rangle$ denote any $\mathbb{E}_3$-$\mathrm{BP}$-algebra such that the unit $\mathrm{BP} \to \mathrm{BP} \langle n \rangle$ is obtained by modding out a sequence of indecomposable generators $v_{n+1},v_{n+2}, \cdots$.
Then $\mathrm{K}(\mathrm{BP}\langle n\rangle^{\wedge}_p)_p^{\wedge}$
is of $fp$-type $n+1$.
\end{thm}

\begin{cor} For $\mathrm{BP}\langle n\rangle$ any
$\mathbb{E}_3$-$\mathrm{BP}$-algebra as above, both maps
	\begin{align*}
	\mathrm{K}(\mathrm{BP}\langle n\rangle^{\wedge}_p)_{(p)}
	&\to L^f_{n+1} 
	\mathrm{K}(\mathrm{BP}\langle n\rangle^{\wedge}_p)_{(p)},\\
	\mathrm{K}(\mathrm{BP}\langle n\rangle)_{(p)}
	&\to L^f_{n+1}\mathrm{K}(\mathrm{BP}\langle n\rangle)_{(p)}
	\end{align*}
induce isomorphisms on $\pi_*$ for $*\gg 0$.
\end{cor}

To prove Theorem \ref{thm:intro-fpshift} by trace methods, the critical thing to show is that
$\pi_*(V\otimes \mathrm{TC}(\mathrm{BP}\langle n\rangle))$ is
bounded above for some
type $(n+2)$ complex $V$. We recall \cite{nikolaus-scholze}
that the $p$-completed topological cyclic
homology of $\mathrm{BP}\langle n\rangle$
can be computed as the fiber:
	\[
	\mathrm{TC}(\mathrm{BP}\langle n\rangle)
	\simeq
	\mathrm{fib}\left(
	\left(\mathrm{THH}(
	\mathrm{BP}\langle n\rangle)^{hS^1}\right)^{\wedge}_p
	\stackrel{\varphi^{hS^1} - \mathrm{can}}{\xrightarrow{\hspace*{2.5cm}}}
	\left(\mathrm{THH}(\mathrm{BP}\langle n\rangle)^{tS^1}\right)^{\wedge}_p
	\right),
	\]
where the map
	\[
	\varphi:
	\mathrm{THH}(\mathrm{BP}\langle n\rangle)
	\to
	\mathrm{THH}(\mathrm{BP}\langle n\rangle)^{tC_p}
	\]
is the cyclotomic Frobenius. (See \S\ref{sec:cyclotomic-conventions}
for our conventions on cyclotomic spectra.)

One would like to argue that $(\varphi^{hS^1}-\mathrm{can})$
is an equivalence in large degrees after tensoring with
a type $(n+2)$ complex. We will deduce this from
the following two theorems:

\begin{thm}[Segal Conjecture]\label{thm:intro-segal} Let $F$ be any type $n+1$ complex.
Then the cyclotomic Frobenius $\mathrm{THH}(\mathrm{BP}\langle n \rangle) \to \mathrm{THH}(\mathrm{BP}\langle n \rangle)^{tC_p}$ induces an isomorphism
\[F_*\mathrm{THH}(\mathrm{BP}\langle n \rangle) \cong F_*(\mathrm{THH}(\mathrm{BP}\langle n \rangle)^{tC_p})\]
in all sufficiently large degrees $* \gg 0$.
\end{thm}

\begin{thm}[Canonical Vanishing] \label{thm:intro-can-van}
Let $F$ be any type
$n+2$ complex. There exists an integer $d\ge 0$
(depending on $F$) such that, for all $0\le k\le \infty$,
the composite
	\[
	\tau_{\ge d}(F 
	\otimes \mathrm{THH}(\mathrm{BP}\langle n\rangle)^{hC_{p^k}})
	\to
	F 
	\otimes \mathrm{THH}(\mathrm{BP}\langle n\rangle)^{hC_{p^k}}
	\stackrel{\mathrm{can}}{\longrightarrow}
	F 
	\otimes \mathrm{THH}(\mathrm{BP}\langle n\rangle)^{tC_{p^k}}
	\]
is nullhomotopic.
\end{thm}

We note that the first theorem involves only
the cyclotomic Frobenius map, and the second theorem only
the canonical map. We use different techniques to 
analyze each one.

In order to prove Theorem \ref{thm:intro-segal}, we use a filtration
on $\mathrm{BP}\langle n\rangle$ to reduce the statement to
a graded version of the Segal conjecture for polynomial algebras over
$\mathbb{F}_p$, which we then prove directly. This is
done in \S\ref{sec:segal}.

In order to prove Theorem \ref{thm:intro-can-van}, we
first investigate the $S^1$-spectrum $\mathrm{THH}(\mathrm{BP}\langle n
\rangle/\mathrm{MU})$ of Hochschild homology
\emph{relative} to $\mathrm{MU}$. This spectrum is much simpler
to understand because of the following analog of
B\"okstedt's periodicity theorem:

\begin{thm}[Polynomial $\mathrm{THH}$]
%,Theorem \ref{thm:poly-thh}] 
The ring $\mathrm{THH}(\mathrm{BP}\langle n\rangle/\mathrm{MU})_*$
is polynomial over $\pi_*\mathrm{BP}\langle n \rangle$ on even-degree generators, one of which can
be chosen to be the double-suspension class $\sigma^2v_{n+1}$. (For the definition of double-suspension classes, see \ref{exm:double-suspension}.)
\end{thm}

We may take advantage of the circle action on
$\mathrm{THH}$ to shift the class $\sigma^2v_{n+1}$ down to a class detecting $v_{n+1}$. More precisely, we prove:

\begin{thm}[Detection]  \label{thm:intro-detection}
There is an isomorphism of $\mathbb{Z}_{(p)}[v_1, ..., v_n]$-algebras
	\[
	\pi_*(\mathrm{THH}(\mathrm{BP}\langle n\rangle/\mathrm{MU})^{hS^1})
 	\cong \left(\pi_*\mathrm{THH}(\mathrm{BP}\langle n \rangle 
	/\mathrm{MU})\right)\llbracket t \rrbracket, 
	\]
where $|t|=-2$.
This isomorphism can be chosen such that, under the unit map
\[\pi_*(\mathrm{MU}_{(p)}^{hS^1}) \to \pi_*(\mathrm{THH}(\mathrm{BP}\langle n\rangle/\mathrm{MU})^{hS^1}),\] 
the canonical complex orientation maps to $t$ and $v_{n+1}$ is sent to
$t(\sigma^2v_{n+1})$.
\end{thm}

Theorem \ref{thm:intro-detection} already implies the following
weak form of redshift:

\begin{cor} (Corollary \ref{cor:weak-redshift})
$L_{K(n+1)}\mathrm{K}(\mathrm{BP}\langle n\rangle)$ is nonzero.
\end{cor}

Finally, in order to prove Theorem \ref{thm:intro-can-van},
we must descend information along the 
$S^1$-equivariant map
	\[
	\mathrm{THH}(\mathrm{BP}\langle n\rangle)
	\to \mathrm{THH}(\mathrm{BP}\langle n\rangle/\mathrm{MU}).
	\]
In \S\ref{sec:canvan},
we study this descent spectral sequence after tensoring
with a type $n+1$ complex. Using this information
we are able to understand enough about
the homotopy and Tate fixed point spectral sequences
associated to $\mathrm{THH}(\mathrm{BP}\langle n\rangle)$
to prove a weak form of the Canonical Vanishing Theorem.
As explained to us by an anonymous referee, and proven
in \S\ref{sec:abstract-ql}, this weak form of canonical vanishing
together with the Segal conjecture is enough to prove
the strong form of canonical vanishing as well as the main theorem.
In fact, we establish
the following result, which is equivalent to the combination
of Theorems \ref{thm:intro-segal} and \ref{thm:intro-can-van}
and also directly implies Theorem \ref{thm:intro-fpshift}.

\begin{thm}[Bounded $\mathrm{TR}$] \label{thm:intro-bounded-tr}
For any type $n+2$ complex $F$,
the spectrum $F \otimes \mathrm{TR}(\mathrm{BP}\langle n\rangle)$
is bounded.
\end{thm}

For a review of the functor $\mathrm{TR}$,
see \S\ref{sec:cyclotomic-conventions}.

\subsection*{Remarks on the Multiplication Theorem}

As pointed out by Morava, $\mathrm{BP} \langle n \rangle$ may be equipped with different homotopy ring structures \cite{moravaforms}.
Our redshift arguments only apply to forms of $\mathrm{BP} \langle n \rangle$ that are $\mathbb{E}_3$-$\mathrm{BP}$-algebras, which we guarantee to exist by Theorem \ref{thm:intro-main}. 
To check whether previously studied forms of $\mathrm{BP} \langle n \rangle$ admit such structure, there is a convenient criterion due to Basterra and Mandell:

\begin{remark}
Suppose that a form of $\mathrm{BP} \langle n \rangle$ is equipped with an $\mathbb{E}_4$-algebra structure.  Then there are no obstructions to producing an $\mathbb{E}_4$-ring map from $\mathrm{BP}$ to $\mathrm{BP} \langle n \rangle$ \cite[Corollary 4.4 and Lemma 5.1]{bm}.  Any such $\mathbb{E}_4$-map is the unit of an $\mathbb{E}_3$-$\mathrm{BP}$-algebra structure, allowing us to apply Theorem \ref{thm:intro-fpshift}.
\end{remark}

\begin{example}
At $p=2$ connective complex $K$-theory is an $\mathbb{E}_\infty$ form of $\mathrm{BP} \langle 1 \rangle$, and it follows that 
$\K(\mathrm{ku}^{\wedge}_2)_2^{\wedge}$ is of fp-type $2$.  Even the non-vanishing of $L_{K(2)} \K(\mathrm{ku})$ was previously known only for $p \ge 5$ \cite{ausoni-rognes-redshift}.

Similarly, we can deduce at $p=2$ that 
$\K(\mathrm{tmf}_1(3)^{\wedge}_2)^{\wedge}_2$ is of fp-type $3$, since $\mathrm{tmf}_1(3)$ is the Lawson--Naumann $\mathbb{E}_\infty$ form of $\mathrm{BP} \langle 2 \rangle$ \cite{lawson-naumann}.  Applying algebraic $\K$-theory to the $\mathrm{E}_\infty$-ring map $\mathrm{tmf} \to \mathrm{tmf}_1(3)$, we conclude that $L_{K(3)} \K(\mathrm{tmf}) \ne 0$.  
\end{example}

\begin{remark}
Our methods may help to prove that the algebraic $\K$-theories of many other height $n$ rings are not $\K(n+1)$-acyclic, especially when combined with the descent and purity results of \cite{cmnn,lmmt}.  For example, at the prime $3$ these results imply that the non-vanishing of $L_{K(2)} \K(\mathrm{ku})$ is equivalent to the non-vanishing of $L_{K(2)} \K(\mathrm{ko})$ (cf. \cite{ausonidescent}), and the latter follows from the fact that $3$-localized $\mathrm{ko}$ is an $\mathbb{E}_\infty$ form of $\mathrm{BP} \langle 1 \rangle$.
\end{remark}

To give context to Theorem \ref{thm:intro-main}, the question of whether 
$\mathrm{BP} \langle n \rangle$ can be made $\mathbb{E}_\infty$ was 
once a major open problem in algebraic topology \cite{mayproblems}.  
In breakthrough work, Tyler Lawson \cite{LawsonBP} and Andrew 
Senger \cite{sengerBP} showed this to be impossible whenever $n \ge 4$.

While the nonexistence of structure is of great theoretical interest, it is the presence of structure that powers additional computations.  For example, in this work 
we use the $\mathbb{E}_3$-algebra structure guaranteed by Theorem \ref{thm:intro-main} in order to prove the
Polynomial $\mathrm{THH}$ Theorem 
(\ref{prop:e2-mu-envelope}), which is the key computational
input to many of the remaining results of the paper. 
Our proof of Theorem \ref{thm:intro-main} relies on a number of ideas that we have not discussed so far: 
 see \S\ref{subsec:outline} for an outline of the proof of Theorem \ref{thm:intro-main}.

\begin{remark}
Prior to our work, other authors had succeeded in equipping $\mathrm{BP} \langle n \rangle$ with additional structure.  Notably, Baker and Jeanneret produced $\mathbb{E}_1$-ring structures \cite{baker-jeanneret} (cf. \cite{lazarev, angel}), and Richter produced Robinson $(2p-1)$-stage structures on related Johnson--Wilson theories \cite{richterstage}. Lawson and Naumann equipped $\mathrm{BP}\langle 2 \rangle$ with $\mathbb{E}_\infty$-structure at the prime $2$ \cite{lawson-naumann}, and Hill and Lawson produced an $\mathbb{E}_\infty$ form of $\mathrm{BP}\langle 2 \rangle$ at $p=3$ \cite{hill-lawson}.
\end{remark}

\begin{remark}\label{rmk:basterra-mandell-idempotent}
Basterra and Mandell proved that $\mathrm{BP}$ admits a unique $\mathbb{E}_4$-algebra structure, a fact which is necessary to make sense of $\mathbb{E}_3$-$\mathrm{BP}$-algebras \cite{bm}.  They also show that $\mathrm{BP}$ is an $\mathbb{E}_4$-algebra retract of $\mathrm{MU}_{(p)}$, so a $p$-local $\mathbb{E}_3$-$\mathrm{MU}$-algebra inherits an $\mathbb{E}_3$-$\mathrm{BP}$-algebra structure.  Our proof of Theorem \ref{thm:intro-main} most naturally produces an $\mathbb{E}_3$-$\mathrm{MU}$-algebra structure on $\mathrm{BP}\langle n \rangle$.  In fact, if one also formulates Theorem 
\ref{thm:intro-fpshift} in terms of $\mathbb{E}_3$-$\mathrm{MU}$-algebra structures, then none of the statements or proofs in this paper rely on \cite{bm}.
\end{remark}

\begin{remark}
It is not surprising that $\mathbb{E}_3$-algebra structure on $\mathrm{BP} \langle n \rangle$ is useful in the proof of redshift.  As far back as $2000$, Ausoni and Rognes observed that redshift could be proved whenever $\mathrm{BP} \langle n \rangle$ is $\mathbb{E}_\infty$ and the Smith--Toda complex $V(n)$ exists as a homotopy ring spectrum \cite{rognes-oberwolfach}.  Unfortunately, both of these hypotheses are known to generically fail \cite{Nave, LawsonBP}.
\end{remark}

\subsection*{Open Questions}

Our work leaves open many natural questions, chief of which is to determine the homotopy type of $\mathrm{K}(\mathrm{BP} \langle n \rangle)$.  
Since we show this homotopy type to be closely related to its localization $L_{n+1}^{f} \mathrm{K}(\mathrm{BP} \langle n \rangle)$, one might hope to assemble an understanding via chromatic fracture squares (c.f. \cite{ausoni-rognes-rational}).
We would also like to highlight:

\begin{question}
For what ring spectra $R$, other than $R=\mathrm{BP} \langle n \rangle$, is it possible to prove a version of the Segal conjecture?
\end{question}

\begin{question}
For what ring spectra $R$, other than $R=\mathrm{BP} \langle n \rangle$, is it possible to prove a version of the Canonical Vanishing theorem?
\end{question}

While variants of the Segal conjecture have received much study (see Section \ref{sec:segal} for some history), the Canonical Vanishing result does not seem as widely analyzed.
It seems plausible that a ring $R$ might satisfy Canonical Vanishing but not the Segal conjecture, or vice versa.

\begin{question}
What ring spectra $R$, other than $R=\mathrm{BP} \langle n \rangle$, satisfy redshift, or various less precise forms of the Lichtenbaum--Quillen conjecture?
\end{question}

For an arbitrary $\mathrm{BP} \langle n \rangle$-algebra $R$ satisfying the Segal conjecture, Akhil Mathew has deduced (given our work here) various Lichtenbaum--Quillen statements.  He has graciously allowed us to reproduce his results at the end of \S\ref{sec:bound-tr-etc}.

\begin{remark}
Redshift for $\mathbb{E}_1$-rings that are far
from complex oriented remains mysterious. For some
intriguing results in this direction, see the work of
Angelini-Knoll and Quigley \cite{angeliniknoll-quigley}
on the family of spectra $y(n)$.
\end{remark}

One would also like to make many of the above results \emph{effective}, rather than asserting an isomorphism in degrees above an unspecified dimension.  Especially the following question is interesting, since it does not depend on a choice of a finite complex:

\begin{question}
In precisely what range of degrees is the map
\[\mathrm{K}(\mathrm{BP}\langle n\rangle)_{(p)} \to L^{f}_{n+1} \mathrm{K}(\mathrm{BP}\langle n\rangle)_{(p)}\]
a $\pi_*$-isomorphism?
\end{question}

Theorem \ref{thm:intro-main} proves
that \emph{some} form of $\mathrm{BP} \langle n \rangle$ admits an $\mathbb{E}_3$-$\mathrm{BP}$-algebra structure, and it remains an interesting open question to determine exactly which forms admit such structure.

\begin{question} \label{qst:form}
Which forms of $\mathrm{BP}\langle n\rangle$
admit an $\mathbb{E}_3$-$\mathrm{BP}$-algebra structure?
Which of these can be built by the procedure in
\S\ref{sec:mult}?
\end{question}

The subtleties behind Question \ref{qst:form} are indicated by work of Strickland \cite[Remark 6.5]{strickland}, who observed at $p=2$ that neither the Hazewinkel or Araki generators may be used as generators in Theorem \ref{thm:intro-main}.

\begin{remark}
We suspect that our $\mathbb{E}_3$-algebra structure will be of use in additional computations.  For example, Ausoni and Richter give an elegant formula for the $\mathrm{THH}$ of a height $2$ Johnson--Wilson theory, under the assumption that the theory can be made $\mathbb{E}_3$ \cite{ausoni-richter}.  Our result does not directly feed into their work, for the simple reason that they use a form of $\mathrm{BP} \langle 2 \rangle[v_2^{-1}]$ specified by the Honda formal group.  It seems unlikely that their theorem relies essentially on this choice.
\end{remark}

\begin{remark}
By imitating our construction of an $\mathbb{E}_3$-$\mathrm{MU}$-algebra structure on $\mathrm{BP}\langle n \rangle$, we suspect one could produce an $\mathbb{E}_{2\sigma+1}$-$\mathrm{MU}_{\mathbb{R}}$-algebra structure on $\mathrm{BP} \langle n \rangle_{\mathbb{R}}$.  As a result, the fixed points $\mathrm{BP} \langle n \rangle_{\mathbb{R}}^{C_2}$ would acquire an $\mathbb{E}_1$-ring structure.  At the moment, these fixed points are not even known to be homotopy associative \cite{kitchloo-lorman-wilson}.
\end{remark}

%%%%look up citations and complete the below remarks

%\begin{remark} Versions of the ($\mathrm{THH}$-)Segal conjecture
%have been studied and proved in many other contexts. 
%The original Segal conjecture for the group $C_p$ is the statement that
%	\[
%	S^0 = \mathrm{THH}(S^0) \to \mathrm{THH}(S^0)^{tC_p} =
%	(S^0)^{tC_p}
%	\]
%is an equivalence mod $p$. Hesselholt proved that, when $R$
%is a regular Noetherian $\mathbb{F}_p$-algebra of dimension $d$,
%then the Frobenius $\varphi$ is an equivalence mod in degrees
%higher than $d$, which is a sort of
%height $-1$ version of the Segal conjecture.
%In the course of their computations of
%$\mathrm{TC}(\ell)$ at $p\ge 5$, Ausoni-Rognes prove that
%$\varphi$ is an equivalence mod $(p,v_1)$ in degrees larger than
%. Lun{\o}e-Nielsen-Rognes prove that $\varphi$ is an equivalence
%for $R=\mathrm{MU}$; in fact, our proof works in this case as well
%to give an alternative argument. The Segal conjecture for the spectra 
%$X(n)$ was proved by Angelini-Knoll-Quigley.
%\end{remark}
%
%\begin{remark} The Canonical Vanishing phenomenon does
%not seem to have been singled out in the literature, but we
%note that the results of 
%imply that the map $\mathrm{can}: \mathrm{THH}(R)^{hS^1}
%\to \mathrm{THH}(R)^{tS^1}$ vanishes mod $p$ 
%when $R$ is a regular $\mathbb{F}_p$-algebra. 
%\end{remark}

\subsubsection*{Acknowledgements}
We are extremely grateful to the anonymous referees for their
careful reading and many helpful comments.
The first referee's suggestion led to the much simpler
proof of the Multiplication Theorem now given in \S\ref{sec:mult}.
The second referee's
comments inspired us to separate the use of descent
and homotopy fixed point spectral sequences in our
argument for the Canonical Vanishing Theorem; we hope
that our present argument is easier to follow. The
third referee's comments included many of the results
in \S\ref{sec:abstract-ql} and led to simplifications of several proofs
in \S\ref{sec:sseq}.
The authors are also very grateful to
Christian Ausoni, Tyler Lawson, and John Rognes for
their comments on an earlier draft. We thank
Gabriel Angelini-Knoll,
Thomas Nikolaus, Oscar Randal-Williams, and Andrew Salch
for useful conversations related to the paper.
We are very grateful to Akhil Mathew for stimulating conversations
and for permission to include his results in \S\ref{sec:abstract-ql}.
We would also like to thank the participants and speakers
in Harvard's Thursday Seminar in Spring of 2021 for
comments, questions, and discussions that helped improve the paper.
The first author was supported by NSF grant DMS-$1803273$, and
the second author was supported by NSF grant DMS-$1902669$.

\subsubsection*{Conventions and notation}
\begin{itemize}
\item We work in the setting of $\infty$-categories 
as used in \cite{ha}. We
will say \emph{category}
and \emph{groupoid} instead of $\infty$-category and $\infty$-groupoid.
We will denote by $\mathrm{Map}_{\mathcal{C}}(X,Y)$ the mapping
space between objects $X, Y \in \mathcal{C}$.
\item We let $\mathsf{Sp}$ denote the category of spectra.
\item If $\mathcal{C}$ is enriched over a category $\mathcal{A}$
we denote by $\mathrm{map}_{\mathcal{C}}(X,Y)$ the morphism
object associated to a pair of objects $X, Y \in \mathcal{C}$.
For example, if $\mathcal{C}$ is stable then it is canonically enriched
over $\mathsf{Sp}$, and we use $\mathrm{map}_{\mathcal{C}}(X,Y)$
to denote the internal mapping spectrum. In cases below where
$\mathcal{C}$ is stable and canonically enriched over 
$\mathsf{Mod}_R$ for some $\mathbb{E}_{\infty}$-ring no
confusion should arise because the underlying spectrum
of the morphism object in $R$-modules will agree
with the morphism object in spectra.
\item We do not distinguish between a space and its corresponding
groupoid; in particular, we will speak about functors
$X \to \mathcal{C}$ where $X$ is a space and $\mathcal{C}$
is a category.
\item If $M$ is a (discrete) module over a (discrete) ring $R$
with elements $x,y \in M$, then we write $x\doteq y$ to mean
that $x=\lambda y$ where $\lambda$ is a unit in $R$.
\item If $\mathcal{C}$ is a category and
$G$ is a (possibly topological) group then the category
of \textbf{objects
of $\mathcal{C}$ with $G$-action} is the functor category
$\mathsf{Fun}(\mathrm{B}G, \mathcal{C})$. 
When $\mathcal{C}=\mathsf{Sp}$ we will sometimes refer
to these objects as $G$-spectra.
The theory of `genuine $G$-spectra'
is not used in this paper, so there should be no confusion.
\item Our conventions on grading spectral sequences
associated to towers differs from the usual one, since
we prefer to begin our spectral sequences at the second
page.
See \S\ref{ssec:towers}.
\item If
$\mathcal{C}$ is a stable category
equipped with a $t$-structure, we say that
an object $X \in \mathcal{C}$ is \textbf{bounded above}
if $X = \tau_{\le d}X$ for some $d \in \mathbb{Z}$. We say that 
$X$ is \textbf{bounded below} if $X = \tau_{\ge d}X$
for some $d \in \mathbb{Z}$. We say that a map $f: X \to Y$
is \textbf{truncated} if the fiber of $f$ is bounded above.
\item If $A$ is an $\mathbb{E}_1$-$R$-algebra
where $R$ is an $\mathbb{E}_{\infty}$-ring, we denote
by $\mathrm{THH}(A/R)$ the $R$-module
$A \otimes_{A \otimes A^{op}} A$. 
\item Our conventions on cyclotomic spectra differ somewhat
from those in \cite{nikolaus-scholze} since we are only
interested in constructions with $p$-complete spectra.
See \S\ref{sec:cyclotomic-conventions} for a discussion.
\end{itemize}

\section{The Multiplication Theorem}\label{sec:mult}

We begin by giving a more precise formulation of
Theorem \ref{thm:intro-main}. Recall
that there is a canonical inclusion \cite{quillen-fgl},
	\[
	\mathrm{BP}_* \to (\mathrm{MU}_{(p)})_*,
	\]
classifying the $p$-typification of the universal formal group law.
We will write  $\{x_i\}_{i\ge 0}$ for a choice
of indecomposable polynomial generators of
$(\mathrm{MU}_{(p)})_*$, with $|x_i|=2i$, such that the classes
$\{x_{p^j-1}\}_{j\ge 1}$ form
polynomial generators for $\mathrm{BP}_*$ over $\mathbb{Z}_{(p)}$.
It will be convenient, at times, to change these generators,
so we do not fix such a choice.
We write
$v_j$ for the generator $x_{p^j-1}$. By convention we agree
that $v_0 = p$. The following definition may be compared with \cite[3.1 \& 3.2]{lawson-naumann}.
\begin{definition} Let $1 \le k \le \infty$ and $n\ge 0$.
Let $B$ be a $p$-local
$\mathbb{E}_k$-$\mathrm{MU}$-algebra.
We say that $B$ is an \textbf{$\mathbb{E}_k$-$\mathrm{MU}$-algebra
form of $\mathrm{BP}\langle n\rangle$} if the composite
	\[
	\mathbb{Z}_{(p)}[v_1, ..., v_n] \subseteq
	\mathrm{BP}_* \subseteq (\mathrm{MU}_{(p)})_*
	\to B_*
	\]
is an isomorphism. By convention, we consider $\mathbb{F}_p$
as the unique $\mathbb{E}_k$-$\mathrm{MU}$-algebra
form of $\mathrm{BP}\langle -1\rangle$.
\end{definition}

\begin{remark} The subring
	\[
	\mathbb{Z}_{(p)}[v_1, ..., v_n] \subseteq \mathrm{BP}_*
	\]
is equal to the subring generated by all elements of degree
at most $2p^n-2$, and hence is independent of our choice
of polynomial generators. It follows that the definition
of an $\mathbb{E}_k$-$\mathrm{MU}$-algebra form
of $\mathrm{BP}\langle n\rangle$ also does not depend on
this choice.
\end{remark}

\begin{example}
For any $k \ge 1$, there is a unique $\mathbb{E}_k$-$\mathrm{MU}$-algebra form of $\mathrm{BP} \langle 0 \rangle$, which is the $p$-local integers $\mathbb{Z}_{(p)}$.
The Adams summand $\ell$ of $ku_{(p)}$ can be equipped with an $\mathbb{E}_\infty$-$\mathrm{MU}$-algebra structure, which makes $\ell$ into an $\mathbb{E}_\infty$-$\mathrm{MU}$-algebra form of $\mathrm{BP}\langle 1 \rangle$.  
\end{example}

We will now relate the notion of a form of $\mathrm{BP}\langle n\rangle$
to the quotients in Theorem \ref{thm:intro-main}.

\begin{notation} Let $J \subseteq \mathbb{Z}_{\ge 0}$ be
an indexing set, and $\{z_j\}_{j \in J}$ a sequence of elements
in $\pi_*\mathrm{MU}_{(p)}$. 
Define
	\[
	\mathrm{MU}_{(p)}/(z_j : j \in J)
	:=
	\colim_m \bigotimes_{j \in J, j \le m} \mathrm{MU}_{(p)}/z_j,
	\]
where the tensor product is taken over $\mathrm{MU}_{(p)}$,
and $\mathrm{MU}_{(p)}/z_j$ is defined by the cofiber sequence
	\[
	\Sigma^{|z_j|}\mathrm{MU}_{(p)} \stackrel{z_j}{\to} \mathrm{MU}_{(p)}
	\to \mathrm{MU}_{(p)}/z_j.
	\]
\end{notation}

\begin{lemma} If $B$ is an $\mathbb{E}_1$-$\mathrm{MU}$-algebra
form of $\mathrm{BP}\langle n\rangle$, then there is a choice
of indecomposable generators $x_j \in \pi_{2j}\mathrm{MU}_{(p)}$,
$j\ge 1$,
and an extension of the unit map
$\iota: \mathrm{MU}_{(p)} \to B$ to an
equivalence of $\mathrm{MU}$-modules
	\[
	\left(\mathrm{MU}_{(p)}/(x_j: j \ne p^i-1, 1\le i \le n)\right) \simeq B.
	\]
\end{lemma}
\begin{proof} Let $x'_j$ be any choice of indecomposable
generators such that $(x'_{p-1}, x'_{p^2 - 1}, ..., x'_{p^n - 1})
= (v_1, ..., v_n)$. By definition, if $j \ne p^i-1$ for $1\le i \le n$,
then $\iota(x'_j) = f_j(v_1, ..., v_n)$ for some polynomial $f_j$
with coefficients in $\mathbb{Z}_{(p)}$. Define
	\[
	x_j := \begin{cases}
	x'_j & j = p^i -1, 1\le i \le n,\\
	x'_j - f_j(x_{p-1}, ..., x_{p^n-1}) & \text{else}.
	\end{cases}
	\]
Then $\{x_j\}$ gives a new set of indecomposable generators
for $\pi_*\mathrm{MU}_{(p)}$ with the property that $\iota(x_j) = 0$
when $j \ne p^i -1$ for some $1\le i \le n$. 
We may then construct maps
	\[
	\bigotimes_{j \le m, j \ne p^i-1, 1\le i \le n}
	\mathrm{MU}_{(p)}/x_j \to
	\bigotimes_{j \le m, j \ne p^i-1, 1\le i \le n}
	B
	\to B
	\]
where the second map is the multiplication on $B$.
Passing to the colimit gives the desired equivalence
	\[
	\mathrm{MU}_{(p)}/(x_j: j \ne p^i -1, 1\le i \le n) \to B.
	\]
\end{proof}

It follows from the above lemma and Remark \ref{rmk:basterra-mandell-idempotent} that Theorem \ref{thm:intro-main} is a consequence
of the following theorem, which will be the main result of this
section:

\begin{theorem}\label{thm:multn-form} For all $n\ge -1$,
there exists an $\mathbb{E}_3$-$\mathrm{MU}$-algebra
form of $\mathrm{BP}\langle n\rangle$.
\end{theorem}

\begin{remark}
There are only a few results from Section \ref{sec:mult} which will be needed later in the paper.  In addition to Theorem \ref{thm:multn-form}, the reader interested in redshift need only understand Proposition \ref{prop:e2-mu-envelope} and Theorem \ref{thm:poly-thh}.
\end{remark}

\subsection{Outline
of the Proof}\label{subsec:outline}

For ease of exposition, we will not take care in this outline to distinguish between different forms of $\mathrm{BP}\langle n\rangle$. 
Our proof of Theorem \ref{thm:multn-form} proceeds by induction on $n$: assuming that $\mathrm{BP}\langle n \rangle$ is an $\mathbb{E}_3$-$\mathrm{MU}$-algebra, we will construct $\mathrm{BP}\langle n+1 \rangle$ as an $\mathbb{E}_3$-$\mathrm{MU}$-algebra.

Consider the tower of $\mathrm{MU}$-modules:
	\begin{align*}
	\mathrm{BP}\langle n+1\rangle
	\to \cdots
	\to \mathrm{BP}\langle n+1\rangle/(v_{n+1}^{k})
	\to \cdots \to \mathrm{BP}\langle n\rangle.
	\end{align*}
By our inductive hypothesis, the base of the tower,
$\mathrm{BP}\langle n\rangle$, has been refined
to an $\mathbb{E}_3$-$\mathrm{MU}$-algebra.
One possible way to proceed would be to inductively
equip each $\mathrm{BP}\langle n+1\rangle/(v_{n+1}^k)$ with
an $\mathbb{E}_3$-$\mathrm{MU}$-algebra structure.
Unfortunately, this would involve understanding the
$\mathbb{E}_3$-$\mathrm{MU}$-algebra cotangent
complex of $\mathrm{BP}\langle n+1\rangle/(v_{n+1}^k)$,
which becomes increasingly difficult to control as $k$ grows. 

Instead, we will make a stronger inductive hypothesis.
As we review in \S\ref{ssec:prove-mult}, 
there are $\mathbb{E}_{\infty}$-$\mathrm{MU}$-algebras
$\mathrm{MU}[y]/(y^k)$
refining the truncated polynomial algebras
$\mathrm{MU}_*[y]/(y^k)$, where
$|y|=2p^{n+1}-2$. We will induct on $k$ to build
$\mathrm{BP}\langle n+1\rangle/(v_{n+1}^k)$ as
an $\mathbb{E}_3$-$\mathrm{MU}[y]/(y^k)$-algebra. 
Taking a limit, we produce $\mathrm{BP}\langle n+1\rangle$ as
an $\mathbb{E}_3$-$\mathrm{MU}[y]$-algebra, where
$y$ acts by $v_{n+1}$. 

In \S\ref{ssec:envelopes} and \S\ref{ssec:deform}
we review some background in deformation theory.
In \S\ref{ssec:base-case} and \S\ref{ssec:inductive-step}
we make the key non-formal computations of
enveloping algebras and cotangent complexes, which
ultimately rest on Steinberger's computation of the
action of Dyer-Lashof operations on
the dual Steenrod algebra and on
Kochman's computation of the action of Dyer-Lashof
operations on the homology of $\mathrm{BU}$.
Finally, in \S\ref{ssec:prove-mult}, we put the pieces together
and prove Theorem \ref{thm:multn-form}.

\begin{remark} Our original argument, appearing in
a preprint version of this paper, relied on the theory of
centers and some manipulations with Koszul duality.
We are extremely grateful to the first referee for explaining how
our two uses of Koszul duality `cancel each other out,' suggesting the more intuitive argument sketched above.
\end{remark}

\begin{remark}
Our argument constructs $\mathrm{BP} \langle n+1 \rangle$ as an $\mathbb{E}_3$-$\mathrm{MU}[y]$-algebra, where $y$ acts by $v_{n+1}$, but we remember $\mathrm{BP}\langle n +1 \rangle$ only as an $\mathbb{E}_3$-$\mathrm{MU}$-algebra when constructing $\mathrm{BP}\langle n+2 \rangle$.  One might wonder whether, with more care, it is possible to construct $\mathrm{BP}\langle n \rangle$ as an $\mathbb{E}_3$-$\mathrm{MU}[y_0,y_1,\cdots,y_n]$-algebra, where $y_i$ acts as $v_i$.  In fact, this is not possible, even when $n=1$.
If $\ell$ were an $\mathbb{E}_3$-$\mathrm{MU}[y_0]$-algebra,  tensoring over $\mathrm{MU}[y_0]$ with the augmentation $\mathrm{MU}[y_0] \to \mathrm{MU}$ would construct $\ell / p = k(1)$ as an $\mathbb{E}_3$-algebra.  However, any $\mathbb{E}_2$-algebra with $p=0$ in its homotopy groups must be an $\mathbb{F}_p$-module \cite[Theorem 4.18]{may-nilpotence}.
\end{remark}

\subsection{Background: Operadic Modules and Enveloping Algebras}\label{ssec:envelopes}

Fix an $\mathbb{E}_{n+1}$-algebra $k$
and let $\mathcal{C}=\mathsf{LMod}_k$.
If $A \in \mathsf{Alg}_{\mathbb{E}_n}(\mathcal{C})$
is an $\mathbb{E}_n$-algebra, then we can define an
$\mathbb{E}_n$-monoidal category,
$\mathsf{Mod}^{\mathbb{E}_n}_{A}(\mathcal{C})$,
of $\mathbb{E}_n$-$A$-modules
(\cite[3.3.3.9]{ha}). The relevance of this category in our case
is the equivalence of 
$\mathsf{Mod}_A^{\mathbb{E}_n}(\mathcal{C})$ 
(\cite[7.3.4.14]{ha}) with the tangent category
$\mathsf{Sp}(\mathsf{Alg}_{\mathbb{E}_n}(\mathcal{C})_{/A})$
controlling deformations of $A$
(see Recollection \ref{recall-cotangent}).

It follows from \cite[7.1.2.1]{ha} that we have an equivalence
	\[
	\mathsf{Mod}_A^{\mathbb{E}_n}(\mathcal{C}) \simeq
	\mathsf{LMod}_{\mathcal{U}^{(n)}(A)},
	\]
where $\mathcal{U}^{(n)}(A)$, the $\mathbb{E}_n$-$k$-enveloping algebra
of $A$, is the endomorphism algebra
spectrum of the free $\mathbb{E}_n$-$A$-module on $k$.

\begin{remark}\label{rmk:enveloping-algebra-monoidal}
It follows from \cite[4.8.5.11]{ha}
that the assignment $B \mapsto \mathcal{U}^{(n-j)}(B)$
is a lax $\mathbb{E}_j$-monoidal functor of $B$.
In particular, if $A$ is an
$\mathbb{E}_n$-algebra in $\mathcal{C}$,
then $\mathcal{U}^{(n-j)}(A)$ has a canonical
$\mathbb{E}_{j+1}$-algebra structure. 
\end{remark}

We will need the following standard fact:

\begin{proposition} There is a canonical equivalence
of algebras:
	\[
	\mathcal{U}^{(n)}(A) \simeq
	A \otimes_{\mathcal{U}^{(n-1)}(A)} A^{\mathrm{op}},
	\]
where $A^{\mathrm{op}}$ denotes $A$ regarded as
an $\mathbb{E}_1$-$\mathcal{U}^{(n-1)}(A)$-algebra
with its opposite multiplication.
\end{proposition}

\begin{proof} The enveloping algebra is obtained by
taking the endomorphism algebra of a free object.
So it suffices, by \cite[4.8.5.11, 4.8.5.16]{ha}, to provide an equivalence:
	\[
	\mathsf{Mod}_A^{\mathbb{E}_n}(\mathcal{C})
	\simeq \mathsf{LMod}_A(\mathcal{C}) \otimes_{
	\mathsf{Mod}^{\mathbb{E}_{n-1}}_A(\mathcal{C})} 
	\mathsf{RMod}_A(\mathcal{C}).
	\]
By \cite[4.8.4.6,4.3.2.7]{ha}, we may identify the right hand side as
a category of bimodules:
	\[
	\mathsf{LMod}_A(\mathcal{C}) \otimes_{
	\mathsf{Mod}^{\mathbb{E}_{n-1}}_A(\mathcal{C})} 
	\mathsf{RMod}_A(\mathcal{C})
	\simeq
	\mathsf{BMod}_A(\mathsf{Mod}_A^{\mathbb{E}_{n-1}}(\mathcal{C})).
	\]
The result now follows from
\cite[1.0.4]{harpaz-nuiten-prasma}
by taking tangent categories at $A$ of the equivalence
(\cite[5.1.2.2]{ha}):
	\[
	\mathsf{Alg}_{\mathbb{E}_n}(\mathcal{C})
	\simeq
	\mathsf{Alg}_{\mathbb{E}_1}(\mathsf{Alg}_{\mathbb{E}_{n-1}}(
	\mathcal{C})).
	\]
\end{proof}

\begin{remark} One can use this result and induction on $n$
to prove that there is an equivalence
	\[
	\mathcal{U}^{(n)}(A) \simeq \int_{\mathbb{R}^n - \{0\}} A
	\]
of the enveloping algebra with the factorization homology
of $A$ over $\mathbb{R}^n -\{0\}$.
\end{remark}

\subsection{Background: Deformation Theory}\label{ssec:deform}

In this section we will review the obstruction theory for
deforming an algebra over a square-zero extension.
Throughout, if $f: S \to S'$ is a map of
$\mathbb{E}_{\infty}$-rings, we will denote by
$f^*$ the induced functor $S'\otimes_S(-)$ and
by $f_*$ the restriction of scalars along $f$, to emphasize
the dependence on $f$.

In this section we will be using the cotangent
complex formalism as in \cite[\S 7.3]{ha}, which we
briefly review now.

\begin{recollection}\label{recall-cotangent}
If $\mathcal{C}$ is a presentable
category, then there is a cocartesian fibration
$T_{\mathcal{C}} \to \mathcal{C}$ whose fiber
over $A \in \mathcal{C}$ is given by the 
stabilization $\mathsf{Sp}(\mathcal{C}_{/A})$.  
If $M \in \mathsf{Sp}(\mathcal{C}_{/A})$ then
we will denote by $A\oplus M$ the image of $M$
under the functor
$\Omega^{\infty}:
\mathsf{Sp}(\mathcal{C}_{/A}) \to \mathcal{C}_{/A}$,
and we will refer to this object as the \textbf{trivial square-zero
extension} of $A$ by $M$. The \textbf{cotangent complex
of $A$}, denoted
$\mathbf{L}_A$, is the image of $\mathrm{id}_A$
under the adjoint
$\Sigma^{\infty}_+: \mathcal{C}_{/A} \to \mathsf{Sp}(\mathcal{C}_{/A})$,
so that
	\[
	\mathrm{Map}_{\mathsf{Sp}(\mathcal{C}_{/A})}(\mathbf{L}_A, M)
	\simeq
	\mathrm{Map}_{\mathcal{C}_{/A}}(A, A\oplus M).
	\]
Given $\eta: \mathbf{L}_A \to M$ we will refer to the adjoint map
$d_{\eta}: A \to A\oplus M$ as the \textbf{derivation classified
by $\eta$}. Given such a derivation, we may
form the pullback
	\[
	\xymatrix{
	A^{\eta} \ar[r]\ar[d] & A \ar[d]^{d_0}\\
	A \ar[r]_{d_{\eta}} & A \oplus M
	}
	\]
where $d_0$ is classified by the zero map $0: \mathbf{L}_A \to M$.
We refer to $A^{\eta}$ as the \textbf{square-zero extension
classified by $\eta$}. 

In the special case $\mathcal{C} = 
\mathsf{Alg}_{\mathbb{E}_m}(\mathcal{D})$, where
$\mathcal{D}$ is a stable and presentably $\mathbb{E}_m$-monoidal
category, there is a canonical
equivalence \cite[Theorem 7.3.4.13]{ha} 
$\mathsf{Sp}(\mathsf{Alg}_{\mathbb{E}_m}(\mathcal{D})_{/A})
\simeq \mathsf{Mod}^{\mathbb{E}_m}_A(\mathcal{D})$
with the category of $\mathbb{E}_m$-$A$-modules.
(When $m=\infty$ this is equivalent to the ordinary
category of $A$-modules.)
We denote $\mathbf{L}_A$ by $\mathbf{L}^{\mathbb{E}_m}_A$,
and further decorate it as $\mathbf{L}^{\mathbb{E}_m}_{A/R}$ in the
setting where $\mathcal{D} = \mathsf{Mod}_R$ for an 
$\mathbb{E}_{\infty}$-ring $R$.
\end{recollection}

We now turn to the problem of classifying algebras over
square-zero extensions.
Let $R$ be a connective $\mathbb{E}_{\infty}$-ring and
$I$ a connective $R$-module. Let $\eta: \mathbf{L}_R \to
\Sigma I$ be a map of $R$-modules
from the $\mathbb{E}_{\infty}$-cotangent
complex of $R$ to $\Sigma I$, and denote by $R^{\eta}$
the corresponding square-zero extension. By definition,
this sits in a pullback diagram
	\[
	\xymatrix{
	R^{\eta}\ar[r]\ar[d] & R\ar[d]^-{d_0}\\
	R\ar[r]_-d & R \oplus \Sigma I
	}
	\]
where $d$ is adjoint to $\eta$ and $d_0$ is the trivial
derivation.

\begin{recollection}\label{recall:modules-over-sqzero}
If $S$ is a connective $\mathbb{E}_{\infty}$-ring, we will
denote by $\mathsf{Mod}_S^{\mathrm{cn}}$ the
category of connective $S$-modules.
By \cite[Theorem 16.2.0.2]{sag}, the pullback diagram above
induces a symmetric monoidal equivalence
	\[
	\mathsf{Mod}^{\mathrm{cn}}_{R^{\eta}}
	\stackrel{\simeq}{\longrightarrow}
	\mathsf{Mod}^{\mathrm{cn}}_R \times_{
	\mathsf{Mod}^{\mathrm{cn}}_{R \oplus \Sigma I}}
	\mathsf{Mod}^{\mathrm{cn}}_R,
	\]
and hence an equivalence upon taking categories
of $\mathbb{E}_m$-algebras:
	\[
	\mathsf{Alg}_{\mathbb{E}_m}(
	\mathsf{Mod}^{\mathrm{cn}}_{R^{\eta}})
	\stackrel{\simeq}{\longrightarrow}
	\mathsf{Alg}_{\mathbb{E}_m}(\mathsf{Mod}^{\mathrm{cn}}_R)
	\times_{
	\mathsf{Alg}_{\mathbb{E}_m}(
	\mathsf{Mod}^{\mathrm{cn}}_{R \oplus \Sigma I})}
	\mathsf{Alg}_{\mathbb{E}_m}(\mathsf{Mod}^{\mathrm{cn}}_R).
	\]
Denoting an element in the target by $(A, B, \alpha:
d^*A \simeq d_0^*B)$, the inverse to this
equivalence is implemented by
the functor $(A,B,\alpha) \mapsto A \times_{d^*A}B$. 
\end{recollection}

\begin{lemma}\label{lem:alg-over-triv-sq-zero} Suppose $\eta =0$ classifies the trivial
derivation, so that $R^{\eta} = R \oplus I$ and $d=d_0$. Then
$\mathsf{Alg}_{\mathbb{E}_m}(\mathsf{Mod}^{\mathrm{cn}}_{R^{\eta}})$
is equivalent to the category of pairs 
$(A, \rho: \mathbf{L}^{\mathbb{E}_m}_{A/R} \to A\otimes_R\Sigma I)$
where $A$ is a connective $\mathbb{E}_m$-$R$-algebra
and $\rho$ is a map of $\mathbb{E}_m$-$A$-modules.
Under this equivalence, the $\mathbb{E}_m$-$R^{\eta}$-algebra
corresponding to $(A,\rho)$ has underlying $\mathbb{E}_m$-$R$-algebra
given by the square-zero extension classified by $\rho$. 
\end{lemma}
\begin{proof} By Recollection \ref{recall:modules-over-sqzero},
the category of connective $\mathbb{E}_m$-$R^{\eta}$-algebras
is equivalent to the category of triples $(A, B, \alpha: 
A\otimes_R(R\oplus \Sigma I)\simeq B\otimes_R(R\oplus \Sigma I))$,
where $A$ and $B$ are connective 
$\mathbb{E}_m$-$R$-algebras and
$\alpha$ is an equivalence. Since $d_0$ is a section of
the projection $p: R\oplus \Sigma I \to R$, we have
a canonical equivalence $p^*d_0^*A = A$. It follows
that this category of triples is equivalent to
the subcategory where $A=B$, so we are left with understanding
automorphisms of the $\mathbb{E}_m$-$(R\oplus \Sigma I)$-algebra
$A\otimes_R(R \oplus \Sigma I)=d_0^*A$. 
By adjunction,
this is equivalent to understanding
the space of $\mathbb{E}_m$-$R$-algebra
sections of the projection 
\[(d_0)_*d_0^*A=A\otimes_R(R\oplus\Sigma I)
\to (d_0)_*p_*p^*d_0^*A = A.\] 

If $A$ is an $\mathbb{E}_m$-$R$-algebra,
then $A\otimes_R(R \oplus \Sigma I)$ is canonically
a spectrum object in $\mathsf{Alg}_{\mathbb{E}_m}(
\mathsf{Mod}_R)_{/ A}$ with deloopings given by
$A\otimes_R(R \oplus \Sigma^{j+1} I)$. In other words,
$A\otimes_R(R\oplus \Sigma I)$ is a trivial square-zero extension
of $A$ by $A\otimes_R\Sigma I$. The result now follows
by the universal property of the cotangent complex. 
\end{proof}

\begin{construction} Returning to the case of
a general square-zero extension $R^{\eta}$, suppose
that $A$ is an $\mathbb{E}_m$-$R$-algebra. Then
$d^*A$ is an $\mathbb{E}_m$-$(R\oplus \Sigma I)$-algebra.
Moreover, since $d$ is a section of the projection map,
base changing $d^*A$ along $R\oplus \Sigma I \to R$ 
recovers $A$. By the previous lemma (applied
to the trivial square-zero extension $R\oplus \Sigma I$
rather than $R\oplus I$), the
$\mathbb{E}_m$-$(R\oplus\Sigma I)$-algebra $d^*A$
is determined by a pair $(A, o(A):
\mathbf{L}^{\mathbb{E}_m}_{A/R} \to A \otimes_R\Sigma^2 I)$.
We refer to $o(A)$ as the \textbf{obstruction class for deforming
$A$}. Though it is not indicated in the notation, this class
also depends on $\eta$. 
\end{construction}

\begin{definition} Let $R$ and $R^{\eta}$ be as above, and
let $A$ be a connective $\mathbb{E}_m$-$R$-algebra.
Define
the \textbf{category of lifts of $A$} by the pullback:
	\[
	\xymatrix{
	\mathsf{Lifts}(A) \ar[r]\ar[d] & \{A\}\ar[d]\\
	\mathsf{Alg}_{\mathbb{E}_m}(
	\mathsf{Mod}^{\mathrm{cn}}_{R^{\eta}})
	\ar[r] &  \mathsf{Alg}_{\mathbb{E}_m}(
	\mathsf{Mod}^{\mathrm{cn}}_{R})
	}
	\]
\end{definition}

\begin{proposition}\label{prop:obstrn-theory} 
The category $\mathsf{Lifts}(A)$ is a groupoid
equivalent to the space of nullhomotopies of
the $\mathbb{E}_m$-$A$-module map
$o(A): \mathbf{L}^{\mathbb{E}_m}_{A/R} \to A\otimes_R\Sigma^2I$.
In particular:
	\begin{enumerate}[{\rm(i)}]
	\item  There exists an $\mathbb{E}_m$-$R^{\eta}$-algebra
$\widetilde{A}$ such that $\widetilde{A}\otimes_{R^{\eta}}R \simeq
A$ if and only if $o(A)$ is nullhomotopic.
	\item If $o(A)$ is nullhomotopic then the space
	$\mathsf{Lifts}(A)$ is equivalent to
	\[\mathrm{Map}_{\mathsf{Mod}^{\mathbb{E}_m}_A(
	\mathsf{Mod}_R)}(\mathbf{L}^{\mathbb{E}_m}_{A/R},
	A\otimes_R\Sigma I).\]
	\end{enumerate}
\end{proposition}
\begin{proof} By Recollection \ref{recall:modules-over-sqzero}
and the definition of the category of lifts,
each square in the rectangle below
	\[
	\xymatrix{
	\mathsf{Lifts}(A) \ar[r]\ar[d] & \{A\}\ar[d]\\
	\mathsf{Alg}_{\mathbb{E}_m}(
	\mathsf{Mod}^{\mathrm{cn}}_{R^{\eta}})
	\ar[r]\ar[d] &  \mathsf{Alg}_{\mathbb{E}_m}(
	\mathsf{Mod}^{\mathrm{cn}}_{R})\ar[d]\\
	\mathsf{Alg}_{\mathbb{E}_m}(
	\mathsf{Mod}^{\mathrm{cn}}_{R}) \ar[r]&
	\mathsf{Alg}_{\mathbb{E}_m}(
	\mathsf{Mod}^{\mathrm{cn}}_{R\oplus\Sigma I})
	}
	\]
is a pullback. Thus $\mathsf{Lifts}(A)$ is equivalent
to the category of pairs $(B, d^*B \simeq d_0^*A)$,
where $B$ is a connective $\mathbb{E}_m$-$R$-algebra,
$d$ is the derivation classified by $\eta$, and $d_0$ is
the trivial derivation. Let $p: R\oplus \Sigma I\to R$ be the projection
so that $p\circ d = p \circ d_0 = \mathrm{id}_R$. 
By Lemma \ref{lem:alg-over-triv-sq-zero}, an equivalence
$d^*B \simeq d_0^*A$ corresponds to an equivalence
$B=p^*d^*B \simeq p^*d_0^*A =A$ together with 
a homotopy between the two resulting maps 
$\mathbf{L}^{\mathbb{E}_m}_{A/R} \to  A\otimes_R\Sigma I$. 
Again, we may restrict to the equivalent subcategory with
$A =B$, and we have arrived at an equivalence between
$\mathsf{Lifts}(A)$ and the category of homotopies between
$o(A)$, which yields $d^*A$, and the zero map,
which yields $d_0^*A$. This completes the proof.
\end{proof}

\subsection{Grounding the Induction}\label{ssec:base-case}

The purpose of this section is to compute the higher 
$\mathrm{MU}$-enveloping algebras of $\mathbb{F}_p$. This
will allow us to resolve extension problems when computing the
$\mathbb{E}_3$-$\mathrm{MU}$-cotangent complex 
of $\mathrm{BP}\langle n\rangle$ during the inductive step.

In the course of this computation we will make use of the
Kudo-Araki-Dyer-Lashof operations \cite{hinfty}, which are natural
maps of spectra (see, e.g., \cite{glasman-lawson}) for any 
$\mathbb{E}_{\infty}$-$\mathbb{F}_p$-algebra, $A$,
	\[
	\begin{cases}
	Q^i: A \to \Sigma^{-2i(p-1)}A & p>2\\
	Q^i: A \to \Sigma^{-i}A & p=2
	\end{cases}
	\]
We will also use the suspension operation $\sigma$ discussed
in \S\ref{sec:suspension}.

\begin{lemma} The $\mathbb{E}_1$-$\mathrm{MU}$-enveloping
algebra of $\mathbb{F}_p$ has homotopy given by
	\[
	\pi_*\mathcal{U}^{(1)}_{\mathrm{MU}}(\mathbb{F}_p) \simeq
	\Lambda(\sigma v_i : i\ge 0) \otimes_{\mathbb{F}_p}
	\Lambda(\sigma x_j: j\ne p^k-1).
	\]
When regarded as an $\mathbb{E}_{\infty}$-$\mathbb{F}_p$-algebra
via the map\footnote{We make this choice for definiteness, but it does not have a significant effect on our computations. Indeed, once $k\ge 2$, there is a canonical $\mathbb{E}_{\infty}$-$\mathbb{F}_p$-algebra structure on $\mathcal{U}^{(k)}_{\mathrm{MU}}(\mathbb{F}_p) = \int_{\mathbb{R}^k-\{0\}}\mathbb{F}_p$ up to equivalence, since $\mathbb{R}^k-\{0\}$ is connected.} $\mathrm{id}_{\mathbb{F}_p}\otimes 1: \mathbb{F}_p
\to \mathbb{F}_p \otimes_{\mathrm{MU}}\mathbb{F}_p$,
we have the identities:
	\begin{align*}
	Q^{p^i}\sigma v_i &\doteq \sigma v_{i+1}, & p>2
	\\
	Q^{2^{i+1}}\sigma v_i &\doteq \sigma v_{i+1} & p=2\\
	Q^{j+1}
	\sigma x_j &\doteq
	\sigma x_{jp+p-1}, \bmod \textup{ decomposables} &
	p>2\\
	Q^{2j+2}
	\sigma x_j &
	\doteq \sigma x_{jp+p-1}, \bmod \textup{ decomposables}
	&p=2
	\end{align*}
\end{lemma}
\begin{proof} The algebra structure follows from \cite[Proposition 3.6]{angel}. To compute the action of the operations we use the two 
$\mathbb{E}_{\infty}$-maps:
	\[
	\mathbb{F}_p \otimes \mathbb{F}_p \stackrel{f}{\to}
	\mathbb{F}_p\otimes_{\mathrm{MU}}\mathbb{F}_p
	\stackrel{g}{\to}
	\mathbb{F}_p\otimes_{\mathbb{F}_p\otimes \mathrm{MU}}\mathbb{F}_p
	\]
We have $f(\overline{\tau}_i)
\doteq \sigma v_i$ (independently of our choice
of $v_i$) and $g(\sigma x_j) = \sigma b_j$,
where $b_j$ is the Hurewicz image of $x_j$ (for $j\ne p^k-1$). 
The first identity now follows from
Steinberger's calculation \cite[III.2]{hinfty} that 
	\[
	Q^{p^i}\overline{\tau}_i = \overline{\tau}_{i+1},
	\]
(and the analogous result at $p=2$).
For the second identity, first recall that the Thom isomorphism
is an equivalence
	\[
	\mathbb{F}_p \otimes \mathrm{MU} \simeq
	\mathbb{F}_p \otimes \mathrm{BU}_+
	\]
of $\mathbb{E}_{\infty}$-$\mathbb{F}_p$-algebras, and hence we have
an equivalence
	\[
	\mathbb{F}_p\otimes_{\mathbb{F}_p\otimes \mathrm{MU}}\mathbb{F}_p
	\simeq
	\mathbb{F}_p\otimes_{\mathbb{F}_p\otimes \mathrm{BU}_+}\mathbb{F}_p
	\simeq \mathbb{F}_p \otimes \mathrm{B}^2\mathrm{U}_+.
	\]
of $\mathbb{E}_{\infty}$-$\mathbb{F}_p$-algebras. 

Since $\Sigma^{\infty}_+$ is symmetric monoidal,
the canonical map $\Sigma^{\infty}_+(\Omega X) \to
\Omega \Sigma^{\infty}_+X$ is a map of
nonunital $\mathbb{E}_{\infty}$-algebras for any
$\mathbb{E}_{\infty}$-space $X$. In particular, 
taking $X = \mathrm{B}^2U$, we see that 
the homology suspension 
$H_*(\mathrm{BU};\mathbb{F}_p) \to H_{*+1}(\mathrm{B}^2U;\mathbb{F}_p)$
preserves Dyer-Lashof operations.
The result now follows from 
Kochman's computation \cite[Theorem 6]{kochman} of the action
of Dyer-Lashof operations on $H_*(\mathrm{BU})$.
\end{proof}

The previous lemma implies that the K\"unneth spectral sequence for $\mathcal{U}^{(2)}_{\mathrm{MU}}(\mathbb{F}_p)$ collapses
at the $E^2$ term, which is a divided power algebra:
	\[
	E^2=E^{\infty} = \Gamma\{\sigma^2v_i, \sigma^2 x_j: i\ge 0, j\ne p^k-1\}
	\Rightarrow \pi_*(\mathbb{F}_p\otimes_{\mathcal{U}^{(1)}_{\mathrm{MU}}(\mathbb{F}_p)}\mathbb{F}_p)=\pi_*\mathcal{U}^{(2)}_{\mathrm{MU}}(\mathbb{F}_p).
	\]

Here we recall that, in the bar complex computing 
$\mathrm{Tor}^{\Lambda(z)}(\mathbb{F}_p, \mathbb{F}_p)$,
the class
$\gamma_{p^i}(\sigma z)$, where $i\ge 0$,
is represented by the element
	\[
	z \otimes z\otimes \cdots \otimes z \in \Lambda(z)^{\otimes p^i}.
	\]

\begin{proposition}\label{prop:e2-env-algebra-of-fp} 
There are nontrivial multiplicative extensions
in the bar spectral sequence for 
$\mathcal{U}^{(2)}_{\mathrm{MU}}(\mathbb{F}_p)$ as follows:
	\begin{enumerate}[{\rm (i)}]
	\item If $w_{0, i} \in 
	\pi_*\mathcal{U}^{(2)}_{\mathrm{MU}}(\mathbb{F}_p)$ is detected by 	the divided power
	$\gamma_{p^i}(\sigma^2v_0)$, then
	$w_{0,i}^{p^j}$ is detected by $\gamma_{p^i}(\sigma^2v_j)$,
	up to a unit.
	\item
	If $y_{j, i} \in 
	\pi_*\mathcal{U}^{(2)}_{\mathrm{MU}}(\mathbb{F}_p)$ is detected by 	the divided power
	$\gamma_{p^i}(\sigma^2x_j)$, then
	$y_{j,i}^{p}$ is detected by $\gamma_{p^i}(\sigma^2x_{jp+p-1})$,
	up to a unit.
	\end{enumerate}

In particular, the homotopy groups of $\mathcal{U}^{(2)}_{\mathrm{MU}}(\mathbb{F}_p)$ are polynomial on even dimensional classes,
one of which can be chosen to be $\sigma^2v_0$.
\end{proposition}
\begin{proof} 
This follows from the computation
of power operations in the previous lemma
by applying \cite[Theorem 3.6]{bm} 
to the standard representatives of divided powers in
the bar complex.
\end{proof}

\begin{remark}\label{rmk:boksedt-implies-polynomial} From the equivalence
$\mathcal{U}^{(2)}_{\mathrm{MU}}(\mathbb{F}_p)
\simeq \mathrm{THH}(\mathbb{F}_p/\mathrm{MU})$,
we deduce the crucial fact that the homotopy groups
of $\mathrm{THH}(\mathbb{F}_p/\mathrm{MU})$
are polynomial. An anonymous referee points out that
one may also deduce this from B\"okstedt's perioidicity
theorem, as we now sketch. 
We have $\mathrm{THH}(\mathbb{F}_p/\mathrm{MU})
\simeq \mathrm{THH}(\mathbb{F}_p) \otimes_{\mathrm{THH}(\mathrm{MU})}
\mathrm{MU}$. By B\"okstedt's theorem, $\pi_*\mathrm{THH}(\mathbb{F}_p)$
is polynomial, and one is reduced to proving that
$\mathbb{F}_p \otimes_{\mathrm{THH}(\mathrm{MU})}\mathrm{MU}$
has polynomial homotopy. By a Thom spectrum argument
\cite{blumberg-cohen-schlichtkrull}, this
spectrum is equivalent to $\mathbb{F}_p \otimes \mathrm{BSU}_+$,
which is known to have polynomial homotopy groups.
\end{remark}

\begin{proposition} The $\mathbb{E}_3$-$\mathrm{MU}$-enveloping
algebra of $\mathbb{F}_p$ has homotopy given by
an exterior algebra on odd dimensional generators,
one of which can be chosen to be $\sigma^3v_0$.
\end{proposition}
\begin{proof} Immediate from \cite[Proposition 3.6]{angel}.
\end{proof}

The spectral sequence
	\[
	\mathrm{Ext}_{\pi_*\mathcal{U}_{\mathrm{MU}}^{(3)}(\mathbb{F}_p)}(\mathbb{F}_p,\mathbb{F}_p)
	\Rightarrow \pi_*\mathrm{map}_{\mathcal{U}^{(3)}_{\mathrm{MU}}(
	\mathbb{F}_p)}(\mathbb{F}_p, \mathbb{F}_p)
	\]
then immediately collapses with no possible
$\mathbb{F}_p$-algebra extensions, and so proves:

\begin{theorem} The spectrum
$\mathrm{map}_{\mathcal{U}^{(3)}_{\mathrm{MU}}(
\mathbb{F}_p)}(\mathbb{F}_p, \mathbb{F}_p)$
has homotopy given by
a polynomial algebra on even degree generators.
\end{theorem}

\subsection{Computation at the inductive step}\label{ssec:inductive-step}

In this section, we will assume that we have
constructed an $\mathbb{E}_3$-$\mathrm{MU}$-algebra form 
of $\mathrm{BP}\langle n\rangle$, and simply denote it by
$\mathrm{BP}\langle n\rangle$. We will also choose our
polynomial generators for $\pi_*(\mathrm{MU}_{(p)})$ in such a way that
$\mathrm{ker}((\mathrm{MU}_*)_{(p)} \to \mathrm{BP}\langle n\rangle_*)$
is generated by the $v_i$ with $i\ge n+1$ and by the $x_j$ with
$j \ne p^k-1$. 

\begin{remark} If $R$ is a $p$-local 
$\mathbb{E}_k$-$\mathrm{MU}$-algebra,
then, with notation as in \S\ref{sec:suspension},
$\int_MR$, for any nonempty $k$-manifold $M$, can be computed
in $\mathrm{MU}_{(p)}$-modules instead of $\mathrm{MU}$-modules.
We may therefore make sense of the suspension operations
from \S\ref{sec:suspension} for elements
in $\pi_*(\mathrm{cofib}(\mathrm{MU}_{(p)} \to R))$,
rather than just elements of $\pi_*(\mathrm{cofib}(\mathrm{MU} \to R))$.
\end{remark}

\begin{lemma} The $\mathbb{E}_1$-$\mathrm{MU}$-enveloping
algebra of $\mathrm{BP}\langle n\rangle$ has homotopy given,
as a $\mathrm{BP}\langle n\rangle_*$-algebra, by:
	\[
	\pi_*\mathcal{U}^{(1)}_{\mathrm{MU}}(\mathrm{BP}\langle n\rangle)
	\simeq
	\Lambda_{\mathrm{BP}\langle n\rangle_*}(\sigma v_i: i\ge n+1)
	\otimes_{\mathrm{BP}\langle n\rangle_*}
	\Lambda_{\mathrm{BP}\langle n\rangle_*}(\sigma x_j : j\ne p^k-1).
	\]
\end{lemma}
\begin{proof} Immediate from \cite[Proposition 3.6]{angel}.
\end{proof}

It follows that the bar spectral sequence for
$\pi_*\mathcal{U}^{(2)}_{\mathrm{MU}}(\mathrm{BP}\langle n
	\rangle)$
collapses to
a divided power algebra on even classes:
	\[
	E^2 = E^{\infty}=
	\Gamma_{\mathrm{BP}\langle n\rangle_*}(\sigma^2v_i :
	i\ge n+1) \otimes_{\mathrm{BP}\langle n \rangle_*}
	\Gamma_{\mathrm{BP}\langle n\rangle_*}(\sigma^2x_j:
	j\ne p^k-1).
	\]

\begin{proposition}\label{prop:e2-mu-envelope} 
For $i\ge 0$, choose any lift $w_{n+1, i}$ of the class
$\gamma_{p^i}(\sigma^2v_{n+1})$. For 
$j \nequiv -1 \bmod p$ and $i\ge 0$, choose any lift 
$y_{j,i}$ of $\gamma_{p^i}(\sigma^2x_j)$. 
The $\mathbb{E}_2$-$\mathrm{MU}$-enveloping
algebra of $\mathrm{BP}\langle n\rangle$ has homotopy given,
as a $\mathrm{BP}\langle n\rangle_*$-algebra, by
	\[
	\pi_*\mathcal{U}^{(2)}_{\mathrm{MU}}(\mathrm{BP}\langle n
	\rangle) \simeq
	\mathrm{BP}\langle n\rangle_*[w_{n+1, i} : i\ge 0]
	\otimes_{\mathrm{BP}\langle n\rangle_*}
	\mathrm{BP}\langle n\rangle_*[y_{j,i}: j \nequiv -1 
	\bmod p, i\ge 0, j\ge 1].
	\]
Moreover, we may choose $w_{n+1, 0} = \sigma^2v_{n+1}$. 
\end{proposition}
\begin{proof} 
Since $\mathrm{BP}\langle n\rangle$ is an 
$\mathbb{E}_3$-$\mathrm{MU}$-algebra, the enveloping algebra
$\mathcal{U}^{(2)}_{\mathrm{MU}}(\mathrm{BP}\langle n \rangle)$
is an $\mathbb{E}_2$-algebra
(see Remark \ref{rmk:enveloping-algebra-monoidal})
and, in particular, its homotopy
groups are a graded commutative algebra. 
Thus, our choice of elements $w_{n+1,i}$ and $y_{j,i}$ extend
to a map 
	\[
	f: \mathrm{BP}\langle n\rangle_*[w_{n+1, i} : i\ge 0]
	\otimes_{\mathrm{BP}\langle n\rangle_*}
	\mathrm{BP}\langle n\rangle_*[y_{j,i}: j \nequiv -1 
	\bmod p, i\ge 0, j\ge 1]
	\to \pi_*\mathcal{U}^{(2)}_{\mathrm{MU}}(\mathrm{BP}\langle n\rangle)
	\]
which we would like to be an isomorphism.
From the bar spectral sequence we already know that
$\pi_*\mathcal{U}^{(2)}_{\mathrm{MU}}(\mathrm{BP}\langle n\rangle)$
is a connective, free $\mathrm{BP}\langle n\rangle_*$-module with
finitely many generators in each degree. It
suffices from this and a dimension count
to prove that $f$ is injective
modulo $(p, v_1, ..., v_n)$.

But now observe that
the map
	\[
	\pi_*\mathcal{U}^{(2)}_{\mathrm{MU}}(\mathrm{BP}\langle n\rangle)/(p, v_1, ..., v_n)
	\longrightarrow 
	\pi_*\mathcal{U}^{(2)}_{\mathrm{MU}}(\mathbb{F}_p)
	\]
is injective by our previous calculation of the target and naturality
of the bar spectral sequence, since it is so on the $E^{\infty}$-page
of the bar spectral sequence. The result now follows
by Proposition \ref{prop:e2-env-algebra-of-fp}.
\end{proof}

Since 
$\mathcal{U}^{(2)}_{\mathrm{MU}}(
\mathrm{BP}\langle n\rangle)$ coincides with
$\mathrm{THH}(\mathrm{BP}\langle n\rangle/\mathrm{MU})$
as an $\mathbb{E}_1$-algebra, this is also the computation
of Hochschild homology given in the introduction:

\begin{theorem}[Polynomial THH]\label{thm:poly-thh}
There is an isomorphism of 
$\mathrm{BP}\langle n\rangle_*$-algebras
	\[
	\mathrm{THH}(\mathrm{BP}\langle n\rangle/\mathrm{MU})_*
	\simeq
	\mathrm{BP}\langle n\rangle_*[w_{n+1, i} : i\ge 0]
	\otimes_{\mathrm{BP}\langle n\rangle_*}
	\mathrm{BP}\langle n\rangle_*[y_{j,i}: j \nequiv -1 
	\bmod p, i\ge 0, j\ge 1].
	\]
Moreover, we may take $w_{n+1,0} = \sigma^2v_{n+1}$.
\end{theorem}

%\begin{warning} We have not
%ruled out nontrivial extensions as an $\mathrm{MU}_*$-module.
%\end{warning}

Again, it follows from
\cite[Proposition 3.6]{angel} that the $\mathbb{E}_3$-$\mathrm{MU}$-enveloping 
algebra has
homotopy given by an exterior algebra, and hence that the
spectral sequence
	\[
	\mathrm{Ext}_{\pi_*\mathcal{U}^{(3)}_{\mathrm{MU}}(\mathrm{BP}\langle
	n\rangle)}(\mathrm{BP}\langle n\rangle_*, \mathrm{BP}\langle n\rangle_*)
	\Rightarrow
	\pi_*\mathrm{map}_{\mathcal{U}^{(3)}_{\mathrm{MU}}(\mathrm{BP}\langle
	n\rangle)}
	(\mathrm{BP}\langle n\rangle, \mathrm{BP}\langle n\rangle)
	\]
collapses at the $E_2$ page. This proves:

\begin{theorem}\label{thm:e3-mu-center-bpn} The
spectrum
$\mathrm{map}_{\mathcal{U}^{(3)}_{\mathrm{MU}}(\mathrm{BP}\langle
	n\rangle)}
	(\mathrm{BP}\langle n\rangle, \mathrm{BP}\langle n\rangle)$
has homotopy groups
isomorphic to a polynomial algebra over $\mathbb{Z}_{(p)}[v_1, ..., v_n]$
on even degree generators. In particular, the homotopy groups
are concentrated in even degrees.
\end{theorem}

For the purposes of our obstruction theory argument, we will
require the following closely related statement:

\begin{proposition}\label{prop:bpn-obstrn-zero} Let $\mathbf{L}^{\mathbb{E}_3}_{\mathrm{BP}\langle n\rangle/\mathrm{MU}}$ denote the $\mathbb{E}_3$-$\mathrm{MU}$-algebra cotangent complex of $\mathrm{BP}\langle n\rangle$. Let $I$ denote the fiber
$\mathrm{fib}(\mathrm{MU_{(p)} \to \mathrm{BP}\langle n\rangle)}$.
	\begin{enumerate}[{\rm (i)}]
	\item The groups
	$\pi_{-2k}\mathrm{map}_{\mathcal{U}^{(3)}_{\mathrm{MU}}(
		\mathrm{BP}\langle n\rangle)}
		(\mathbf{L}^{\mathbb{E}_3}_{\mathrm{BP}\langle n
		\rangle/\mathrm{MU}},
		\mathrm{BP}\langle n\rangle)$
	vanish for $k\ge 0$. 
	\item Let $\delta v_{n+1} \in
	\pi_*\mathrm{map}_{\mathrm{MU}}(\mathrm{BP}\langle n\rangle,
	\mathrm{BP}\langle n\rangle)$
	denote the $\mathrm{BP}\langle n\rangle_*$-linear dual
	of the element $\sigma v_{n+1}$ with respect to the standard
	monomial basis of
		\[
		\pi_*(\mathrm{BP}\langle n\rangle
		\otimes_{\mathrm{MU}} \mathrm{BP}\langle 
		n\rangle) \simeq
		\Lambda(\sigma v_i: i\ge n+1) \otimes_{\mathbb{F}_p}
		\Lambda(\sigma x_j : j \ne p^k-1).
		\]
	Identifying $\pi_*\mathrm{map}_{
	\mathrm{MU}}(\Sigma I,
	\mathrm{BP}\langle n\rangle)$ with the 
	$\mathrm{BP}\langle n\rangle_*$-module summand
	of \\ $\pi_*\mathrm{map}_{\mathrm{MU}}(\mathrm{BP}\langle n\rangle,
	\mathrm{BP}\langle n\rangle)$ complementary to
	the unit, the class $\delta v_{n+1}$ lies in the image of the
	forgetful map
		\[
		\pi_*\mathrm{map}_{\mathcal{U}^{(3)}_{\mathrm{MU}}(
		\mathrm{BP}\langle n\rangle)}
		(\mathbf{L}^{\mathbb{E}_3}_{\mathrm{BP}\langle n
		\rangle/\mathrm{MU}}, \mathrm{BP}\langle n\rangle)
		\longrightarrow
		\pi_*\mathrm{map}_{
		\mathrm{MU}}(\Sigma I,
		\mathrm{BP}\langle n\rangle).
		\]
	\end{enumerate}
\end{proposition}
\begin{proof} By \cite[Theorem 7.3.5.1]{ha}, we have
a cofiber sequence of 
$\mathcal{U}^{(3)}_{\mathrm{MU}}(\mathrm{BP}\langle n\rangle)$-modules
	\[
	\mathcal{U}^{(3)}_{\mathrm{MU}}(\mathrm{BP}\langle n\rangle)
	\to
	\mathrm{BP}\langle n\rangle 
	\to
	\Sigma^3\mathbf{L}^{\mathbb{E}_3}_{\mathrm{BP}\langle
	n\rangle/\mathrm{MU}}.
	\]
Claim (i) then follows by applying the functor
$\mathrm{map}_{\mathcal{U}^{(3)}_{\mathrm{MU}}(\mathrm{BP}\langle
n\rangle)}(-, \mathrm{BP}\langle n\rangle)$ and the previous theorem.
The same reasoning also shows that the spectral sequence
	\[
	\mathrm{Ext}_{
	\pi_*\mathcal{U}^{(3)}_{\mathrm{MU}}(\mathrm{BP}\langle n\rangle)}
	(\pi_*\mathbf{L}^{\mathbb{E}_3}_{\mathrm{BP}\langle n\rangle/
	\mathrm{MU}}, \pi_*\mathrm{BP}\langle n\rangle)
	\Rightarrow
	\pi_*\mathrm{map}_{\mathcal{U}^{(3)}_{\mathrm{MU}}(
	\mathrm{BP}\langle n\rangle)}
	(\mathbf{L}^{\mathbb{E}_3}_{\mathrm{BP}\langle n
	\rangle/\mathrm{MU}}, \mathrm{BP}\langle n\rangle)
	\]
collapses at the $E_2$-page. It will therefore suffice to show
that $\delta v_{n+1}$ lies in the image of the following map
(which arises from the forgetful functor from $\mathbb{E}_3$-algebras
to $\mathbb{E}_0$-algebras):
	\[
	\mathrm{Hom}_{
	\pi_*\mathcal{U}^{(3)}_{\mathrm{MU}}(\mathrm{BP}\langle n\rangle)}
	(\pi_*\mathbf{L}^{\mathbb{E}_3}_{\mathrm{BP}\langle n\rangle/
	\mathrm{MU}}, \pi_*\mathrm{BP}\langle n\rangle)
	\to
	\mathrm{Hom}_{\mathrm{MU}_*}
	(\pi_*\Sigma I,
	\pi_*\mathrm{BP}\langle n\rangle).
	\]
So we must study the $\mathrm{MU}_*$-module map
	\[
	\pi_*\Sigma I \to
	\pi_*\mathbf{L}^{\mathbb{E}_3}_{\mathrm{BP}\langle n\rangle/
	\mathrm{MU}}
	\]
The homotopy groups of the source are given by the suspension of
$I_*=\mathrm{ker}(\mathrm{MU}_* \to 
\mathrm{BP}\langle n\rangle_*)$ while the homotopy groups
of the target are given by $\Sigma^4J_*$ where
$J_* = \mathrm{ker}(\mathcal{U}^{(3)}_{\mathrm{MU}}(
\mathrm{BP}\langle n\rangle)_* \to \mathrm{BP}\langle n\rangle_*)$.
Under these identifications, Lemma \ref{lem:suspn-and-cotangent}
implies that $x_j \mapsto \sigma^3x_j$ for those $x_j$ lying in
$I_*$. By our computation of $\mathcal{U}^{(3)}_{\mathrm{MU}}(
\mathrm{BP}\langle n\rangle)_*$, the subset of \emph{nonzero} classes
of the form $\{\sigma^3x_j\}$ spans an
$\mathrm{MU}_*$-module summand of
$\mathcal{U}^{(3)}_{\mathrm{MU}}(
\mathrm{BP}\langle n\rangle)_*$. Among these is the nonzero
class $\sigma^3v_{n+1}$, and the result follows.
\end{proof}

\subsection{Proof of Theorem \ref{thm:intro-main}} \label{ssec:prove-mult}

We will now prove Theorem \ref{thm:multn-form}, and hence
Theorem \ref{thm:intro-main}. We will deduce the theorem
as a consequence of a more precise assertion. In order to
state it we will need to recall a construction.

\begin{construction} Recall that we have an $\mathbb{E}_{\infty}$-map
of spaces
	\[
	J_{\mathbb{C}}: \mathrm{BU} \times \mathbb{Z} \to
	\mathrm{Pic}(\mathsf{Sp})
	\]
where the target denotes the Picard space of the category 
of spectra. Left Kan extension then yields a symmetric monoidal
functor
	\[
	\xymatrix{
	\mathrm{BU} \times \mathbb{Z} \ar[r] \ar[d]&
	\mathrm{Pic}(\mathsf{Sp}) \ar[r] & \mathsf{Sp}\\
	\mathbb{Z} \ar@{-->}[urr]_{\mathrm{MU}[z^{\pm 1}]}
	}
	\]
We interpret $\mathrm{MU}[z^{\pm 1}]$ as a graded 
$\mathbb{E}_{\infty}$-ring. Here the notation is justified by the fact
that the homotopy groups of the underlying spectrum
(i.e. the direct sum of the graded components)
are given by $\mathrm{MU}_*[z^{\pm 1}]$ where
$z$ is a class in degree $2$.
For any $j \in \mathbb{Z}$
we may then construct a (nonnegatively)
graded $\mathbb{E}_{\infty}$-ring
	\[
	\mathrm{MU}[y]: \mathbb{Z}_{\ge 0} \to 
	\mathbb{Z} \stackrel{\cdot j}{\to}
	\mathbb{Z} \stackrel{\mathrm{MU}[z^{\pm 1}]}{\to}
	\mathsf{Sp}.
	\]
Here, the homotopy groups of the underlying spectrum
of $\mathrm{MU}[y]$ are given by $\mathrm{MU}_*[y]$ where
$|y| = 2j$. By Lemma \ref{lem:truncating}, we may write
$\mathrm{MU}[y]$ as the limit of a tower of 
$\mathbb{E}_{\infty}$-$\mathrm{MU}$-algebra square-zero
extensions
	\[
	\mathrm{MU}[y] \to \cdots \to
	\mathrm{MU}[y]/(y^k) \to \mathrm{MU}[y]/(y^{k-1})
	\to \cdots \to \mathrm{MU}
	\] 
When $j>0$, this is also a limit diagram of underlying 
$\mathbb{E}_{\infty}$-$\mathrm{MU}$-algebras
and square-zero extensions thereof. In our work below,
we regard $\mathrm{MU}[y]$ and
$\mathrm{MU}[y]/(y^k)$ as \emph{ungraded}
$\mathbb{E}_{\infty}$-$\mathrm{MU}$-algebras.
\end{construction}

\begin{proposition}\label{prop:muy-algebra}
Fix $n\ge -1$ and let $B$ be any 
$\mathbb{E}_3$-$\mathrm{MU}$-algebra form of 
$\mathrm{BP}\langle n\rangle$.
Then there exists 
a sequence of maps
	\[
	\cdots \to B_k \to B_{k-1} \to \cdots B_1 = B
	\]
where:
	\begin{enumerate}[{\rm (a)}]
	\item $B_k$ is given the structure of an 
	$\mathbb{E}_3$-$\mathrm{MU}[y]/(y^k)$-algebra,
	where $|y|=2p^{n+1}-2$.
	\item each map $B_k \to B_{k-1}$ is given the structure of
	a map of $\mathbb{E}_3$-$\mathrm{MU}[y]/(y^k)$-algebras,
	\end{enumerate}
and such that the following properties are satisfied:
	\begin{enumerate}[{\rm (i)}]
	\item Each map $B_k \to B_{k-1}$ induces an
	equivalence of $\mathbb{E}_3$-$\mathrm{MU}[y]/(y^k)$-algebras
		\[
		\mathrm{MU}[y]/(y^{k-1}) \otimes_{\mathrm{MU}[y]/(y^k)} B_k
		\stackrel{\simeq}{\to}B_{k-1}.
		\]
	\item The map
		\[
		B=B_1 \to \mathrm{cofib}(B_2 \to B_1) \simeq
		\Sigma^{|y|+1}B
		\]
	is detected by $\delta v_{n+1}$ in
		\[
		E_2=E_{\infty} =
		\mathrm{Ext}^*_{\mathrm{MU}_*}(B_*, B_*)
		\Rightarrow \pi_*\mathrm{map}_{\mathrm{MU}}(B,B).
		\]
	\end{enumerate}

\end{proposition}

\begin{proof}
First we prove that we can build $B_2$ satisfying (ii). 
Observe that $\mathrm{MU}[y]/(y^2)$ is a trivial
square-zero extension of $\mathrm{MU}$ by $\Sigma^{|y|}\mathrm{MU}$.
By Lemma \ref{lem:alg-over-triv-sq-zero}, it suffices to supply
a map
	\[
	\mathbf{L}^{\mathbb{E}_3}_{B/\mathrm{MU}} \to
	\Sigma^{|y|+1}\mathrm{BP}\langle n\rangle
	\]
in $\mathsf{Mod}^{\mathbb{E}_3}_{B}(\mathsf{Mod}_{\mathrm{MU}})$
whose image under the forgetful map

	\[
		\pi_*\mathrm{map}_{\mathcal{U}^{(3)}_{\mathrm{MU}}(
		\mathrm{BP}\langle n\rangle)}
		(\mathbf{L}^{\mathbb{E}_3}_{\mathrm{BP}\langle n
		\rangle/\mathrm{MU}}, \mathrm{BP}\langle n\rangle)
		\longrightarrow
		\pi_*\mathrm{map}_{
		\mathrm{MU}}(\mathrm{fib}(\mathrm{MU}_{(p)}
		\to \mathrm{BP}\langle n\rangle),
		\mathrm{BP}\langle n\rangle).
		\]
detects $\delta v_{n+1}$. But this is precisely the content of
Proposition \ref{prop:bpn-obstrn-zero}(ii).

Suppose by induction that we have constructed the algebras
$B_j$ for $j\le k$ as in (a) and (b), satisfying (i) and (ii).
By Proposition \ref{prop:obstrn-theory}, the obstruction to building
$B_{k+1}$ is a map
	\[
	o(B_k): \mathbf{L}^{\mathbb{E}_3}_{B_{k}/(\mathrm{MU}[y]/(y^k))}
	\to 
	B_k \otimes_{\mathrm{MU}[y]/(y^k)}
	\Sigma^{(k+1)|y|+2}\mathrm{MU}
	\]
in $\mathsf{Mod}^{\mathbb{E}_3}_{B_k}(
\mathsf{Mod}_{\mathrm{MU}[y]/(y^k)})$. Base change along
the augmentation
$\mathrm{MU}[y]/(y^k) \to \mathrm{MU}$ gives rise to a functor
	\[
	\varepsilon^*: 
	\mathsf{Mod}^{\mathbb{E}_3}_{B_k}(
	\mathsf{Mod}_{\mathrm{MU}[y]/(y^k)})
	\longrightarrow
	\mathsf{Mod}^{\mathbb{E}_3}_{B}(
	\mathsf{Mod}_{\mathrm{MU}}),
	\]
where we have used (i) to identify $\varepsilon^*B_k \simeq B$. 
So the obstruction $o(B_k)$ is adjoint to a map
	\[
	\mathbf{L}^{\mathbb{E}_3}_{B/\mathrm{MU}}
	\to 
	\Sigma^{(k+1)|y|+2}B
	\]
in $\mathsf{Mod}^{\mathbb{E}_3}_{B}(\mathsf{Mod}_{\mathrm{MU}})$.
By Proposition \ref{prop:bpn-obstrn-zero}(i),
any such map is nullhomotopic. This completes the proof.
\end{proof}

\begin{proof}[Proof of Theorem \ref{thm:multn-form}]
We prove the claim by induction on $n\ge 1$, the cases
$n=-1, 0$ being trivial. By induction, suppose there exists an
$\mathbb{E}_3$-$\mathrm{MU}$-algebra form of 
$\mathrm{BP}\langle n\rangle$, say $B$, and construct
a tower as in the previous proposition.

Let $\widetilde{B}:= \lim B_k$ be the 
$\mathbb{E}_3$-$\mathrm{MU}[y]$-algebra
at the limit of the tower. We claim that
$\widetilde{B}$ is an $\mathbb{E}_3$-$\mathrm{MU}$-algebra
form of $\mathrm{BP}\langle n+1\rangle$.

By (i), the associated graded tower
is a sum of shifts of $B$, and we see that
	\[
	\mathrm{gr}(\pi_*\widetilde{B}) \simeq B_*[y].
	\]
By (ii), the exact sequence of $\mathrm{MU}_*$-modules
	\[
	0 \to \Sigma^{|y|}B_* \stackrel{\cdot y}{\to} \pi_*B_2 \to B_* \to 0 
	\]
corresponds to the class $\delta v_{n+1}$, so that $y$ acts by
$v_{n+1}$ on $B_2$. Combining these observations we see that the
composite
	\[
	\mathbb{Z}_{(p)}[v_1, ..., v_{n+1}] \to (\mathrm{MU}_*)_{(p)}
	\to \widetilde{B}_*
	\]
is an isomorphism, which is what we wanted to show.
\end{proof}

\section{Unraveling Lichtenbaum-Quillen}\label{sec:abstract-ql}

Having constructed an $\mathbb{E}_3$-$\mathrm{MU}$-algebra form of $\mathrm{BP}\langle n \rangle$, we aim in the remainder of the paper to study its $p$-localized algebraic $K$-theory spectrum $\mathrm{K}(\mathrm{BP}\langle n \rangle)_{(p)}$.  The philosophy of chromatic homotopy theory, together with the vanishing results of \cite{cmnn}, suggests that we should study $\mathrm{K}(\mathrm{BP}\langle n \rangle)_{(p)}$ by computing its chromatic localizations $L_{T(i)} \mathrm{K}(\mathrm{BP}\langle n \rangle)_{(p)}$ for $0 \le i \le n+1$, which assemble into the smashing localization $L_{n+1}^{f} \mathrm{K}(\mathrm{BP}\langle n \rangle)_{(p)}$.  One wants to know whether the localization $L_{n+1}^{f} \mathrm{K}(\mathrm{BP}\langle n \rangle)_{(p)}$ faithfully reproduces $\mathrm{K}(\mathrm{BP} \langle n \rangle)_{(p)}$.  This is far from assured: for example, Quillen proved that $\mathrm{K}(\mathbb{F}_p)_{(p)} \cong \mathbb{Z}_{(p)}$, but $L_0^{f} \mathbb{Z}_{(p)} = \mathbb{Q}$.

It turns out, however, that the difference between $\mathrm{K}(\mathrm{BP}\langle n \rangle)_{(p)}$ and its $L_{n+1}^f$ localization is entirely concentrated in low degrees.  In short, in the remainder of the paper we aim to prove that the localization map $\mathrm{K}(\mathrm{BP}\langle n \rangle)_{(p)} \to L_{n+1}^f \mathrm{K}(\mathrm{BP}\langle n \rangle)_{(p)}$ is truncated.

In the case $n=0$, this becomes the classical Lichtenbaum--Quillen conjecture for $\mathbb{Z}_{(p)}$, which is a celebrated theorem of Voevodsky and Rost \cite{voeI, voeII}.  Our goal is to reduce the general case to the Voevodsky--Rost theorem.  We accomplish this first by leaning on the Dundas--Goodwillie--McCarthy theorem \cite[Theorem 0.0.2]{DGM-book}, which relates $\mathrm{K}$ to the more computable ($p$-completed) $\mathrm{TC}$ invariant that we review in \S \ref{sec:cyclotomic-conventions}. The purpose of this section is to discuss a general strategy for proving, for any connective $\mathbb{E}_1$-ring spectrum $R$, that 
\[\mathrm{TC}(R) \to L_{n+1}^{f} \mathrm{TC}(R)\]
is truncated.  Future sections of the paper then implement the strategy in the case of $R=\mathrm{BP} \langle n \rangle$.

In broad outline, our strategy proceeds as follows.  First, following Ausoni and Rognes, we apply work of Mahowald and Rezk to reduce to proving that $\pi_*(F \otimes \mathrm{TC}(R))$ is finite for some type $n+2$ complex $F$ \cite{mahowald-rezk}.  As we will review, $\mathrm{TC}(R)$ is constructed from the simpler invariant $\mathrm{THH}(R)$ together with two pieces of structure, a \emph{cyclotomic Frobenius} and a \emph{circle action}.  One of the main observations of this paper is that, while the definition of $\mathrm{TC}$ mixes these structures by taking $S^1$ fixed points of the Frobenius, it actually suffices to study the two structures independently.

We reduce the problem of checking that $\pi_*(F \otimes \mathrm{TC}(R))$ is bounded to checking that $\mathrm{THH}(R)$ satisfies the \emph{Segal conjecture}, which is purely about the Frobenius, and that $\mathrm{THH}(R)$ satisfies \emph{canonical vanishing}, which is purely about the $S^1$ action.

These two properties of $\mathrm{THH}(R)$ together imply that $F \otimes \mathrm{TR}(R)$ is bounded above, and we thank the third referee for pointing out that they are in fact \emph{equivalent} to the statement that $F \otimes \mathrm{TR}(R)$ is bounded above.  After checking that the homotopy groups $\pi_i R$ are finitely generated, the statement that $F \otimes \mathrm{TR}(R)$ is bounded above implies that $\pi_*(F \otimes \mathrm{TC}(R))$ is finite.

The remainder of this section fixes conventions and makes precise the reductions to canonical vanishing and the Segal conjecture:
\S\ref{sec:segal} verifies the Segal conjecture for $\mathrm{THH}(\mathrm{BP}\langle n\rangle)$, 
while
\S\ref{sec:canvan} verifies canonical vanishing via entirely different means.

\subsection{The work of Mahowald--Rezk}

\begin{definition}
A $p$-complete, bounded below spectrum $X$ is said to be \textbf{fp} if $H^*(X;\mathbb{F}_p)$ is finitely presented over the mod $p$ Steenrod algebra.
\end{definition}

\begin{theorem}[Mahowald--Rezk]
Suppose that $X$ is an fp spectrum.  Then there exists a non-zero $p$-local finite complex $F$ such that $\pi_*(X \otimes F)$ is finite. In other words, $X \otimes F$ has finitely many non-zero homotopy groups, and $\pi_i(X \otimes F)$ is finite for each $i$.

On the other hand, suppose $Y$ is any bounded below $p$-complete spectrum. If there exists a non-zero $p$-local finite complex $F$ such that $\pi_*(Y \otimes F)$ is finite, then $Y$ is an fp spectrum.
\end{theorem}

\begin{proof}
This is \cite[Proposition 3.2]{mahowald-rezk}.
\end{proof}

The collection of $F$ such that $\pi_*(X \otimes F)$ is finite
is obviously a thick subcategory.  One says that an fp spectrum $X$ is \textbf{of fp-type $n$} if $\pi_*(X \otimes F)$ is infinite when $F$ has type $n$, but finite when $F$ has type $n+1$.  Our key interest in fp spectra comes from the following result of Mahowald and Rezk:

\begin{theorem}[Mahowald--Rezk] \label{fp-implies-LQ}
If $Y$ is a spectrum of fp-type $n$ then the localization map
$Y \to L_n^fY$ is an equivalence on homotopy in large degrees.
\end{theorem}
\begin{proof} Mahowald and Rezk \cite[Theorem 8.2(2)]{mahowald-rezk}
prove that the fiber $C$ of the map $Y \to L_n^fY$ has Brown-Comenetz
dual $IC$ which is bounded below. It follows that $C$ is bounded
above (since, for any abelian group $N$, if 
$\mathrm{Hom}(N, \mathbb{Q}/\mathbb{Z}) = 0$ then
$N=0$), whence the claim.
\end{proof}

\subsection{Background and conventions on cyclotomic spectra} \label{sec:cyclotomic-conventions}

\begin{definition}
A \textbf{$p$-typical cyclotomic spectrum} $X$
is a $p$-complete object $X \in \mathsf{Fun}(\mathrm{B}S^1, \mathsf{Sp})$
equipped with an $S^1$-equivariant map $\varphi: X \to X^{tC_p}$,
where the action on the target is via the equivalence $S^1\cong S^1/C_p$.
If $X$ is a bounded below 
$p$-typical cyclotomic spectrum, so that $(X^{tC_p})^{hC_{p^k}}
\simeq X^{tC_{p^{k+1}}}$ by \cite[II.4.1]{nikolaus-scholze},
then we define
invariants:
	\[
	\mathrm{TR}^j(X):= X\times_{X^{tC_p}}X^{hC_p} \times_{
	X^{tC_{p^2}}} X^{hC_{p^2}} \times_{
	X^{tC_{p^3}}} \cdots
	\times_{X^{tC_{p^j}}} X^{hC_{p^j}}
	\]
using the maps $\varphi^{hC_{p^k}}$ and the canonical
maps from homotopy fixed points to the Tate fixed points.
We define $\mathrm{TR}(X) = \lim_j \mathrm{TR}^j(X)$ where
the connecting maps 
	\[
	R: \mathrm{TR}^j(X) \to \mathrm{TR}^{j-1}(X)
	\]
are projection away from the last factor.
Observe that each object $\mathrm{TR}^j(X)$ and the limit
$\mathrm{TR}(X)$ has a residual $S^1$-action.

\end{definition}

\begin{remark} This is slightly different than the notion of
a `$p$-cyclotomic spectrum' considered in \cite{nikolaus-scholze}.
However, when restricting attention to bounded below and $p$-complete
objects, as we do here,
the two notions coincide (see \cite[Remark II.1.3]{nikolaus-scholze}).
The definition above is the same as in \cite{antieau-nikolaus}
except that we have added the hypothesis that $X$ be $p$-complete.
\end{remark}

\begin{definition} If $X$ is a bounded below,
$p$-typical cyclotomic spectrum
then we define
	\[
	\mathrm{TC}(X):= \mathrm{fib}(\varphi^{hS^1}-\mathrm{can}:
	X^{hS^1} \to X^{tS^1}).
	\]
\end{definition}

\begin{remark} There are maps
$F: \mathrm{TR}^n(X) \to \mathrm{TR}^{n-1}(X)$ corresponding
to projecting away from the first factor and then using
the inclusion of each $C_{p^k}$ homotopy fixed points into
the $C_{p^{k-1}}$ homotopy fixed points. These assemble to
a map $F: \mathrm{TR}(X) \to \mathrm{TR}(X)$ and the original
definition of ($p$-adic) topological cyclic homology was as the fiber:
	\[
	\mathrm{fib}(1-F: \mathrm{TR}(X) \to \mathrm{TR}(X)).
	\]
It is shown in \cite[Theorem II.4.10]{nikolaus-scholze} that this
agrees with the definition above.
\end{remark}

\begin{remark} Antieau-Nikolaus construct
\cite[Example 3.4]{antieau-nikolaus} an $S^1$-equivariant map
$V: \mathrm{TR}(X)_{hC_p} \to \mathrm{TR}(X)$ which fits into a
cofiber sequence
	\[
	\mathrm{TR}(X)_{hC_p} \stackrel{V}{\to} \mathrm{TR}(X) \to X
	\]
of $S^1$-spectra. Thus, one can recover both $\mathrm{TC}(X)$
and $X$ from knowledge of $\mathrm{TR}(X)$. 
\end{remark}

\begin{definition}
Suppose that $A$ is a connective $\mathbb{E}_1$-ring spectrum.
Then $\mathrm{THH}(A)^{\wedge}_p$ is a bounded below
$p$-typical cyclotomic
spectrum.
We will in this circumstance abbreviate 
$\mathrm{TR}^j(\mathrm{THH}(A)^{\wedge}_p)$ by $\mathrm{TR}^j(A)$, and similarly for $\mathrm{TR}(\mathrm{THH}(A)^{\wedge}_p)$ and $\mathrm{TC}(\mathrm{THH}(A)^{\wedge}_p)$.
\end{definition}

\subsection{Bounds on $\mathrm{TR}$ and related conditions on cyclotomic spectra}\label{sec:bound-tr-etc}

\begin{definition} Let $X$ be a bounded below
$p$-typical cyclotomic spectrum. We will be interested in the following
conditions on $X$, which may or may not hold:
	\begin{itemize}
	\item \fbox{$\mathrm{Bounded}\text{ }\mathrm{TR}$} The spectrum
	$\mathrm{TR}(X)$ is bounded.
	\item \fbox{$\mathrm{Segal}\text{ }\mathrm{Conjecture}$} The Frobenius
	$\varphi: X \to X^{tC_p}$ is truncated.
	\item \fbox{$\mathrm{Canonical}\text{ }\mathrm{Vanishing}$}
	There is an integer $d\ge 0$ such that the composite
		\[
		\tau_{\ge d}(X^{hC_{p^k}}) \to
		X^{hC_{p^k}} \stackrel{\mathrm{can}}{\to}
		X^{tC_{p^k}}
		\]
	is nullhomotopic for all $0\le k \le \infty$.
	\item \fbox{$\mathrm{Weak}\text{ }\mathrm{Canonical}\text{ }\mathrm{Vanishing}$} There is an integer
	$d\ge 0$ such that, for $\ast\ge d$, the map
		\[
		\pi_*(\mathrm{can}): \pi_*X^{hC_{p^k}} \to \pi_*X^{tC_{p^k}}
		\]
	is zero for all $0\le k\le \infty$.
	\item \fbox{$\mathrm{Tate}\text{ }\mathrm{Nilpotence}$} $X^{tC_p}$ lies in the thick tensor ideal
	of $\mathsf{Fun}(\mathrm{B}S^1, \mathsf{Sp})$ generated
	by $\mathbb{D}S^1_+$, the Spanier-Whitehead dual of $S^1_+$.
	\item \fbox{$\mathbb{F}_p\text{ }\mathrm{Nilpotence}$} $\mathrm{TR}(X) \in \mathsf{Fun}(\mathrm{B}S^1,
\mathsf{Sp})$ lies in the thick tensor ideal generated by
$\mathbb{F}_p$, where $\mathbb{F}_p$ is considered to have trivial $S^1$ action.
	\item \fbox{$\mathrm{Finiteness}$} For each $i \in \mathbb{Z}$ and $0\le k \le \infty$,
	the groups
	$\pi_iX^{hC_{p^k}}$ and $\pi_iX^{tC_{p^k}}$ are finite,
	and hence so too are the groups $\pi_i\mathrm{TC}(X)$.
	\end{itemize}
\end{definition}

As suggested by its name, the \fbox{$\mathrm{Segal}\text{ }\mathrm{Conjecture}$} condition holds particular historical significance, some of which we recall in Section \ref{sec:segal}.
It turns out that there are many nontrivial implications between the conditions, summarized by the following theorem:

\begin{theorem}\label{thm:abstract-implications}
Let $X$ be a bounded below, $p$-power torsion $p$-typical
cyclotomic spectrum.
That is, we assume there is some $N\ge 0$
for which $p^N: X \to X$ is nullhomotopic as a map
of $p$-typical cyclotomic spectra. Then the following implications hold:

\begin{enumerate}[{\rm (a)}]
\item (Antieau--Nikolaus) \fbox{$\mathrm{Bounded}\text{ }\mathrm{TR}$} $\Rightarrow$ 
\fbox{$\mathrm{Segal}\text{ }\mathrm{Conjecture}$}.
\item \fbox{$\mathrm{Bounded}\text{ }\mathrm{TR}$} $\Rightarrow$ \fbox{$\mathbb{F}_p\text{ }\mathrm{Nilpotence}$}.
\item (Mathew) \fbox{$\mathbb{F}_p\text{ }\mathrm{Nilpotence}$} $\Rightarrow$ \fbox{$\mathrm{Tate}\text{ }\mathrm{Nilpotence}$}.
\item If each homotopy group $\pi_iX$ is finite, then
	\fbox{$\mathrm{Bounded}\text{ }\mathrm{TR}$} $\Rightarrow$
	\fbox{$\mathrm{Finiteness}$}.
\item \fbox{$\mathrm{Segal}\text{ }\mathrm{Conjecture}$} $+$ \fbox{$\mathrm{Tate}\text{ }\mathrm{Nilpotence}$}
$\Rightarrow$ \fbox{$\mathrm{Canonical}\text{ }\mathrm{Vanishing}$}.
\item \fbox{$\mathrm{Segal}\text{ }\mathrm{Conjecture}$} $+$ \fbox{$\mathrm{Weak}\text{ }\mathrm{Canonical}\text{ }\mathrm{Vanishing}$}
$\Rightarrow$ \fbox{$\mathrm{Bounded}\text{ }\mathrm{TR}$}.
\end{enumerate}
\end{theorem}

\begin{remark}
We thank the anonymous referee for suggesting both the formulations and proofs of several of the statements in the above \Cref{thm:abstract-implications}.  Of the statements, $(a)$ appeared in previous work of Antieau and Nikolaus \cite[Proposition 2.25]{antieau-nikolaus}, and $(c)$ was communicated to us by Akhil Mathew. We thank both the referee and Mathew for suggesting that we present their work within this paper.
\end{remark}

We postpone the proof of Theorem \ref{thm:abstract-implications} to Section \ref{subsec:abstract-implications-proof}.  Let us now describe how we apply it. The main theorem of the remainder of the paper, stated
as Theorem \ref{thm:intro-bounded-tr} in the introduction,
is the following:

\begin{theorem} \label{thm:bpn-bounded-tr}
Let $\mathrm{BP} \langle n \rangle$ denote any $\mathbb{E}_3$-$\mathrm{MU}$-algebra form of $\mathrm{BP}\langle n \rangle$, and suppose that $F$ is a type $n+2$ complex.  Then $F \otimes \mathrm{THH}(\mathrm{BP}\langle n \rangle)$ satisfies \fbox{$\mathrm{Bounded}\text{ }\mathrm{TR}$}.
\end{theorem}

By the thick subcategory theorem of Hopkins and Smith
\cite{hopkins-smith},
Theorem \ref{thm:bpn-bounded-tr} holds for an arbitrary type $n+2$ complex $F$ if and only if it holds for \emph{some} type $n+2$ complex $F$.  Thus, given Theorem \ref{thm:abstract-implications}(d,f), we can prove Theorem \ref{thm:bpn-bounded-tr} by checking the following two results independently:

\begin{theorem} \label{thm:bpn-segal}
(see Theorem \ref{thm:segal-conjecture-main})
For all type $n+2$ complexes $F$,
$F \otimes \mathrm{THH}(\mathrm{BP} \langle n \rangle)$
satisfies the Segal conjecture.
\end{theorem}

\begin{theorem} \label{thm:bpn-weak-can-vanishing}
(see Theorem \ref{thm:canvan-new})
For some type $n+2$ complex $F$,
$F \otimes \mathrm{THH}(\mathrm{BP} \langle n \rangle)$
satisfies weak canonical vanishing.
\end{theorem}

While it is convenient for our proof of Theorem \ref{thm:bpn-weak-can-vanishing} that we pick a particularly nice $F$, it follows from Theorem \ref{thm:abstract-implications}(a,b) that it holds for all choices of $F$. As Akhil Mathew explained to us, we may also use Theorem \ref{thm:abstract-implications} to deduce results about general $\mathbb{E}_1$-$\mathrm{BP}\langle n \rangle$-algebras:

\begin{proposition}[Mathew] \label{prop:mathew-bpn} Suppose that $A$ is a connective $\mathbb{E}_1$-$\mathrm{BP}\langle n\rangle$-algebra, and that $F$ is a type $n+2$ complex.
Then, if $F \otimes \mathrm{THH}(A)$ satisfies the Segal conjecture, $F \otimes \mathrm{TR}(A)$ is bounded.
\end{proposition}

\begin{proof}
It suffices to show that $F \otimes \mathrm{TR}(A)$ satisfies \fbox{$\mathbb{F}_p\text{ }\mathrm{Nilpotence}$}.  However, $F \otimes \mathrm{TR}(A)$ is a retract of $F \otimes \mathrm{TR}(\mathrm{BP}\langle n\rangle) \otimes \mathrm{TR}(A)$, so it suffices to check that $F \otimes \mathrm{TR}(\mathrm{BP}\langle n\rangle)$ satisfies \fbox{$\mathbb{F}_p\text{ }\mathrm{Nilpotence}$}.  This follows from 
Theorem \ref{thm:bpn-bounded-tr} and 
Theorem \ref{thm:abstract-implications}(b).
\end{proof}

\subsection{Lichtenbaum--Quillen and bounded $\mathrm{TR}$}

\begin{theorem} \label{tc-lq-master}
Let $A$ be a connective $\mathbb{E}_1$-ring and $F$ a type $n+2$ complex.
Suppose that
	\begin{enumerate}[{\rm (i)}]
	\item $F\otimes \mathrm{TR}(A)$ is bounded.
%	\item $\mathrm{K}(\pi_0A)$ is fp.
%	\item $L_{K(n+1)}\mathrm{TC}(A)\ne 0$.
	\item $\pi_i(A_p^{\wedge})$ is a finitely generated $\mathbb{Z}_p$-module, for all $i$. 
	\end{enumerate}
Then $\mathrm{TC}(A)$ is an fp-spectrum of fp-type at most $n+1$.  In particular, this implies that
	\[\mathrm{TC}(A) \to L_{n+1}^f \mathrm{TC}(A)\]
	is truncated.
\end{theorem}

\begin{proof}
Since $\mathrm{TC}$ is calculated as the fiber of a self-map of $\mathrm{TR}$, we know by assumption (i) that $\pi_*(F \otimes \mathrm{TC}(A))$ is bounded.
It remains only to check that each $\pi_i(F \otimes \mathrm{TC}(A))$ is finite.  By Theorem \ref{thm:abstract-implications}(e), it suffices to show that each homotopy group
$\pi_i(F \otimes \mathrm{THH}(A))$
is finite.  Recall that $\mathrm{THH}(A)$ can be computed as the geometric realization of the cyclic bar construction $\bullet \mapsto A^{\otimes \bullet+1}$.  Since $A$ is connective, it therefore suffices to prove $\pi_{i-k}(F \otimes A^{\otimes k+1})$ is finite for each $i$ and $k$.  The $p$-completion of the tensor product $A^{\otimes k+1}$ will have finitely generated homotopy groups, by connectivity and $(ii)$.  Since $F$ is not type $0$, the result follows.
\end{proof}

\begin{theorem}
Let $A$ be a connective $\mathbb{E}_1$-ring and $F$ a type $n+2$ complex.
Suppose that
	\begin{enumerate}[{\rm (i)}]
	\item $F\otimes \mathrm{TR}(A)$ is bounded.
%	\item $\mathrm{K}(\pi_0A)$ is fp.
%	\item $L_{K(n+1)}\mathrm{TC}(A)\ne 0$.
	\item $\pi_i (A_p^{\wedge})$ is a finitely generated $\mathbb{Z}_p$-module for all $i$. 
	\item $F \otimes \mathrm{TR}\left(\pi_0A\right)$ is bounded.
	\end{enumerate}
	Then,
	\begin{enumerate}[{\rm (a)}]
	\item If $\pi_*(F \otimes \mathrm{K}(\pi_0 A))$ is finite, then $\mathrm{K}(A)_p^{\wedge}$ is an fp-spectrum of fp-type at most $n+1$.
	\item If the map $\mathrm{K}(\pi_0 A)_{(p)} \to L^{f}_{n+1} \mathrm{K}(\pi_0A)_{(p)}$ is truncated, then the map
		\[
		\mathrm{K}(A)_{(p)} \to L_{n+1}^f\mathrm{K}(A)_{(p)}
		\]
	is truncated.
	\end{enumerate}
\end{theorem}

\begin{remark}
The condition in $(a)$ of the above theorem is that $\mathrm{K}(\pi_0 A)_p^{\wedge}$ is fp of fp type at most $n+1$.  Mitchell's theorem \cite{Mitchell} ensures that, if $\mathrm{K}(\pi_0 A)_p^{\wedge}$ is fp, then it will be of fp type at most $1$.  Similarly, Mitchell's theorem implies that the spectrum $L^{f}_{n+1} \mathrm{K}(\pi_0A)$ appearing in $(b)$ is equivalent to $L_1^{f} \mathrm{K}(\pi_0A)$.
\end{remark}

\begin{proof}
By $p$-completing the pullback square in the Dundas--Goodwillie--McCarthy theorem \cite[Chapter VII, Theorem 0.0.2]{DGM-book}, we obtain a pullback square
\[
\begin{tikzcd}
\mathrm{K}(A)_p^{\wedge} \arrow{r} \arrow{d} & \mathrm{TC}(A) \arrow{d} \\
\mathrm{K}(\pi_0A)_p^{\wedge} \arrow{r} & \mathrm{TC}(\pi_0 (A)).
\end{tikzcd}
\]
Here we note that the symbol $\mathrm{TC}$ above agrees,
by our conventions, with the $p$-completion of what the authors
of \cite{DGM-book} denote by $\mathrm{TC}$.

The assumption that each homotopy group $\pi_i A_p^{\wedge}$ is finitely generated ensures additionally that $(\pi_0A)_p^{\wedge}$, by which we mean the $p$-completion of the Eilenberg--MacLane spectrum $\pi_0A$, also has finitely generated homotopy groups.
By Theorem \ref{tc-lq-master}, we know that $\mathrm{TC}(\pi_0 A)$ and $\mathrm{TC}(A)$ are fp-spectra of type at most $n+1$.  We then observe that the condition of being an fp-spectrum of type at most $n+1$ is preserved under fiber sequences and finite coproducts, proving $(a)$.

Similarly, to prove $(b)$ we observe that the collection of spectra $X$ such that
\[X_{(p)} \to L_{n+1}^{f} X_{(p)}\]
is truncated is also closed under fiber sequences and finite coproducts. This class of
spectra includes all rational spectra. The claim $(b)$ now follows
from Theorems \ref{fp-implies-LQ} and \ref{tc-lq-master}, the Dundas-Goodwillie-McCarthy square above, and the arithmetic pullback square
	\[
	\xymatrix{
	X_{(p)} \ar[r]\ar[d] & X^{\wedge}_p \ar[d]\\
	X_{(p)}[p^{-1}] \ar[r] & X^{\wedge}_p[p^{-1}]
	}
	\]
\end{proof}

As a corollary of these results, we deduce the main theorems of the Introduction.

\begin{corollary}
Let $A$ denote any $\mathbb{E}_3$-$\mathrm{MU}$-algebra form of $\mathrm{BP}\langle n \rangle$.  Then:
\begin{itemize}
\item $\mathrm{TC}(A)$ is fp of fp-type at most $n+1$, as is $\mathrm{K}(A_p^{\wedge})_{p}^{\wedge}$.
\item The map 
\[\mathrm{K}(A)_{(p)} \to L_{n+1}^{f} \mathrm{K}(A)_{(p)}\]
is an equivalence in large degrees.
\end{itemize}
\end{corollary}

\begin{proof}
We observe that the $\mathbb{Z}_p$-module
\[\pi_*(\mathrm{BP}\langle n \rangle_p^{\wedge}) \cong \mathbb{Z}_p[v_1,v_2,\cdots,v_n]\]
is finitely generated in each degree.  If we let $F$ denote any type $n+2$ complex, our main Theorem \ref{thm:bpn-bounded-tr} states that $F \otimes \mathrm{TR}(A)$ is bounded. We also observe that $F \otimes \mathrm{TR}(\pi_0 A) \simeq F \otimes \mathrm{TR}(\mathbb{Z}_{(p)})$ is bounded, for example by our work here and the fact that $\mathbb{Z}_{(p)}$ is an $\mathbb{E}_{\infty}$-$\mathrm{MU}$-algebra form of $\mathrm{BP}\langle 0\rangle$ (cf. \cite{BokstedtMadsen} and \cite{rognes-2adicZ} for more explicit proofs of this fact).  It remains to check only that
\[\mathrm{K}(\mathbb{Z}_{(p)})_{(p)} \to L_{n+1}^{f} \mathrm{K}(\mathbb{Z}_{(p)})_{(p)} \simeq L_1^{f} \mathrm{K}(\mathbb{Z}_{(p)})_{(p)}\]
is an equivalence in large degrees.  But this is exactly Waldhausen's reformulation \cite[\S 4]{waldhausen} of the Lichtenbaum--Quillen conjecture for $\mathbb{Z}_{(p)}$, which is proven by the celebrated work of Voevodsky and Rost \cite{voeI,voeII}.
\end{proof}

\begin{remark}
In fact, in the notation of the above corollary, $\mathrm{TC}(A)$ is of fp-type exactly $n+1$, as follows from Corollary \ref{cor:weak-redshift}.
\end{remark}

\subsection{Proof of Theorem \ref{thm:abstract-implications}} \label{subsec:abstract-implications-proof}

In this section we supply the proof of Theorem \ref{thm:abstract-implications}.
The statements and proofs will rely on the following notion of
nilpotence, which
goes back at least to Bousfield \cite{bousfield}. For an excellent survey
of modern uses of this notion we recommend \cite{mathew-nilpotence}.

\begin{definition} Let $\mathcal{C}$ be a stable, symmetric monoidal
category and let $A$ be an $\mathbb{E}_1$-algebra object.
We say that $M \in \mathcal{C}$ is \textbf{$A$-nilpotent} if 
$M$ lies in the thick tensor subcategory generated by
$A$. Equivalently, we can ask that $M$ lies in the thick subcategory
generated by those objects of $\mathcal{C}$ which admit the structure
of a left $A$-module.
\end{definition}

\begin{definition} Let $G$ be a compact Lie group. 
We say that a (Borel) $G$-spectrum $Y$ is
\textbf{$G$-nilpotent} if it is $\mathbb{D}G_+$-nilpotent,
where $\mathbb{D}G_+$ denotes the Spanier-Whitehead
dual of $G_+$.
\end{definition}

\begin{lemma}\label{lem:nilpotent-check} Let $F: \mathcal{C} \to \mathcal{D}$
be a lax symmetric monoidal functor between stable,
symmetric monoidal categories, and $A \in \mathsf{Alg}(\mathcal{C})$
and $B \in \mathsf{Alg}(\mathcal{D})$ be algebra objects.
If $F(A)$ is $B$-nilpotent, then $F(M)$ is $B$-nilpotent
for any $A$-nilpotent object $M$.
\end{lemma}
\begin{proof} The subcategory $\mathcal{E}\subseteq \mathcal{C}$
of objects $M$ such that $F(M)$ is $B$-nilpotent is thick
so we need only show that it contains all $A$-modules.
If $M$ is an $A$-module then $F(M)$ is an $F(A)$-module
and hence a retract of $F(A) \otimes F(M)$. But $F(A)$
is $B$-nilpotent and hence so is $F(A) \otimes F(M)$
and the retract $F(M)$.
\end{proof}

\begin{lemma}\label{lem:fp-nil-orbits} If $Y \in \mathsf{Fun}(\mathrm{B}S^1,
\mathsf{Sp})$ is $\mathbb{F}_p$-nilpotent, where
$\mathbb{F}_p$ is given the trivial action, then
$Y_{hC_p}$, $Y^{hC_p}$, and $Y^{tC_p}$
are also $\mathbb{F}_p$-nilpotent, where
we give each the residual $S^1/C_p \simeq S^1$ action.
\end{lemma}
\begin{proof} From the cofiber sequence
$Y_{hC_p} \to Y^{hC_p} \to Y^{tC_p}$ it's enough to
prove the claim for $Y^{hC_p}$ and $Y^{tC_p}$.
By Lemma \ref{lem:nilpotent-check} it's enough
to check that $\mathbb{F}_p^{hC_p}$ and $\mathbb{F}_p^{tC_p}$
are $\mathbb{F}_p$-nilpotent. Both are modules over
$\mathbb{F}_p^{hC_p}$ and hence also, by restriction,
modules over $\mathbb{F}_p$, hence $\mathbb{F}_p$-nilpotent.
\end{proof}

\begin{lemma}\label{lem:fp-nil-to-s1-nil} If $Y\in \mathsf{Fun}(\mathrm{B}S^1, \mathsf{Sp})$
is $\mathbb{F}_p$-nilpotent, then $Y^{tC_p}$ is $S^1$-nilpotent.
\end{lemma}
\begin{proof} First we claim that $(\mathbb{F}_p^{hC_p})^{tS^1}=0$.
One can check this by direct calculation, or else argue as follows.
Since $\mathbb{F}_p^{hC_p}$ is a module over $\mathbb{F}_p$,
we have that $(\mathbb{F}_p^{hC_p})^{tS^1}$ is $p$-complete
and $(\mathbb{F}_p^{hC_p})^{tS^1}/p = 
(\mathbb{F}_p^{hC_p})^{tC_p}$, by \cite[Lemma IV.4.12]{nikolaus-scholze}.
But $(\mathbb{F}_p^{hC_p})^{tC_p}=0$ by
\cite[Lemma I.2.2]{nikolaus-scholze}, whence the claim.
It now follows from \cite[Theorem 4.19]{mathew-naumann-noel}
that $\mathbb{F}_p^{hC_p}$ is
$S^1$-nilpotent. The lemma now follows from
Lemma \ref{lem:nilpotent-check}.
\end{proof}

\begin{lemma} If $Y \in \mathsf{Fun}(\mathrm{B}G, \mathsf{Sp})$
is $G$-nilpotent then there is a $d\ge 0$ such that,
for all integers $n$, the map
$\tau_{\ge d+n}Y \to \tau_{\ge n}Y$ factors through a $G$-nilpotent spectrum.
\end{lemma}
\begin{proof} Choose an $N$ so that the map
$Y \to \mathrm{map}(\mathrm{sk}_N(EG)_+, Y)$ has a retract $r$.
Let $d$ be the dimension of the finite complex $\mathrm{sk}_N(EG)_+$.
Then, for all $n\in \mathbb{Z}$,
the spectrum $\mathrm{map}(\mathrm{sk}_N(EG)_+, \tau_{\ge d+n}Y)$
is $n$-connective, so the composite
	\[
	\mathrm{map}(\mathrm{sk}_N(EG)_+, \tau_{\ge d+n}Y)
	\to
	\mathrm{map}(\mathrm{sk}_N(EG)_+, Y)
	\stackrel{r}{\to} Y
	\]
factors through $\tau_{\ge n}Y$. The map $\tau_{\ge d+n}Y \to
\tau_{\ge n}Y$ then factors through the diagonal
	\[
	\tau_{\ge d+n}Y
	\to
	\mathrm{map}(\mathrm{sk}_N(EG)_+, \tau_{\ge d+n}Y),
	\]
the target of which is $G$-nilpotent.
\end{proof}

\begin{lemma}\label{lem:X-almost-nilpotent} Let $X$ be an $S^1$-spectrum
and suppose we have a map of $S^1$-spectra
$f: X \to Y$, where $Y$ is $S^1$-nilpotent,
which induces an equivalence $\tau_{\ge m}X \simeq \tau_{\ge m}
Y$ for some $m\ge 0$. Then there is a $d\ge 0$ such that the map
$\tau_{\ge d}X \to X$ factors through an $S^1$-nilpotent spectrum.
\end{lemma}
\begin{proof} By the previous lemma there is a $d'\ge 0$
such that $\tau_{\ge d'+n}Y \to \tau_{\ge n}Y$ factors through an
$S^1$-nilpotent spectrum, for all integers $n$. Set
$d = d'+m$. Then $\tau_{\ge d}X \simeq \tau_{\ge d'+m}Y
\to \tau_{\ge m}Y \simeq \tau_{\ge m}X$ factors through an $S^1$-nilpotent
spectrum and hence so does the composite
$\tau_{\ge d}X \to \tau_{\ge m}X \to X$.
\end{proof}

\begin{lemma}\label{lem:abstract-canvan}
Let $X$ and $d$ be as in Lemma \ref{lem:X-almost-nilpotent}.
Then, for all $0\le k\le \infty$,
	\begin{enumerate}[{\rm (i)}]
	\item $(\tau_{\ge d}X)^{tC_{p^k}} \to X^{tC_{p^k}}$
	is nullhomotopic.
	\item $\tau_{\ge d}(X^{hC_{p^k}}) \to X^{tC_{p^k}}$
	is nullhomotopic.
	\item The map $X^{tC_{p^k}} \to (\tau_{<d}X)^{tC_{p^k}}$
	has a retract.
	\end{enumerate}
\end{lemma}
\begin{proof} 
The Tate construction $(-)^{tC_{p^k}}$ annihilates all
$S^1$-nilpotent spectra, so (i) is immediate from the previous lemma.
The map in (ii) factors as
	\[
	\tau_{\ge d}(X^{hC_{p^k}})
	\to
	(\tau_{\ge d}X)^{hC_{p^k}}
	\to
	(\tau_{\ge d}X)^{tC_{p^k}}
	\to X^{tC_{p^k}}
	\]
and so is nullhomotopic by (i). The claim (iii) follows from (i) and
the cofiber sequence
	\[
	(\tau_{\ge d}X)^{tC_{p^k}} \to
	X^{tC_{p^k}} \to (\tau_{<d}X)^{tC_{p^k}}
	\]
\end{proof}

\begin{proof}[Proof of Theorem \ref{thm:abstract-implications}]
Claim (a) is \cite[Proposition 2.25]{antieau-nikolaus}.

We now prove (b). By assumption, there is some $N$ for which
$p^N: X \to X$ is nullhomotopic as a map in
$\mathsf{CycSp}_p$. It follows that $p^N$ annihilates each
homotopy group $\pi_i\mathrm{TR}(X)$. It follows that each
object $\pi_i\mathrm{TR}(X)$ admits the structure of
a $\mathbb{Z}/p^N$-module in $S^1$-spectra, and hence that
each object $\pi_i\mathrm{TR}(X)$ is $\mathbb{F}_p$-nilpotent.
If $\mathrm{TR}(X)$
is bounded, then we conclude that $\mathrm{TR}(X)$ is also
$\mathbb{F}_p$-nilpotent, as an $S^1$-spectrum.

Now we prove (c), so we assume that
$\mathrm{TR}(X)$ is $\mathbb{F}_p$-nilpotent.
By Lemma \ref{lem:fp-nil-orbits}, we deduce that
$\mathrm{TR}(X)_{hC_p}$ is $\mathbb{F}_p$-nilpotent as well,
and, from the $S^1$-equivariant cofiber sequence
	\[
	\mathrm{TR}(X)_{hC_p} \stackrel{V}{\to} \mathrm{TR}(X)
	\to X
	\]
we deduce that $X$ is $\mathbb{F}_p$-nilpotent. By
Lemma \ref{lem:fp-nil-to-s1-nil}, we deduce that $X^{tC_p}$ is
$S^1$-nilpotent, which completes the proof of (c).

Claim (e) is immediate from Lemma \ref{lem:abstract-canvan}(ii).

For claim (d), first observe that (a), (b), (c),
and Lemma \ref{lem:abstract-canvan}(iii)
imply that $X^{tC_{p^k}}$ is a retract of $(\tau_{<d}X)^{tC_{p^k}}$.
The finiteness assumption on $X$ ensures that the homotopy groups
of $(\tau_{<d}X)^{tC_{p^k}}$ are finite and hence so are the
homotopy groups of $X^{tC_{p^k}}$. The homotopy groups of
$X$ being finite also implies that the homotopy groups of
$X_{hC_{p^k}}$ are finite, since $X$ was assumed bounded below.
The claim (d) now follows from the
cofiber sequence $X_{hC_{p^k}} \to X^{hC_{p^k}} \to
X^{tC_{p^k}}$.

We are left with establishing the claim (f), for which we argue
as in \cite{akhil-tr}. Recall that we have
pullback squares
	\[
	\xymatrix{
	\mathrm{TR}^k(X) \ar[r]^-R\ar[d] & \mathrm{TR}^{k-1}(X)\ar[d]\\
	X^{hC_{p^k}} \ar[r]_{\mathrm{can}} & X^{tC_{p^k}}
	}
	\]
The right vertical map is an equivalence
in large degrees, independent of $k$,
by Tsalidis's theorem \cite[II.4.9]{nikolaus-scholze}
and the assumption that 
\fbox{Segal Conjecture} holds.
The bottom horizontal map is zero in
in large degrees by the assumption that
\fbox{Weak Canonical Vanishing} holds. It follows that 
the top horizontal map is zero in large degrees,
and hence that the limit $\mathrm{TR}(X)$ is bounded above. Since $X$
was assumed bounded below, $\mathrm{TR}(X)$
is bounded below, and hence $\mathrm{TR}(X)$ is
bounded.
\end{proof}

\section{The Segal Conjecture} \label{sec:segal}

We fix throughout this section an $\mathbb{E}_3$-$\mathrm{MU}$-algebra form of $\mathrm{BP} \langle n \rangle$.
Our purpose is to prove the Segal
Conjecture (Theorem \ref{thm:intro-segal}), which
we restate here for convenience.

\begin{theorem}\label{thm:segal-conjecture-main}
Let $F$ denote any type $n+1$ finite complex.
Then the cyclotomic Frobenius $\mathrm{THH}(\mathrm{BP}\langle n \rangle) \to \mathrm{THH}(\mathrm{BP}\langle n \rangle)^{tC_p}$ induces an isomorphism
\[F_*\mathrm{THH}(\mathrm{BP}\langle n \rangle) \cong F_*(\mathrm{THH}(\mathrm{BP}\langle n \rangle)^{tC_p})\]
in all sufficiently large degrees $* \gg 0$.
\end{theorem}

\begin{remark} This theorem implies the corresponding statement
for $F$ a type $n+2$ complex, which is all that is used in deducing
the Lichtenbaum-Quillen statements as in the previous section.
\end{remark}

The idea of the proof is to use
(the d\'ecalage of) the Adams filtration on 
$\mathrm{BP}\langle n\rangle$ to reduce
the claim
to a much simpler one about graded polynomial algebras
over $\mathbb{F}_p$. 

\begin{remark} There are several antecedents to the
Segal conjecture. First, the classical 
Segal conjecture for the group $C_p$ states that the map
	\[
	S^0 = \mathrm{THH}(S^0) \to
	\mathrm{THH}(S^0)^{tC_p} = (S^0)^{tC_p}
	\]
is $p$-completion, and is a theorem of Lin \cite{lin} (at $p=2$) and
Gunawardena \cite{gunawardena} (for $p$ odd). For various classes
of ordinary commutative rings $R$, the map
	\[
	\varphi:
	\mathrm{THH}(R) \to \mathrm{THH}(R)^{tC_p}
	\]
is a $p$-adic equivalence in large degrees:
this is the case for DVRs of mixed characteristic
with perfect residue field in odd characteristic
\cite{hesselholt-madsen, hesselholt-madsen-derham},
for smooth algebras in positive characteristic
\cite[Prop. 6.6]{hesselholt},
and for $p$-torsionfree excellent noetherian
rings $R$ with $R/p$ finitely generated over its
$p$th powers \cite[Cor. 5.3]{akhil-tr}. 

When $R=\mathrm{\ell}$
is the Adams summand, it is proved in
\cite[Thm. 5.5]{ausoni-rognes-redshift} for $p\ge 5$ that
	\[
	\varphi: \mathrm{THH}(\ell)/(p,v_1)
	\to \mathrm{THH}(\ell)^{tC_p}/(p,v_1)
	\]
is an equivalence in degrees larger than $2(p-1)$ (cf. \cite{sverre-thesis}).
When $R=\mathrm{MU}$, Lun{\o}e-Nielsen and Rognes
show \cite{lunoenielsen-rognes} that
	\[
	\varphi: \mathrm{THH}(\mathrm{MU}) \to 
	\mathrm{THH}(\mathrm{MU})^{tC_p}
	\]
is a $p$-adic equivalence. In another direction, Angelini-Knoll
and Quigley \cite{angeliniknoll-quigley-segal}
have shown that $\varphi$
is an equivalence for Ravenel's $X(n)$ spectra.
\end{remark}

\subsection{Polynomial rings over the sphere} \label{sec:segal-over-sphere}

We begin by recording some facts about polynomial rings over
the sphere spectrum, starting with their construction.

\begin{construction}\label{build-poly}
For $r,w \in \mathbb{Z}$ we will denote
by $S^{2r}(w)$ the graded spectrum which is
$S^{2r}$ in weight $w$ and zero elsewhere.
Recall (see, e.g., \cite[3.4.1,3.4.2]{rotation})
that there is a graded $\mathbb{E}_2$-ring
$S^0[y_{-2, -1}]$ equipped with a class 
$y_{-2, -1}: S^{-2}(-1) \to S^0[y_{-2,-1}]$ which exhibits the target
as the free graded $\mathbb{E}_1$-algebra on $S^{-2}(-1)$. 
This graded $\mathbb{E}_2$-ring corresponds to an
$\mathbb{E}_2$-monoidal functor
	\[
	S^0[y_{-2,-1}]: \mathbb{Z}^{\mathrm{ds}} \to \mathsf{Sp}
	\]
which factors through the subcategory $\mathrm{Pic}(\mathsf{Sp})$
of invertible spectra. When $r\le 0$ we define the graded
$\mathbb{E}_2$-ring
$S^0[y_{2r, -1}]$ as the composite
	\[
	\mathbb{Z}^{\mathrm{ds}} \stackrel{\cdot r}{\to}
	\mathbb{Z}^{\mathrm{ds}} \stackrel{S^0[y_{-2,-1}]}{\longrightarrow}
	\mathrm{Pic}(\mathsf{Sp}) \to \mathsf{Sp}
	\]
When $r\ge 0$ we define the graded $\mathbb{E}_2$-ring
$S^0[y_{2r,-1}]$ as the composite
	\[
	\mathbb{Z}^{\mathrm{ds}} \stackrel{\cdot (-r)}{\to}
	\mathbb{Z}^{\mathrm{ds}} \stackrel{S^0[y_{-2,-1}]}{\longrightarrow}
	\mathrm{Pic}(\mathsf{Sp}) \stackrel{D}{\to}
	\mathrm{Pic}(\mathsf{Sp}) \to \mathsf{Sp}
	\]
where $D$ denotes the duality functor. Finally, 
we define $S^0[y_{2r,w}]$ for arbitrary $r, w \in \mathbb{Z}$
by left Kan extending $S^0[y_{2r,-1}]$ along the map
$(-w): \mathbb{Z}^{\mathrm{ds}} \to \mathbb{Z}^{\mathrm{ds}}$. 
Thus, for each $r,w \in \mathbb{Z}$, we have constructed
a graded $\mathbb{E}_2$-ring $S^0[a]$ equipped with
a class $a: S^{2r}(w) \to S^0[a]$ which exhibits the target
as the free graded $\mathbb{E}_1$-ring on $S^{2r}(w)$.
\end{construction}

Next we establish an important finiteness property for
$\mathrm{THH}(S^0[a])$. But first we recall a definition.

\begin{definition} If $G$ is a (topological) group,
we will say that a spectrum with $G$-action,
$X \in \mathsf{Fun}(\mathrm{B}G, \mathsf{Sp})$,
is \textbf{finite} if it lies in the thick subcategory generated
by the objects $G/H_+$ where $H \subseteq G$ is a closed
subgroup and $G/H_+$ denotes 
$\Sigma^{\infty}_+(G/H)$. 
\end{definition}

\begin{lemma} \label{lem:poly-over-sphere}
The graded $\mathbb{E}_1$-ring map
\[S^0[a] \to \mathrm{THH}(S^0[a])\]
induces on $\mathbb{F}_p$-homology the ring map
\[\mathbb{F}_p[a] \to \mathbb{F}_p[a]\otimes \Lambda_{\mathbb{F}_p}(\sigma a).\]
Here, the weights of $a$ and $\sigma a$ are both $w$.
Furthermore, if $w\ne 0$, then,
as a graded $C_p$-spectrum $\mathrm{THH}(S^0[a])$ is pointwise finite. That is: at each weight $j$, the $C_p$-spectrum
$\mathrm{THH}(S^0[a])_j$ lies in the thick subcategory generated
by $S^0$ and $C_{p+}$. 
\end{lemma}

\begin{proof}
We have that $\mathbb{F}_p\otimes \mathrm{THH}(-)=
\mathrm{THH}(\mathbb{F}_p \otimes(-)/\mathbb{F}_p)$, so the induced map on
homology is given by
	\[
	\mathbb{F}_p[a] \to \mathrm{THH}(\mathbb{F}_p[a]/\mathbb{F}_p)_*.
	\]
This map depends
only on the $\mathbb{E}_1$-algebra structure of
$\mathbb{F}_p[a]$, which is free, so this is equivalent to the
classical calculation of $\mathrm{THH}(\mathbb{F}_p[a]/\mathbb{F}_p)_*$
(i.e. ordinary Hochschild homology over $\mathbb{F}_p$).

We now show that $\mathrm{THH}(S^0[a])$ is pointwise finite as a graded $C_p$-spectrum. This statement only depends
on the graded $\mathbb{E}_1$-algebra structure on $S^0[a]$,
which is free. The Hochschild homology of free $\mathbb{E}_1$-algebras
is well-known (see, e.g., the argument in \cite[Theorem 3.8]{akhil-tr},
which applies verbatim to the graded case), and in this case specializes to:
	\[
	\mathrm{THH}(S^0[a])
	\simeq \bigoplus_{k \ge 0} \mathrm{Ind}_{C_k}^{S^1}(S^{2rk}(wk)).
	\]
Here $\mathrm{Ind}_{C_k}^{S^1}$ denotes the induction
functor, given by left Kan extension along the
functor $\mathrm{B}C_k \to \mathrm{B}S^1$,
and $C_k$ acts by permuting the factors in 
$(S^{2r}(w))^{\otimes k}=S^{2rk}(wk)$.
Observe that, since $w\ne 0$, there is at most one nonzero
summand in each fixed weight. To complete the proof
we need to show that each summand is finite as a $C_p$-spectrum.

The property of finiteness is always
preserved by induction.
In this case,
the restriction functor $\mathsf{Fun}(\mathrm{B}S^1, \mathsf{Sp})
\to \mathsf{Fun}(\mathrm{B}C_p, \mathsf{Sp})$ also preserves
finiteness. Indeed, when $k=mp$
the object $S^1/C_{k+}$ is equivalent
to $S^1_+ = \Sigma (C_p/C_p)_+$ as a $C_p$-spectrum,
and when $k$ is coprime to $p$ then 
we have a cofiber sequence
	\[
	C_{p+} \to S^1/C_{k+} \to \Sigma C_{p+}.
	\]
So it suffices to show that $(S^{2r})^{\otimes k}$
is finite as a $C_k$-spectrum. After possibly dualizing we may
assume that $r\ge 0$, and then this is the suspension 
spectrum of the one-point compactification of $2r$ copies of
the regular representation of $C_k$, which admits a finite
$C_p$-CW-structure.
\end{proof}

We now prove the Segal conjecture for these graded
polynomial rings over the sphere. For the
statement, recall that the cyclotomic Frobenius
on filtered objects multiplies filtrations by $p$,
and we review the formalism for this using the functor $L_p$
in \S\ref{ssec:hh-of-filtered}.

\begin{proposition}\label{prop:segal-for-sphere-poly} Suppose $w\ne 0$.
Then the cyclotomic Frobenius
	\[
	L_p\mathrm{THH}(S^0[a]) \to
	\mathrm{THH}(S^0[a])^{tC_p}
	\]
witnesses the target as the $p$-completion of the source.
\end{proposition}
\begin{proof} As in the previous proposition, we may compute:
	\[
	\mathrm{THH}(S^0[a])
	\simeq \bigoplus_{k \ge 0} \mathrm{Ind}_{C_k}^{S^1}(S^{2rk}(wk)).
	\]
Since there is at most one nonzero summand in each fixed weight,
taking the Tate fixed points (in the category of graded
spectra) commutes with this sum, so that we have
	\[
	\mathrm{THH}(S^0[a])^{tC_p}
	\simeq \bigoplus_{k \ge 0}
	 (\mathrm{Ind}_{C_k}^{S^1}(S^{2rk}(wk)))^{tC_p}.
	\]
When $k$ is not divisible by $p$, the restriction of 
$\mathrm{Ind}_{C_k}^{S^1}(S^{2rk})$ to a $C_p$-spectrum
lies in the thick subcategory generated by $C_{p+}$, and so is
annihilated by $(-)^{tC_p}$. When $k=mp$ is divisible by $p$,
then the restriction of $\mathrm{Ind}_{C_k}^{S^1}(S^{2rk})$
is equivalent to $S^1_+ \otimes (S^{2rm})^{\otimes p}$,
where $C_p$ acts trivially on the first term and by cyclic permutations
on the second. Thus:
	\[
	\mathrm{THH}(S^0[a])^{tC_p}
	\simeq \bigoplus_{m \ge 0}
	((S^{2rm})^{\otimes p}(wmp))^{tC_p} \oplus \Sigma 
	((S^{2rm})^{\otimes p}(wmp))^{tC_p}.
	\]
To compute what the cyclotomic Frobenius does, recall
that, directly from the construction of the cyclotomic Frobenius,
we have a commutative diagram, for any graded $\mathbb{E}_1$-ring
$A$:
	\[
	\xymatrix{
	L_pA^{\otimes m} \ar[r]\ar[d] & (A^{\otimes mp})^{tC_p}\ar[d]\\
	L_p\mathrm{THH}(A) \ar[r] & \mathrm{THH}(A)^{tC_p}
	}
	\] 
The bottom arrow is $S^1\cong S^1/C_p$-equivariant, so we may induce up
the targets of the vertical maps to get a diagram:
	\[
	\xymatrix{
	S^1_+ \otimes
	L_pA^{\otimes m} \ar[r]\ar[d] & 
	S^1_+\otimes (A^{\otimes mp})^{tC_p}\ar[d]\\
	L_p\mathrm{THH}(A) \ar[r] & \mathrm{THH}(A)^{tC_p}
	}
	\]
If we now take $A=S^0[a]$ and restrict to the summand
corresponding to $a^m$, then we learn that the cyclotomic
Frobenius map in weight $mp$ is given by tensoring the Tate diagonal
	\[
	S^{2rm} \to ((S^{2rm})^{\otimes p})^{tC_p}
	\]
with $S^1_+$. The
Tate diagonal here
witnesses the target as the $p$-completion of the source
by the classical Segal conjecture.
\end{proof}

\subsection{The Segal conjecture for polynomial $\mathbb{F}_p$-algebras} 

We consider in this section a graded $\mathbb{E}_2$-$\mathbb{F}_p$-algebra $R$, with homotopy groups a polynomial ring
\[\pi_*(R) \cong \mathbb{F}_p[a_1,a_2,\cdots,a_n].\]
Each $a_i$ will have non-negative even degree $|a_i|$ and 
positive weight $\mathrm{wt}(a_i)$, though we suppress the weights from the notation.  In fact, there is a unique ring $R$ with the above description:

\begin{proposition}\label{prop:poly-unique}
As a graded $\mathbb{E}_2$-$\mathbb{F}_p$-algebra, the ring $R$ above must be equivalent to
\[\mathbb{F}_p \otimes S^0[a_1] \otimes S^0[a_2] \otimes \cdots \otimes S^0[a_n],\]
where $S^0[a_i]$ is the ring constructed in
Construction \ref{build-poly}
with $2r=|a_i|$ and $w=\mathrm{wt}(a_i)$.
\end{proposition}

\begin{proof}
Let us denote
$\mathbb{F}_p \otimes S^0[a_1] \otimes S^0[a_2] \otimes \cdots \otimes S^0[a_n]$ by $A$. We first claim that $A$ has a 
graded $\mathbb{E}_2$-$\mathbb{F}_p$-algebra cell structure
with cells in even degrees. Indeed, this
algebra is canonically augmented over $\mathbb{F}_p$, so
we may apply 
\cite[Theorem 11.21,Theorem 13.7]{kupers-galatius-randall-williams}
which, together, show that, if 
	\[
	\mathbb{F}_p\otimes_{
	\mathbb{F}_p\otimes_A\mathbb{F}_p}\mathbb{F}_p
	\]
has homotopy groups in even degrees, then $A$ has a minimal
cell structure as a graded $\mathbb{E}_2$-$\mathbb{F}_p$-algebra
with cells in even degrees.\footnote{Here it is important that we are considering
$\mathbb{E}_2$-algebras:
the iterated bar construction is related to the $\mathbb{E}_k$-cotangent complex up to a shift by $k$. Since $k=2$, the property
of being concentrated in even degrees is insensitive to this shift.}
But the K\"unneth spectral sequence
computing these homotopy groups collapses at the $E_2$-page
as a divided power algebra on even degree classes, so the claim follows.

There is then no obstruction to constructing an $\mathbb{E}_2$-map
$A \to R$ sending $a_i$ to $a_i$, since the homotopy groups of $R$
are concentrated in even degrees. The result follows. 
\end{proof}

Our main theorem about this $\mathbb{E}_2$-$\mathbb{F}_p$-algebra $R$ is as follows:

\begin{proposition} \label{ModpSegal}
The cyclotomic Frobenius 
\[
L_p\mathrm{THH}(R) \to \mathrm{THH}(R)^{tC_p}\]
induces on homotopy groups the ring map
\[\mathbb{F}_p[x,a_1,a_2,\cdots,a_n] \otimes \Lambda(\sigma a_1,\sigma a_2, \cdots, \sigma a_n) \to \mathbb{F}_p[x^{\pm 1},a_1,a_2,\cdots,a_n] \otimes \Lambda(\sigma a_1,\sigma a_2, \cdots, \sigma a_n) \]
that inverts $x$.  Here, $x$ is in degree $2$ and weight $0$. The degree of $\sigma a_i$ is one more than the degree of $a_i$, and the weight of $\sigma a_i$ is the same as the weight of $a_i$. 
\end{proposition}

A version of the above is well-known in the case that all $a_i$ are in degree $0$ and weight zero,
so $R$ is a classical commutative ring (see, e.g.,
\cite[6.6]{hesselholt} for a much stronger result).  Our main observation is that the result extends to the case where not all $a_i$ are in degree $0$, in which case $R$ is not discrete.
Since an exterior algebra on classes of degree $|a_i|+1$ has no homotopy above degree $n+\sum |a_i|$, we obtain the following result:

\begin{corollary}[Segal conjecture for graded
polynomial $\mathbb{F}_p$-algebras] \label{cor:fp-poly-segal}
The map
\[\pi_*(\mathrm{THH}(R)/(a_1,\cdots,a_n)) \to \pi_*(\mathrm{THH}(R)^{tC_p}/(a_1,\cdots,a_n))\]
is an equivalence in degrees $*>n+\sum^{n}_{i=1} |a_i|$.
\end{corollary}

\begin{proof}[Proof of Propostion \ref{ModpSegal}]
For convenience we will omit the grading shear, $L_p$, from the
notation throughout.

By Proposition \ref{prop:poly-unique}, we may assume that $R$ is a tensor product of graded $\mathbb{E}_2$-rings
\[
R \simeq \mathbb{F}_p \otimes S^0[a_1] \otimes \cdots \otimes S^0[a_{n}].\]
Since $\mathrm{THH}$ is symmetric monoidal as a functor to cyclotomic spectra \cite[p.341]{nikolaus-scholze}, we may compute
\[
\mathrm{THH}(R) \simeq \mathrm{THH}(\mathbb{F}_p) \otimes \mathrm{THH}(S^0[a_1]) \otimes \cdots \otimes \mathrm{THH}(S^0[a_n])
\]
as a $C_p$-equivariant $\mathbb{E}_1$-ring spectrum.  We next compute the cyclotomic Frobenius on each component of the above tensor product.
It follows from B\"okstedt's unpublished computation of 
$\mathrm{THH}(\mathbb{F}_p)$ (see \cite[\S IV.4]{nikolaus-scholze}
for a modern reference) that the map
	\[
	\varphi: \mathrm{THH}(\mathbb{F}_p) \to
	\mathrm{THH}(\mathbb{F}_p)^{tC_p}
	\]
induces, on homotopy groups, the map
	\[
	\mathbb{F}_p[x] \to \mathbb{F}_p[x^{\pm 1}]
	\]
which inverts $x$. Here $x=\sigma^2v_0$ is in degree 2 and
weight zero. We have already seen (Proposition 
\ref{prop:segal-for-sphere-poly}) that each map
	\[
	\varphi:
	\mathrm{THH}(S^0[a_i]) \to \mathrm{THH}(S^0[a_i])^{tC_p}
	\]
is an equivalence after $p$-completion. It follows that the map
	\[
	\mathrm{THH}(\mathbb{F}_p) \otimes 
	\bigotimes_i \mathrm{THH}(S^0[a_i])
	\to
	\mathrm{THH}(\mathbb{F}_p)^{tC_p}
	 \otimes \bigotimes_i \mathrm{THH}(S^0[a_i])^{tC_p}
	\]
has the desired effect on homotopy groups. To finish the proof
we need to show that the lax monoidal structure map
	\[
	\mathrm{THH}(\mathbb{F}_p)^{tC_p}
	 \otimes 
	 \bigotimes_i \mathrm{THH}(S^0[a_i])^{tC_p}
	 \to
	\left(
	\mathrm{THH}(\mathbb{F}_p) \otimes 
	\bigotimes_i \mathrm{THH}(S^0[a_1]) 
	\right)^{tC_p}
	\]
is an equivalence. By Lemma \ref{lem:poly-over-sphere}
it suffices to prove the following more general statement:
if $X$ and $Y$
are nonnegatively graded $C_p$-spectra, and, for each weight
$j$ the $C_p$-spectrum $Y_j$ lies in the thick subcategory
generated by $C_{p+}$ and $S^0$, then the map
	\[
	\alpha: X^{tC_p} \otimes Y^{tC_p} \to (X\otimes Y)^{tC_p}
	\]
is an equivalence after $p$-completion.

The Tate construction on graded $C_p$-spectra is computed pointwise,
so we need to prove that 
	\[
	\bigoplus_{i+j=w} X_i^{tC_p} \otimes Y_j^{tC_p}
	\to \left(\bigoplus X_i\otimes Y_j\right)^{tC_p}
	\]
is a $p$-adic equivalence. Since $X$ and $Y$ are nonnegatively graded,
these sums are finite.
We are therefore reduced
to proving the analogous ungraded statement: that
$\alpha$ is a $p$-adic equivalence where $X$ and $Y$ are (ungraded)
$C_p$-spectra and where we assume
that $Y$ belongs to the thick subcategory generated by $S^0$
and $C_{p+}$. Since the Tate consruction is exact,
the 
category
of $C_p$-spectra $Y$ for which $\alpha$ an equivalence
after $p$-completion is
a thick subcategory. So we need only prove the claim for
$Y=S^0$ and $Y=C_{p+}$. When $Y=C_{p+}$ both sides
vanish. When $Y=S^0$ we may identify $Y^{tC_p}$
with the $p$-complete sphere,
$(S^0)^{\wedge}_p$, by the Segal conjecture,
so this map becomes the canonical one
	\[
	X^{tC_p} \otimes (S^0)^{\wedge}_p \to X^{tC_p},
	\]
which is indeed an equivalence after $p$-completion.
\end{proof}

\subsection{The Segal conjecture for $\mathrm{BP} \langle n \rangle$}

The key to the proof of Theorem \ref{thm:intro-segal} is the following:

\begin{theorem}\label{thm:segal-for-bpn-simpler}
The map of $\mathrm{BP}$-algebras
\[\mathrm{THH}(\mathrm{BP} \langle n \rangle) / (p,v_1,v_2,\cdots,v_n) \to \mathrm{THH}(\mathrm{BP} \langle n \rangle)^{tC_p}/(p,v_1,v_2,\cdots,v_n)\]
is an equivalence in large degrees.
Here we regard $p, v_1, ..., v_n$ as elements in the homotopy
of the right hand side via the ring map $\varphi$.
\end{theorem}

Before proving it, we need to recall a few things about
the Adams spectral sequence for $\mathrm{BP}\langle n\rangle$.

\begin{recollection}\label{recall-adams-bpn}
Recall the descent tower 
$\mathrm{desc}_{\mathbb{F}_p}^{\ge \bullet}(\mathrm{BP}\langle n\rangle)$
discussed in \S\ref{sec:sseq}.
We claim that the associated graded object has homotopy groups
given by:
	\[
	\pi_*(\mathrm{gr}(\mathrm{desc}^{\ge \bullet}_{\mathbb{F}_p}
	\mathrm{BP} \langle n \rangle))
	\simeq
	\mathbb{F}_p[v_0, v_1, ..., v_n],
	\]
where each $v_i$ lies in weight $2p^i-1$
(recall that the \emph{weight} of a class in
$E_2^{s,t}$ is $t$, see Convention \ref{conv:grading-sseq}).
Indeed, from the definition
of the descent tower, these homotopy groups agree with
	\[
	\mathrm{Ext}^{*,*}_{\mathcal{A}_*}(\mathbb{F}_p,
	\mathrm{H}_*(\mathrm{BP}\langle n\rangle;\mathbb{F}_p)).
	\]
Recall that the homology of $\mathrm{BP}\langle n\rangle$ as a comodule
is coextended from the quotient Hopf algebra
$\Lambda(\overline{\tau_0}, ..., \overline{\tau_n})$ (where we write $\tau_j$
for $\zeta_{j+1}$ at the prime 2)\footnote{Here we use the convention of \cite{milnor}
for the definition of $\zeta_j$.}:
	\[
	\mathrm{H}_*(\mathrm{BP}\langle n\rangle;\mathbb{F}_p)
	=
	\mathcal{A}_* \square_{\Lambda(\overline{\tau_0}, ..., 
	\overline{\tau_n})} \mathbb{F}_p
	\]
(this goes back to the construction of $\mathrm{BP}\langle n\rangle$,
see \cite[Proposition 1.7]{wilson}). By the change of rings isomorphism
\cite[A1.3.13]{ravenel}, we have
	\[
	\mathrm{Ext}^{*,*}_{\mathcal{A}_*}(\mathbb{F}_p,
	\mathrm{H}_*(\mathrm{BP}\langle n\rangle;\mathbb{F}_p))
	\cong
	\mathrm{Ext}^{*,*}_{\Lambda(\overline{\tau_0}, ..., 
	\overline{\tau_n})}(\mathbb{F}_p,
	\mathbb{F}_p)
	\cong
	\mathbb{F}_p[v_0, ..., v_n],
	\]
where the $v_i$ are represented by $[\overline{\tau_i}]$ in the cobar complex.
The classes are so-named because they detect the corresponding
classes in $\pi_*\mathrm{BP}\langle n\rangle$, and $v_0$ detects
$p$. For $i>0$
we denote by $\tilde{v}_i$ chosen lifts of each $v_i$
to elements in $\pi_*\mathrm{desc}_{\mathbb{F}_p}^{\ge 2p^i -1}
(\mathrm{BP}\langle n\rangle)$, and we observe that
$v_0$ admits a unique lift to an element in $\pi_*\mathrm{desc}_{\mathbb{F}_p}(\mathrm{BP}\langle n\rangle)$, which we denote
by the same symbol.
\end{recollection}

\begin{proof}[Proof of Theorem \ref{thm:segal-for-bpn-simpler}]
For convenience, in this proof we will suppress
the functor $L_p$ from the notation
when discussing the cyclotomic Frobenius
for filtered and graded objects.

First observe that we may reformulate this claim
as saying that the map
	\[
	\mathbb{F}_p\otimes_{\mathrm{BP}\langle n\rangle}
	\varphi:
	\mathbb{F}_p\otimes_{\mathrm{BP}\langle n\rangle}
	\mathrm{THH}(\mathrm{BP}\langle n\rangle)
	\to \mathbb{F}_p\otimes_{\mathrm{BP}\langle n\rangle}
	\mathrm{THH}(\mathrm{BP}\langle n\rangle)^{tC_p}
	\]
is an equivalence in large degrees, since
	\[
	\mathrm{BP}\langle n\rangle/(p, v_1, ..., v_n) \simeq
	\mathbb{F}_p.
	\]
	
To define this map we are using that
$\mathrm{THH}(\mathrm{BP}\langle n\rangle)$ is
an $\mathbb{E}_1$-$\mathrm{BP}\langle n\rangle$-algebra,
and the map $\varphi$ is an $\mathbb{E}_2$-algebra map,
and hence $\varphi$ in particular has the structure of a map
of modules over $\mathrm{BP}\langle n\rangle$
(where the module structure on the target is defined
using the map $\varphi$).

The $\mathbb{E}_2$-algebra $\mathrm{BP}\langle n\rangle$
has a refinement to a filtered $\mathbb{E}_2$-algebra
$\mathrm{desc}_{\mathbb{F}_p}^{\ge \bullet}(\mathrm{BP}\langle n\rangle)$,
and $\mathrm{desc}_{\mathbb{F}_p}^{\ge \bullet}(\mathbb{F}_p)=\mathbb{F}_p$ is a module over this algebra, where
the right hand side is the tower with $0$ in positive filtration
and $\mathbb{F}_p$ in non-positive filtration. Moreover,
$\mathrm{THH}$ inherits a filtration, and so we can ask whether
the map
	\[
	\mathbb{F}_p\otimes_{
	\mathrm{desc}_{\mathbb{F}_p}^{\ge \bullet}(\mathrm{BP}\langle n\rangle)}
	\mathrm{THH}(\mathrm{desc}_{\mathbb{F}_p}^{\ge \bullet}(\mathrm{BP}\langle n\rangle))
	\to \mathbb{F}_p\otimes_{\mathrm{desc}_{\mathbb{F}_p}^{\ge \bullet}(\mathrm{BP}\langle n\rangle)}
	\mathrm{THH}(\mathrm{desc}_{\mathbb{F}_p}^{\ge \bullet}(\mathrm{BP}\langle n\rangle))^{tC_p}
	\]
is an equivalence in large degrees on homotopy groups.

We would like to reduce this to a claim on the associated graded,
but in order to do so we need to know that the towers on
both sides are conditionally convergent.
By Proposition \ref{prop:filtered-thh-converges}, 
the towers
$\mathrm{THH}(\mathrm{desc}^{\ge \bullet}_{\mathbb{F}_p} \mathrm{BP} \langle n \rangle)$ and $\mathrm{THH}(\mathrm{desc}^{\ge \bullet}_{\mathbb{F}_p} \mathrm{BP} \langle n \rangle)^{tC_p}$ are conditionally convergent, after $v_0$-completion. 
Using the notation in Recollection \ref{recall-adams-bpn},
it suffices by a thick subcategory
argument to prove
that
	\[
	\mathrm{desc}^{\ge \bullet}_{\mathbb{F}_p}(\mathrm{BP}\langle n\rangle)/(v_0, \tilde{v_1}, ..., \tilde{v_n}) \simeq \mathbb{F}_p
	\]
as filtered modules over 
$\mathrm{desc}^{\ge \bullet}_{\mathbb{F}_p}(\mathrm{BP}\langle n\rangle)$.
Again, since
$\mathrm{desc}^{\ge \bullet}_{\mathbb{F}_p}(\mathrm{BP}\langle n\rangle)$
is conditionally convergent after $v_0$-completion, it suffices
to check this equivalence upon taking the associated graded, where it is
clear from Recollection \ref{recall-adams-bpn}.

We are now reduced to checking that the associated graded
of the map
	\[
	\mathbb{F}_p\otimes_{
	\mathrm{desc}_{\mathbb{F}_p}^{\ge \bullet}(\mathrm{BP}\langle n\rangle)}
	\mathrm{THH}(\mathrm{desc}_{\mathbb{F}_p}^{\ge \bullet}(\mathrm{BP}\langle n\rangle))
	\to \mathbb{F}_p\otimes_{\mathrm{desc}_{\mathbb{F}_p}^{\ge \bullet}(\mathrm{BP}\langle n\rangle)}
	\mathrm{THH}(\mathrm{desc}_{\mathbb{F}_p}^{\ge \bullet}(\mathrm{BP}\langle n\rangle))^{tC_p}
	\]
is an equivalence in high enough degrees.

Upon taking associated graded, we may,
by Recollection \ref{recall-adams-bpn}, identify this map with
\[\mathrm{THH}(\mathbb{F}_p[v_0,v_1,\cdots,v_n])/(v_0, ..., v_n)
\to \mathrm{THH}(\mathbb{F}_p[v_0,v_1,\cdots,v_n])^{tC_p}/
(v_0, ..., v_n),\]
and it follows from Corollary \ref{cor:fp-poly-segal} that this map is an equivalence in large degrees. This completes the proof.
\end{proof}

From a thick subcategory argument in 
$\mathrm{BP}$-modules, we then learn the following: 

\begin{corollary}\label{cor:segal-conjecture}
For any positive integers $i_0,i_1,\cdots,i_n$,
the map of $\mathrm{BP}$-algebras
\[\mathrm{THH}(\mathrm{BP} \langle n \rangle) / (p^{i_0},v_1^{i_1},v_2^{i_2},\cdots,v_n^{i_n}) \to \mathrm{THH}(\mathrm{BP} \langle n \rangle)^{tC_p}/(p^{i_0},v_1^{i_1},v_2^{i_2},\cdots,v_n^{i_n})\]
is an equivalence in large degrees.

In particular, if we let $S/I$ denote
a generalized Moore spectrum of the form
 $S^0/(p^{i_0},v_1^{i_1},\cdots,v_n^{i_n})$, then 
\[(S/I)_* \mathrm{THH}(\mathrm{BP} \langle n \rangle) \to (S/I)_* \mathrm{THH}(\mathrm{BP} \langle n \rangle)^{tC_p}\]
is an equivalence in large degrees.
\end{corollary}

The Segal conjecture (Theorem \ref{thm:intro-segal}) now follows by
a thick subcategory argument in spectra, since
any $S/I$ generates the thick subcategory of
type $n+1$ spectra.
 
\section{The Detection Theorem}\label{sec:detection}

Throughout this section, we will use $\mathrm{BP}\langle n \rangle$ to denote a fixed $\mathbb{E}_3$-$\mathrm{MU}$-algebra form $\mathrm{BP}\langle n \rangle$.  By $v_{n+1} \in \pi_{2p^{n+1}-2} \mathrm{MU}_{(p)}$ we will refer to a specific indecomposable generator, with:
\begin{itemize}
\item trivial mod $p$ Hurewicz image, and
\item the key property that the unit map $\mathrm{MU}_{(p)} \to \mathrm{BP} \langle n \rangle$ sends $v_{n+1}$ to $0$ in homotopy.
\end{itemize}
This last assumption ensures that $v_{n+1}$ admits a unique lift to an element in the homotopy of the fiber of the unit map $\mathrm{MU}_{(p)} \to \mathrm{BP} \langle n \rangle$.  Our main aim will be to prove
Theorem \ref{thm:intro-detection} from the introduction,
which we restate for convenience:

\begin{theorem}[Detection] \label{thm:body-detection}
There is an isomorphism of $\mathbb{Z}_{(p)}[v_1, ..., v_n]$-algebras
	\[
	\pi_*(\mathrm{THH}(\mathrm{BP}\langle n\rangle/\mathrm{MU})^{hS^1})
 	\cong \left(\pi_*\mathrm{THH}(\mathrm{BP}\langle n \rangle 
	/\mathrm{MU})\right)\llbracket t \rrbracket, 
	\]
where $|t|=-2$.
This isomorphism can be chosen such that, under the unit map
\[\pi_*(\mathrm{MU}_{(p)}^{hS^1}) \to \pi_*(\mathrm{THH}(\mathrm{BP}\langle n\rangle/\mathrm{MU})^{hS^1}),\] 
the canonical complex orientation maps to $t$ and $v_{n+1}$ is sent to $t(\sigma^2v_{n+1})$.
\end{theorem}

Before turning to the proof, we observe that the Detection Theorem
implies a weak form of redshift.

\begin{corollary}\label{cor:weak-redshift}
For each $0 \le m \le n+1$, $L_{K(m)}\mathrm{K}(\mathrm{BP}\langle n\rangle)\ne 0$.
In particular, $L_{K(n+1)}\mathrm{K}(\mathrm{BP}\langle n\rangle)\ne 0$.
\end{corollary}
\begin{proof} By \cite{bgt}, the cyclotomic trace map 
\[\mathrm{K}(-)\to \mathrm{TC}(-)\]
 is a lax
symmetric monoidal natural transformation. It follows that the trace
$\mathrm{K}(\mathrm{BP}\langle n\rangle) \to \mathrm{TC}(\mathrm{BP}\langle n\rangle)$
is a map of $\mathbb{E}_2$-rings. Recall that there is a canonical
map $\mathrm{TC}(-) \to \mathrm{THH}(-)^{hS^1}$, to negative cyclic homology.
Thus we have a sequence of $\mathbb{E}_2$-ring maps:
	\[
	\mathrm{K}(\mathrm{BP}\langle n\rangle)
	\to
	\mathrm{TC}(\mathrm{BP}\langle n\rangle)
	\to 
	\mathrm{THH}(\mathrm{BP}\langle n\rangle)^{hS^1}
	\to
	\mathrm{THH}(\mathrm{BP}\langle n\rangle/\mathrm{MU})^{hS^1}
	\]
and hence an $\mathbb{E}_2$-ring map
	\[
	L_{K(m)}\mathrm{K}(\mathrm{BP}\langle n\rangle) \to
	L_{K(m)}\mathrm{THH}(\mathrm{BP}\langle n\rangle/\mathrm{MU})^{hS^1}
	\]
for each height $m \le n+1$.
If the source of this map were zero, then the target would be zero as well, since
this is a map of rings. The relative negative cyclic homology
$\mathrm{THH}(\mathrm{BP}\langle n\rangle/\mathrm{MU})^{hS^1}$
has the structure of an $\mathrm{MU}$-module. It follows from
\cite[Theorem 1.9]{hovey-splitting} and \cite[Theorem 1.5.4]{hovey-vn}
that
	\[
	L_{K(m)}\mathrm{THH}(\mathrm{BP}\langle n\rangle/\mathrm{MU})^{hS^1}
	= (\mathrm{THH}(\mathrm{BP}\langle n\rangle/\mathrm{MU})^{hS^1}
	)[v_{m}^{-1}]^{\wedge}_{(p,v_1, ..., v_{m-1})}.
	\]
By Theorem \ref{thm:body-detection} and
Theorem \ref{thm:poly-thh}, this completion and localization can be
computed algebraically and the result is nonzero.
\end{proof}

\begin{remark} In the statement and proof of the theorem we have
used that the $S^1$-action on 
$\mathrm{THH}(\mathrm{BP}\langle n\rangle/\mathrm{MU})$
is compatible with the algebra structure. One way to see
this is to use the generality in which $\mathrm{THH}$
is defined. Recall that for any
symmetric monoidal category $\mathcal{C}$
with tensor product compatible with sifted colimits,
Hochschild homology gives a functor:
	\[
	\mathrm{HH}_{\mathcal{C}}: 
	\mathsf{Alg}_{\mathbb{E}_1}(\mathcal{C})
	\to \mathsf{Fun}(\mathrm{B}S^1, \mathcal{C}).
	\]
For a reference, one could observe that the construction
of $\mathrm{THH}$ with its circle action in 
\cite[\S III.2]{nikolaus-scholze} works just the same for
$\mathcal{C}$ in place of $\mathsf{Sp}$. Alternatively, one can use
the identification of $\mathrm{THH}$ with factorization
homology over $S^1$, which is defined in this generality
(\cite[\S 5.5.2]{ha}, \cite{ayala-francis}). Now apply this in the case
$\mathcal{C} = 
\mathsf{Alg}_{\mathbb{E}_2}(\mathsf{Mod}_{\mathrm{MU}})$
to see that $\mathrm{THH}(\mathrm{BP}\langle n\rangle/\mathrm{MU})$
has a canonical enhancement to an object in
$\mathsf{Fun}(\mathrm{B}S^1, 
\mathsf{Alg}_{\mathbb{E}_2}(\mathsf{Mod}_{\mathrm{MU}}))$.
\end{remark}

Let us now proceed with the proof of Theorem \ref{thm:body-detection}.

First we compute 
$\mathrm{THH}(\mathrm{BP}\langle n\rangle/\mathrm{MU})^{hS^1}$.
Recall that we computed the homotopy groups of
$\mathrm{THH}(\mathrm{BP}\langle n\rangle/\mathrm{MU})$ in
Theorem \ref{thm:poly-thh}.  An immediate consequence of that calculation is the following proposition:

\begin{proposition}\label{prop:gr-of-tc} The homotopy fixed point spectral sequence
for $\mathrm{THH}(\mathrm{BP}\langle n\rangle/\mathrm{MU})^{hS^1}$
collapses at the $E_2$-page, with 
	\[
	E_{\infty} = 
	\mathrm{THH}(\mathrm{BP}\langle n\rangle/\mathrm{MU})_*[t],
	\]
where $t \in H^2(\mathbb{C}P^{\infty})$ is the standard generator.
\end{proposition}
\begin{proof}
The homotopy fixed point spectral sequence
computing $\mathrm{THH}(\mathrm{BP}\langle n\rangle/\mathrm{MU})^{hS^1}$
is concentrated in even degrees by Theorem \ref{thm:poly-thh},
and hence collapses as indicated.
%The $E_{\infty}$-page is free as a $\mathbb{Z}_{(p)}[v_1, ..., v_n]$-algebra
%so there can be no $\mathbb{Z}_{(p)}[v_1, ..., v_n]$-algebra extensions.
\end{proof}

As we will shortly explain, the remainder of the argument for Theorem \ref{thm:body-detection}
is a formal consequence of the relationship between the
suspension map $\sigma^2$ and the circle action. We explore this relationship in \S \ref{sec:suspension}. 

\begin{proof}[Proof of \ref{thm:body-detection}]

The image of the canonical complex orientation under the unit map
\[\pi_* \mathrm{MU}_{(p)}^{hS^1} \to \pi_* \mathrm{THH}(\mathrm{BP}\langle n \rangle /\mathrm{MU})^{hS^1} \]
will be detected by $t$ in the homotopy fixed point spectral sequence.
We recall that $\pi_* \mathrm{THH}(\mathrm{BP}\langle n \rangle /\mathrm{MU})$ is a polynomial $\mathbb{Z}_{(p)}[v_1,\cdots,v_n]$-algebra generated by classes
$w_{n+1,i}$ (for $i \ge 0$), and $y_{j,i}$ (for $i \ge 0$, $j \ge 1$, and $j \notequiv -1$ modulo $p$).  Furthermore, we may set $w_{n+1,0}$ equal to $\sigma^2 v_{n+1}$.

Since polynomial algebras are free commutative algebras, an isomorphism 
\[\pi_*
\mathrm{THH}(\mathrm{BP}\langle n \rangle /\mathrm{MU})^{hS^1} \cong \left(\pi_*\mathrm{THH}(\mathrm{BP}\langle n \rangle /\mathrm{MU})\right)\llbracket t \rrbracket, \]
is determined by a choice of elements $\widetilde{w_{n+1,i}} , \widetilde{y_{j,i}} \in \pi_*
\mathrm{THH}(\mathrm{BP}\langle n \rangle /\mathrm{MU})^{hS^1}$ detecting the similarly named classes in the homotopy fixed point spectral sequence.

If we make any such choice, then Lemma \ref{lem:detection} ensures that $t \widetilde{w_{n+1,0}}$ will be $v_{n+1}$ modulo $t^2$,
say $t\widetilde{w_{n+1,0}} = v_{n+1} +t^2y$.
We may then replace $\widetilde{w_{n+1,0}}$ by
$\widetilde{w_{n+1,0}} - ty$, which also lifts the class $w_{n+1,0}$,
and so guarantee that
$t \widetilde{w_{n+1,0}} = v_{n+1}$.
\end{proof}

\begin{remark} Rognes has sketched an alternative
proof that $v_{n+1}$ is detected
in the homotopy of $\mathrm{THH}(\mathrm{BP}\langle n\rangle/\mathrm{MU})^{hS^1}$. 
The strategy is to consider the exact sequence in mod $p$
homology:
	\[
	H_*(\lim_{\mathbb{C}P^1}
	\mathrm{THH}(\mathrm{BP}\langle n\rangle/\mathrm{MU})) \to 
	H_*(\mathrm{THH}(\mathrm{BP}\langle n\rangle/\mathrm{MU}))
	\stackrel{B}{\to}
	H_{*+1}(\mathrm{THH}(\mathrm{BP}\langle n\rangle/\mathrm{MU})).
	\]
One computes that $B\overline{\tau}_{n+1}$ is nonzero and hence does not lie
in the kernel, i.e. does not arise in the first term.
It follows from an Adams spectral sequence argument
that $v_{n+1}$ must be detected in $\pi_*
\lim_{\mathbb{C}P^1}
	\mathrm{THH}(\mathrm{BP}\langle n\rangle/\mathrm{MU})$.
\end{remark}

%\subsection{Finishing the proof of the Detection Theorem}

%Now we can prove Theorem \ref{thm:redshift-main}. The ring $(\pi_*(A)/(p,v_1, ..., v_n))[v_{n+1}^{-1}]$
%is nonzero if and only if $v_{n+1}$ is nonnilpotent in $\pi_*(A)/(p, v_1, ..., v_n)$.
%We have seen (Proposition \ref{prop:gr-of-tc}) 
%that there is a filtration on this ring with associated graded ring
%	\[
%	\mathbb{F}_p[\gamma_{p^i}(\sigma^2v_{n+1}),
%	\gamma_{p^i}(\sigma^2 x_j): j\ne -1 \bmod p][t]
%	\]
%and such that $v_{n+1}$ maps to a unit times $t \sigma^2v_{n+1}$ in this associated graded
%(Proposition \ref{prop:extension}).
%It follows that $v_{n+1}^k$ is detected by the nonzero
%element $t^k(\sigma^2 v_{n+1})^k$ in the associated graded, and this
%completes the proof.

\section{Canonical Vanishing}\label{sec:canvan}

Fix an $\mathbb{E}_3$-$\mathrm{MU}$-algebra form of $\mathrm{BP} \langle n \rangle$. In this section, we will study the canonical map
\[\mathrm{can}: \mathrm{THH}(\mathrm{BP}\langle n \rangle)^{hS^1}
	\longrightarrow \mathrm{THH}(\mathrm{BP}\langle n \rangle)^{tS^1}.\]
Our goal will be to establish
Theorem \ref{thm:intro-can-van}, which by the results of Section \ref{sec:abstract-ql} can be reduced to weak canonical vanishing (Theorem \ref{thm:canvan-new}). % We will prove this weak canonical vanishing theorem as Theorem \ref{thm:canvan-new}.

To be specific, we will choose a convenient type $n+1$ complex $M$, with $v_{n+1}$ self map $v$, and consider the map
\[1 \otimes \mathrm{can}: M/v \otimes \mathrm{THH}(\mathrm{BP}\langle n \rangle)^{hS^1} \to M/v \otimes \mathrm{THH}(\mathrm{BP}\langle n \rangle)^{tS^1},\]
where $1$ is the identity map of the type $n+2$ complex $M/v$.  We will prove that the $\pi_*(1 \otimes \mathrm{can})$ map is zero for all sufficiently large degrees $*$.

As a prototype for the result and its proof, consider the case $n=-1$, where the statement is that
\[\mathrm{can}/p: \mathrm{THH}(\mathbb{F}_p)^{hS^1}/p
	\longrightarrow \mathrm{THH}(\mathbb{F}_p)^{tS^1}/p,\]
induces the zero map on homotopy groups in large degrees.

%In fact, the map
%	\[
%	\mathrm{can}: \mathrm{THH}(\mathbb{F}_p)^{hS^1}
%	\longrightarrow \mathrm{THH}(\mathbb{F}_p)^{tS^1}
%	\]
%is given on homotopy groups by
%	\[
%	\mathbb{Z}_p[\sigma^2v_0, t]/(t\sigma^2v_0 - p) \longrightarrow
%	\mathbb{Z}_p[t^{\pm 1}].
%	\]
%Since $|t|=-2$, any positive degree class in $\pi_*\mathrm{THH}(\mathbb{F}_p)^{hS^1}$ is a multiple of $\sigma^2 v_0$, and this class is sent under the canonical map to a multiple of $p$.

%We would like to understand a proof of mod $p$ canonical vanishing for $\mathrm{THH}(\mathbb{F}_p)$ that does not involve completely understanding the homotopy groups of $\mathrm{THH}(\mathbb{F}_p)^{hS^1}$

We may compute $\mathrm{THH}(\mathbb{F}_p)^{hS^1}$ via the homotopy fixed point spectral sequence 
\[E_2=\mathbb{F}_p[\sigma^2v_0, t] \implies \pi_*\mathrm{THH}(\mathbb{F}_p)^{hS^1}.\]
Here, $\sigma^2v_0$ is in homotopy dimension $2$ and filtration $0$, while $t$ is in homotopy dimension $-2$ and filtration $2$.

We may also understand $\mathrm{THH}(\mathbb{F}_p)^{tS^1}$ via the Tate fixed point spectral sequence, with $E_2$ page $\mathbb{F}_p[\sigma^2v_0,t^{\pm 1}]$.
The canonical map is compatible with homotopy and Tate fixed point spectral sequences, and at the level of $E_2$-pages it is approximated by the map
	\[
	\mathbb{F}_p[\sigma^2v_0, t] \longrightarrow
	\mathbb{F}_p[\sigma^2v_0, t^{\pm 1}]
	\]
that inverts $t$.

The element $p \in \pi_*\mathrm{THH}(\mathbb{F}_p)^{hS^1}$ is detected by $t\sigma^2v_0$, which lives in filtration $2$ in both spectral sequences. By killing a filtration $2$ lift of $p$, we build a map between a (modified) homotopy fixed point spectral sequence converging to $\pi_*\left(\mathrm{THH}(\mathbb{F}_p)^{hS^1}/p\right)$ and a (modified) Tate fixed point spectral sequence converging to $\pi_*\left(\mathrm{THH}(\mathbb{F}_p)^{tS^1}/p\right)$.  At the level of $E_2$ pages, the mod $p$ canonical map is approximated by
%	\[
%	\begin{tikzcd}
\[	E_2=\mathbb{F}_p[\sigma^2v_0, t]/(t\sigma^2v_0) \rightarrow
	E_2=\mathbb{F}_p[t^{\pm 1}] \]
%	\end{tizkcd}
%	\]
%This map vanishes in positive degrees by inspection, but there
%is an even more basic reason for that vanishing: 
This map of $E_2$ pages is trivial in positive homotopy dimension, and we would like to conclude that the mod $p$ canonical map is zero in positive degrees.  We might be worried about filtration jumps, but in fact this is no issue.  The source spectral sequence is concentrated in nonnegative filtration, while the target spectral sequence, in positive homotopy dimension, is concentrated in negative filtration.

Our strategy for proving Theorem \ref{thm:canvan-new} is to
mimic the above argument at a general height.
The main challenge in carrying this out (especially
in the absence of Smith-Toda complexes)
is to find and name an appropriate class in the homotopy
fixed point spectral sequence for $M \otimes \mathrm{THH}(\mathrm{BP}\langle n\rangle)^{hS^1}$ that detects the $v_{n+1}$ self map $v$.
We address this issue by descending information from
$\mathrm{THH}(\mathrm{BP}\langle n\rangle/\mathrm{MU})$,
which we understand well thanks to the previous section.

\subsection{Descent}

We will need to know that $\mathrm{THH}(\mathrm{BP}\langle n\rangle)$
is well approximated by 
$\mathrm{THH}(\mathrm{BP}\langle n\rangle/\mathrm{MU})$
in a way made precise in the below proposition.
We will use notation as in \S\ref{ssec:descent-sseq}
(also note Remark \ref{rmk:descent-weaker}).

\begin{proposition}\label{prop:horiz-vanish} 
For any type $(n+1)$-complex 
$F$, the spectral sequence computing\\
$\pi_*(F \otimes \mathrm{THH}(\mathrm{BP}\langle n\rangle))$
by descent along the map
	\[
	\mathrm{THH}(\mathrm{BP}\langle n\rangle)
	\to
	\mathrm{THH}(\mathrm{BP}\langle n\rangle/\mathrm{MU})
	\]
collapses at a finite page with a horizontal vanishing line.
In particular, if $F$ is equipped with a
homotopy ring structure, then the kernel
of the map 
	\[
	\pi_*(F\otimes \mathrm{THH}(\mathrm{BP}\langle n\rangle))
	\to
	\pi_*(F \otimes
	\mathrm{THH}(\mathrm{BP}\langle n\rangle/\mathrm{MU}))
	\]
is nilpotent.
\end{proposition}

\begin{remark} It is possible to obtain much stronger results
about horizontal vanishing lines in these and
related descent spectral
sequences by further developing
the methods used below. We hope to
return to this in future work.
\end{remark}

In the proof we will use Hochschild homology with
coefficients in a bimodule, which we now recall.

\begin{definition} If $M$ is a bimodule over an $\mathbb{E}_1$-algebra $A$
in $\mathcal{C}$, then we define
	\[
	\mathrm{HH}(A; M):= M\otimes_{A\otimes A^{op}} A.
	\]
\end{definition}
\begin{remark}
(Compare \cite[\S 2]{angeltveit-hill-lawson}.)
If $A$ admits the structure of an $\mathbb{E}_2$-algebra, so that
$\mathrm{HH}(A)$ has the structure of a module over $A$, 
and $M$ is a right $A$-module viewed as an $A$-bimodule by restriction
along $A\otimes A^{op} \to A$, then we have a canonical equivalence
	\[
	\mathrm{HH}(A; M)= M\otimes_{A\otimes A^{op}} A
	\simeq
	M\otimes_A (A\otimes_{A\otimes A^{op}}A)
	\simeq M\otimes_A \mathrm{HH}(A).
	\]
If, moreover, the bimodule structure on $M$ arises from an
$\mathbb{E}_1$-algebra map $A \to M$, then we have an equivalence
	\[
	\mathrm{HH}(A; M)
	=M \otimes_{A \otimes A^{op}} A
	\simeq M \otimes_{M \otimes A^{op}}M \otimes A^{op}
	\otimes_{A \otimes A^{op}}A
	\simeq M \otimes_{M \otimes A^{op}}M.
	\]
\end{remark}

\begin{construction} By the previous remark we have an equivalence
	\[
	\mathrm{THH}(\mathrm{BP}\langle n\rangle;\mathbb{F}_p)
	\simeq
	\mathbb{F}_p \otimes_{\mathbb{F}_p \otimes 
	\mathrm{BP}\langle n\rangle} \mathbb{F}_p.
	\]
Recall that 
	\[
	\pi_*(\mathbb{F}_p \otimes \mathrm{BP}\langle n\rangle)
	\simeq
	\Lambda(\overline{\tau}_i : i\ge n+1) \otimes_{\mathbb{F}_p}
	\mathbb{F}_p[t_1, ..., t_{n+1}]
	\]
where the $t_i$ come from the homology of $\mathrm{BP}$. 
Thus we have well-defined elements
$\sigma \overline{\tau}_{n+1}$ and $\sigma t_1, ..., \sigma t_{n+1}$
in $\pi_*\mathrm{THH}(\mathrm{BP}\langle n\rangle;\mathbb{F}_p)$.
We will write $\sigma\overline{\tau}_{n+1}$ as $\sigma^2v_{n+1}$
since this is its image inside $\mathrm{THH}(\mathrm{BP}\langle n\rangle/
\mathrm{MU}; \mathbb{F}_p)$. 
\end{construction}

\begin{proposition}\label{prop:descent-fp} The descent spectral sequence for
	\[
	\mathrm{THH}(\mathrm{BP}\langle n\rangle;\mathbb{F}_p)
	\to
	\mathrm{THH}(\mathrm{BP}\langle n\rangle/\mathrm{MU};
	\mathbb{F}_p)
	\]
collapses at the $E_2$-page as
	\[
	E_2 = \mathbb{F}_2[\sigma^2v_{n+1}]
	\otimes \Lambda(\sigma t_1, ..., \sigma t_{n+1})
	\]
Here $\sigma^2v_{n+1}$ has filtration 0 and homotopy dimension
$2p^{n+1}$, and each $\sigma t_i$ has
filtration 1 and homotopy dimension $2p^i-1$. 
\end{proposition}

\begin{proof}[Proof of Proposition \ref{prop:horiz-vanish}
from Proposition \ref{prop:descent-fp}]

By a thick subcategory argument
(using \cite{hopkins-palmieri-smith})
it suffices to establish the claim
for $F = S^0/(p^{i_0}, ..., v_n^{i_n})$ a generalized Moore complex.
Observe that
	\begin{align*}
	F \otimes \mathrm{THH}(\mathrm{BP}\langle n\rangle)
	&\simeq
	(F\otimes \mathrm{BP}\langle n\rangle)
	\otimes_{\mathrm{BP}\langle n\rangle}
	\mathrm{THH}(\mathrm{BP}\langle n\rangle).
	\end{align*}
The $\mathrm{BP}\langle n\rangle$-module
$F\otimes \mathrm{BP}\langle n\rangle$
lies in the thick subcategory generated by
the $\mathrm{BP}\langle n\rangle$-module $\mathbb{F}_p$,
so
we are reduced to the statement
in Proposition \ref{prop:descent-fp}.
\end{proof}

\begin{proof}[Proof of Proposition
\ref{prop:descent-fp}] For this proof we will abbreviate
	\[
	A:= \mathrm{THH}(\mathrm{BP}\langle n\rangle;\mathbb{F}_p),
	\, B:= \mathrm{THH}(\mathrm{BP}\langle n\rangle/\mathrm{MU};
	\mathbb{F}_p).
	\]
It follows
from \cite[Theorem 5.12]{angeltveit-rognes} that
	\[
	\pi_*A = \mathbb{F}_p[\sigma^2v_{n+1}]
	\otimes \Lambda(\sigma t_1, ..., \sigma t_{n+1}).
	\]
We will see shortly that $\Sigma:=\pi_*(B\otimes_AB)$ is flat over
$\pi_*B$. The proposition will follow if we can show that
$\mathrm{Ext}^*_{\Sigma}(\pi_*B, \pi_*B)$ already has the
correct size. We will prove this by constructing a further spectral
sequence computing $\mathrm{Ext}^*_{\Sigma}(\pi_*B, \pi_*B)$
whose $E_2$-term has the same size as $\pi_*A$.

Regard $A$ and $B$ as filtered algebras via the Whitehead filtrations
$\{\tau_{\ge j}A\}$ and $\{\tau_{\ge j}B\}$, so that the map
$A \to B$ is a map of filtered algebras. We may then regard
the cosimplicial object
	\[
	[n] \mapsto B^{\otimes_A (n+1)}
	\]
as a cosimplicial filtered spectrum. The associated graded
cosimplicial object is then given by
	\[
	[n] \mapsto (\pi_*B)^{\otimes_{\pi_*A}^{\mathbb{L}} (n+1)}
	\]
(where we have used $\otimes^{\mathbb{L}}$ to remind the reader
that the tensor products are derived). Since $\pi_*B$ is concentrated in even degrees, the exterior classes
must vanish under the map $\pi_*A \to \pi_*B$. It follows that
	\[
	\overline{\Sigma}:=
	\pi_*\left(\pi_*B \otimes^{\mathbb{L}}_{\pi_*A} \pi_*B\right)
	\simeq
	\pi_*B \otimes_{\mathbb{F}_p} P \otimes_{\mathbb{F}_p} \Gamma
	\]
where $P$ is a polynomial algebra on even degree classes
and 
$\Gamma = \Gamma\{\sigma^2t_1, ..., \sigma^2t_{n+1}\}$ is
a divided power algebra on the indicated generators.

Since $\overline{\Sigma}$ is concentrated in even degrees, 
we learn that each of the
spectral sequences
	\[
	\pi_*\left((\pi_*B)^{\otimes_{\pi_*A}^{\mathbb{L}} (n+1)}\right)
	\Rightarrow \pi_*(B^{\otimes_A(n+1)})
	\]
collapses at the $E_2$-page. In other words: we have a filtration on
each group $\pi_*(B^{\otimes_A(n+1)})$ whose associated graded
is given by $\pi_*\left((\pi_*B)^{\otimes_{\pi_*A}^{\mathbb{L}} (n+1)}\right)$.
This, in particular, implies that $\Sigma$ is flat over $\pi_*B$ as we claimed
earlier.

Using this filtration on homotopy groups, we may then extract
a spectral sequence (\cite[Theorem A.1.3.9]{ravenel}):
	\[
	\mathrm{Ext}^*_{\overline{\Sigma}}(\pi_*B, \pi_*B)
	\Rightarrow
	\mathrm{Ext}^*_{\Sigma}(\pi_*B, \pi_*B).
	\]
It will now suffice to prove that
$\mathrm{Ext}^*_{\overline{\Sigma}}(\pi_*B, \pi_*B)$ has the same
size as $\pi_*A$. The map $\pi_*A \to \pi_*B$ can be written as a tensor product
(over $\mathbb{F}_p$)
of the three maps
	\[
	\mathbb{F}_p[\sigma^2v_{n+1}] \stackrel{\mathrm{id}}{\to}
	\mathbb{F}_p[\sigma^2v_{n+1}],\,\,
	\mathbb{F}_p \to P,\,\,
	\Lambda(\sigma t_1, ..., \sigma t_{n+1}) \to \mathbb{F}_p
	\]
The descent Hopf algebroid for the first map is just the pair
$(\mathbb{F}_p[\sigma^2v_{n+1}], \mathbb{F}_p[\sigma^2v_{n+1}])$
which has cohomology $\mathbb{F}_p[\sigma^2v_{n+1}]$,
concentrated in cohomological dimension zero.
The descent Hopf algebroid for the second map
is $(P, P\otimes_{\mathbb{F}_p} P)$, which has cohomology
$\mathbb{F}_p$ concentrated in cohomological dimension zero.
The descent Hopf algebroid for the last map is the divided power
Hopf algebra
$(\mathbb{F}_p, \Gamma\{\sigma^2t_1, ..., \sigma^2t_{n+1}\})$.

It follows that we may compute our Ext as
	\[
	\mathrm{Ext}^*_{\overline{\Sigma}}(\pi_*B, \pi_*B)
	\simeq
	\mathbb{F}_p[\sigma^2v_{n+1}] \otimes
	\mathrm{Ext}^*_{\Gamma\{\sigma^2t_1, ..., \sigma^2t_{n+1}\}}(\mathbb{F}_p,
	\mathbb{F}_p)
	\simeq \mathbb{F}_p[\sigma^2 v_{n+1}] \otimes
	\Lambda(\sigma t_1, ..., \sigma t_{n+1})
	\]
which completes the proof.
\end{proof}

\subsection{Recollection on Hopkins-Smith}

It will be convenient for our argument to use a type $n+1$
complex with a $v_{n+1}$-element that has as high an
Adams filtration as possible. We do not know whether
it is possible to do this and also equip our complex
with a homotopy commutative ring structure, but
the below proposition will suffice for our purposes.

\begin{proposition}\label{recall:smiths} There is a finite $p$-local
$\mathbb{E}_1$-ring
spectrum $M$ with the following properties:
	\begin{enumerate}[{\rm (i)}]
	\item $M$ admits a non-nilpotent $v_{n+1}$-element, $v \in \pi_*M$.
	\item The element $v$ is central.
	\item $\mathrm{BP}\langle n\rangle \otimes M$ splits,
	as a $\mathrm{BP}\langle n\rangle$-module, as a direct
	sum of suspensions of $\mathbb{F}_p$.
	\item Let $\mathrm{fil}(v)$ denote the Adams filtration of $v$,
	and $|v|$ the dimension.
	Then $\dfrac{|v|}{\mathrm{fil}(v)}= 2p^{n+1}-2$.
	\item The map $\mathrm{MU}_*(M) \to
	\mathrm{BP}\langle n\rangle_*(M)$ is surjective.
	\end{enumerate}
\end{proposition}
\begin{proof}
We may take $M=\mathrm{End}(X)$ where $X$ is
the type $(n+1)$ spectrum constructed by Jeff Smith in \cite[\S 6.4]{smith}. 
The claims (i), (ii), and (iv) are shown in the course of the proof
of \cite[Theorem 4.12]{hopkins-smith}.
Since the Margolis homology of $H^*(X;\mathbb{F}_p)$
vanishes with respect to each
$Q_i$ with $i\le n$, \cite[Proposition 2.7]{miller-wilkerson}
shows that $H^*(X;\mathbb{F}_p)$ is a finitely generated
free module over $\Lambda(Q_0, ..., Q_n)$, so the same
is true of $M$. Choosing a basis of $H^*(M;\mathbb{F}_p)$
as a $\Lambda(Q_0, ..., Q_n)$-module gives a map
$M \to V$, into a direct sum of suspensions of $\mathbb{F}_p$.
After extending scalars of
the source to $\mathrm{BP}\langle n\rangle$ this becomes
an equivalence on $\mathbb{F}_p$-cohomology, and hence an equivalence,
proving (iii).

We now turn to the proof of (v). It will suffice to prove
the statement for $\mathrm{BP}$ in place of $\mathrm{MU}$.
We will use descent along $\mathrm{BP} \to \mathrm{BP}\langle n\rangle$
to study the $\mathrm{BP}$-modules
$\mathrm{BP} \otimes M$ and $\mathrm{BP}/(p, ..., v_n) \otimes M$.
We have a map between descent spectral sequences:
	\[
	\xymatrix{
	\mathrm{Ext}_{\Lambda_{\mathrm{BP}\langle n\rangle_*}(
	\sigma v_{n+1}, ..)}
	(\mathrm{BP}\langle n\rangle_*, 
	\mathrm{BP}\langle n\rangle_*(M)) 
	\ar@{=>}[r] \ar[d] &
	\mathrm{BP}_*(M) \ar[d]\\
	\mathrm{Ext}_{\Lambda_{\mathrm{BP}\langle n\rangle_*}(
	\sigma v_{n+1}, ..)}
	(\mathrm{BP}\langle n\rangle, 
	H_*(M;\mathbb{F}_p)) \ar@{=>}[r] & 
	(\mathrm{BP}/(p, ..., v_n))_*(M)
	}
	\]
Here we have used that 
$\mathrm{BP}\langle n\rangle \otimes_{\mathrm{BP}}
\mathrm{BP}/(p, ..., v_n) \simeq \mathbb{F}_p$.

Since $\mathrm{BP}\langle n\rangle \otimes M$ and
$\mathbb{F}_p \otimes M$ are $\mathbb{F}_p$-modules, by (iii), we may rewrite the map of Ext groups as
	\[
	\mathrm{Ext}_{\Lambda_{\mathbb{F}_p}(\overline{\tau}_{n+1},...)}
	(\mathbb{F}_p, \mathrm{BP}\langle n\rangle_*(M))
	\to
	\mathrm{Ext}_{\Lambda_{\mathbb{F}_p}(\overline{\tau}_{n+1}, ...)}
	(\mathbb{F}_p, H_*(M;\mathbb{F}_p)).
	\]
As argued above, the map $\mathrm{BP}\langle n\rangle \otimes M
\to \mathbb{F}_p \otimes M$ has a retract, so the map
$\mathrm{BP}\langle n\rangle_*(M) \to
H_*(M;\mathbb{F}_p)$ is injective. Finally, 
$H_*(M;\mathbb{F}_p)$ is a trivial comodule over
the Hopf algebra
$\Lambda(\overline{\tau}_{n+1}, ...)
\simeq \Lambda(\sigma v_{n+1}, ...)$,
and hence so too is
$\mathrm{BP}\langle n\rangle_*(M)$. It follows that the map
on $E_2$-terms above is an injection, and that every class in $\mathrm{BP}\langle n \rangle_*(M)$ has a representative on the 0-line of the spectral sequence computing $\mathrm{BP}_*(M)$. It remains to show that these representative classes survive to the $E_\infty$-page. By the above injectivity, it will suffice
to prove that the spectral sequence
	\[
	\xymatrix{
	\mathrm{Ext}_{\Lambda_{\mathbb{F}_p}(\overline{\tau}_{n+1}, ...)}
	(\mathbb{F}_p, 
	H_*(M;\mathbb{F}_p)) \ar@{=>}[r] & 
	(\mathrm{BP}/(p, ..., v_n))_*(M)
	}
	\]
collapses at the $E_2$-page. 

Observe that the property of a descent
spectral sequence collapsing at the $E_2$-page
is closed under direct sums, suspensions, and retracts. Since
the descent spectral sequence for $\mathrm{BP}/(p, ..., v_n)$
collapses at the $E_2$-page, it will suffice to prove that
$\mathrm{BP}/(p, ..., v_n) \otimes M$ is a direct summand
of a finite direct sum of suspensions of
$\mathrm{BP}/(p, ..., v_n)$.

Recall \cite[Lemma 6.2.6, Theorem C.3.2]{smith} that Smith's complex $X$ is obtained
as a summand of a tensor power of
a finite complex $Y$ with cells in dimensions $2$ through
$2p^{n+1}$. It follows that 
$H_*(Y;\mathbb{F}_p)$ is a trivial
comodule over $\Lambda(\overline{\tau}_{n+1}, ...)$
and that, for dimension reasons, the Adams spectral sequence
	\[
	\mathrm{Ext}_{\Lambda(\overline{\tau}_{n+1}, ...)}
	(\mathbb{F}_p, H_*(Y;\mathbb{F}_p))
	\simeq
	H_*(Y;\mathbb{F}_p)[v_{n+1}, ...]
	\Rightarrow
	(\mathrm{BP}/(p, ..., v_n))_*(Y)
	\]
collapses at the $E_2$-page. The $E_2$-page is
a finite free module over $\mathbb{F}_p[v_{n+1}, ...]$.
Using any homotopy ring structure on
$\mathrm{BP}/(p, ...,v_n)$ as a $\mathrm{BP}$-module,
we may then lift a basis to construct an equivalence
between $\mathrm{BP}/(p, ..., v_n) \otimes Y$ and
a finite direct sum of suspensions of
$\mathrm{BP}/(p, ..., v_n)$.

Similarly, using any homotopy ring structure 
on $\mathrm{BP}/(p, ..., v_n)$, we deduce that both
\[ \mathrm{BP}/(p, ..., v_n) \otimes (Y^{\otimes j}),
\,\,\text{and }
\mathrm{BP}/(p, ..., v_n) \otimes ((\mathbb{D}Y)^{\otimes j})\]
are equivalent to finite direct sums of suspensions of
$\mathrm{BP}/(p, ..., v_n)$, as $\mathrm{BP}$-modules.
So we conclude that
$\mathrm{BP}/(p, ..., v_n) \otimes M$ is a summand of
a finite direct sum of suspensions of
$\mathrm{BP}/(p, ..., v_n)$. This completes the proof.
\end{proof}

We will, several times, use the following elementary lemma found in
\cite{hopkins-smith}, which we recall for the reader's convenience.

\begin{lemma}[{\cite[Lemma 3.4]{hopkins-smith}}] \label{lem:take-powers-equal}
Suppose
that $x$ and $y$ are commuting elements of
a $\mathbb{Z}_{(p)}$-algebra. If $x-y$ is both torsion and nilpotent, then
for $N\gg 0$, 
	\[
	x^{p^N} = y^{p^N}.
	\]
\end{lemma}
\begin{proof} Expand
$(y+(x-y))^{p^N}$ and use that
$p^k(x-y) = 0$ for some $k$.
\end{proof}

From Proposition \ref{recall:smiths}(iii), we see that $\mathrm{THH}(\mathrm{BP}\langle n\rangle;\mathbb{F}_p)$
is a summand of
	\[
	M \otimes \mathrm{THH}(\mathrm{BP}\langle n\rangle)
	\simeq
	(M\otimes \mathrm{BP}\langle n\rangle)
	\otimes_{\mathrm{BP}\langle n\rangle}
	\mathrm{THH}(\mathrm{BP}\langle n\rangle),
	\]
arising from the unit map. In particular, there is a class
which lifts $\sigma^2v_{n+1}$ from
$M \otimes \mathrm{THH}(\mathrm{BP}\langle n\rangle/\mathrm{MU})$. 
We will need the following result ensuring the uniqueness and
centrality of such lifts, up to taking large powers.

\begin{lemma}\label{lem:nice-power} If $x \in \pi_*(M\otimes 
\mathrm{THH}(\mathrm{BP}\langle n\rangle))$ is a lift of
a power of
$\sigma^2v_{n+1} \in \pi_*(M\otimes \mathrm{THH}(\mathrm{BP}\langle n\rangle/
\mathrm{MU}))$, then there is some $k\ge 0$ for which
$x^{p^k}$ is central. Moreover, if $y$ is another such lift,
then there are $j,j' \ge 0$ such that $x^{p^j} = y^{p^{j'}}$
and both elements are central.
\end{lemma}
\begin{proof} Let $F = \mathrm{End}(M)$ and denote by $L_x$ and
$R_x$ the elements in homotopy corresponding to left and right
multiplication by $x$, respectively. These elements commute and
their difference is nilpotent by Proposition \ref{prop:descent-fp}.
It follows from Lemma \ref{lem:take-powers-equal} that
$L_x^{p^k} = R_x^{p^k}$ for some $k\ge 0$, and hence that
$x^{p^k}$ is central. For the second claim, first replace
$x$ and $y$ by $x^{p^k}$ and $y^{p^{k'}}$ so that $x^{p^k}$
is central and both elements map to the same power of
$\sigma^2v_{n+1}$ inside $\pi_*(M \otimes 
\mathrm{THH}(\mathrm{BP}\langle n\rangle/
\mathrm{MU}))$. Then $x$ and $y$ are commuting elements
and $x-y$ maps to zero in
$\pi_*(M\otimes \mathrm{THH}(\mathrm{BP}\langle 
n\rangle/\mathrm{MU}))$.
By Proposition \ref{prop:descent-fp}, $x-y$ is nilpotent
and again Lemma \ref{lem:take-powers-equal} implies that
$x^{p^j} = y^{p^j}$ for some $j\ge 0$. This completes the proof.
\end{proof}

\subsection{Proof of canonical vanishing}

\begin{theorem}\label{thm:canvan-new}
There is a $v_{n+1}$-element $v \in \pi_*M$
and an integer $d\ge 0$ such that, for all
$0\le k \le \infty$, the map
	\[
	\pi_*(M/v \otimes \mathrm{can}):
	\pi_*(M/v \otimes \mathrm{THH}(\mathrm{BP}\langle n\rangle)^{hC_{p^k}})
	\to
	\pi_*(M/v \otimes \mathrm{THH}(\mathrm{BP}\langle n\rangle)^{tC_{p^k}})
	\]
	is zero when $*\ge d$. 
\end{theorem}

\begin{remark} In the proof of the theorem and the lemmas below
we will make use of the homotopy fixed point spectral sequence.
If a group $G$ acts on a spectrum $X$, then we will take
the homotopy fixed point spectral sequence computing
$\pi_*(X^{hG})$ to be the one associated to the tower
$\{(\tau_{\ge j}X)^{hG}\}$ according to our conventions in
\S\ref{sec:sseq}. However, it will be convenient to know that,
for fixed $s$,
an element $x \in \pi_*(X^{hG})$
is detected by a class in $E_2^{s',t}$ for some $s'>s$
if and only if $x$ vanishes when restricted to
$\mathrm{map}(\mathrm{sk}_{s}(EG)_+, X)^{hG}$.
This follows from \cite[Theorem B.8]{greenlees-may}.
\end{remark}

\begin{lemma} \label{lemma:high-Adams}
Let $v \in \pi_*M$ be the $v_{n+1}$-element
from Proposition \ref{recall:smiths}.
Then $v$ is detected
in the homotopy fixed point spectral sequence for
$M\otimes \mathrm{THH}(\mathrm{BP}\langle n\rangle)^{hS^1}$
in filtration at least $\frac{|v|}{p^{n+1}-1}$.
\end{lemma}
\begin{proof} Set $m = \frac{|v|}{2p^{n+1}-2}$. We need to prove
that the image of $v$ vanishes inside
	\[
	Y:=\lim_{\mathbb{C}P^{m-1}}
	M \otimes \mathrm{THH}(\mathrm{BP}\langle n\rangle).
	\]
Since 
	\[
	M \otimes \mathrm{THH}(\mathrm{BP}\langle n\rangle)
	\simeq 
	(M\otimes \mathrm{BP}\langle n\rangle)
	\otimes_{\mathrm{BP}\langle n\rangle}
	\mathrm{THH}(\mathrm{BP}\langle n\rangle)
	\]
is a direct sum of shifts of $\mathbb{F}_p$, 
the skeletal filtration on $\mathbb{C}P^{m-1}$ gives
rise to an Adams resolution of  
$Y$ of length $m-1$.
The claim now follows from Proposition \ref{recall:smiths}(iv).
\end{proof}

\begin{lemma}\label{lem:hfpss-collapses}
The homotopy fixed point spectral sequence converging to
$\pi_*(\mathrm{M} \otimes \mathrm{THH}(\mathrm{BP}\langle n\rangle/
\mathrm{MU})^{hS^1})$ collapses at the $E_2$-page.
\end{lemma}
\begin{proof} The $E_2$-page can be described as
	\[
	\mathrm{BP}\langle n\rangle_*(M) 
	\otimes_{\mathrm{BP}\langle n\rangle_*}
	\mathrm{THH}(\mathrm{BP}\langle n\rangle/\mathrm{MU})_*[t]
	\]
By Proposition \ref{recall:smiths}(v), the images
of the equivariant maps
	\[
	\mathrm{MU} \otimes M \to
	\mathrm{M} \otimes \mathrm{THH}(\mathrm{BP}\langle n\rangle/
	\mathrm{MU})
	\]
	\[
	\mathrm{THH}(\mathrm{BP}\langle n\rangle/
	\mathrm{MU})
	\to
	\mathrm{M} \otimes \mathrm{THH}(\mathrm{BP}\langle n\rangle/
	\mathrm{MU})
	\]
induce maps of spectral sequences
whose images generate the $E_2$-page of 
the target as a ring. Every element
in the homotopy fixed point spectral sequence for both
$\mathrm{MU}^{hS^1}\otimes M$ and
$\mathrm{THH}(\mathrm{BP}\langle n\rangle/
\mathrm{MU})$ is a permanent cycle, so the claim follows.
\end{proof}

\begin{lemma}\label{lem:nice-z} There is an element $z \in
\pi_*(M \otimes \mathrm{THH}(\mathrm{BP}\langle n\rangle))$
with the following properties:
	\begin{enumerate}[{\rm (i)}]
	\item $z$ is central.
	\item $z$ maps to a power of $\sigma^2v_{n+1}$
	inside 
	$\pi_*(M \otimes 
	\mathrm{THH}(\mathrm{BP}\langle n\rangle/\mathrm{MU}))$.
	\item For some $m>0$, $t^mz$, in the $E_2$-term of the 
	homotopy fixed point spectral sequence,
	detects the image of a 
	central $v_{n+1}$-element from $\pi_*(M)$
	inside $\pi_*(M \otimes 
	\mathrm{THH}(\mathrm{BP}\langle n\rangle)^{hS^1})$.
	\item $\pi_*(M \otimes \mathrm{THH}(\mathrm{BP}\langle n\rangle))$
	is a finitely generated $\mathbb{Z}_{(p)}[z]$-module.
	\end{enumerate}
\end{lemma}
\begin{proof} First observe that each of these properties
is preserved after replacing $z$ by any power of itself,
so we may do this at any time in the argument.

By Proposition \ref{prop:descent-fp} and Proposition \ref{recall:smiths}(iii),
we may choose $z \in 
\pi_*(M\otimes \mathrm{THH}(\mathrm{BP}\langle n\rangle))$ which
lifts $\sigma^2v_{n+1}$ and for which 
$\pi_*(M\otimes \mathrm{THH}(\mathrm{BP}\langle n\rangle))$
is a finitely generated $\mathbb{Z}_{(p)}[z]$-module.
By Lemma \ref{lem:nice-power}, after replacing $z$ by a power,
we may assume that $z$ is central as well. So we have
chosen a $z$ which satisfies (i), (ii), and (iv). 

%Finally, by the previous lemma,
%the element $v$ is detected in the homotopy fixed point
%spectral sequence in filtration $\frac{|v|}{p^{n+1}-1}$.
%Setting $m= \frac{|v|}{2p^{n+1}-2}$, it follows that $v$ must
%be detected by an element of the form $t^mz'$ for some
%$z' \in \pi_*(M\otimes \mathrm{THH}(\mathrm{BP}\langle n\rangle))$. 

%The central $v \in \pi_*(M)$ is detected in the homotopy fixed point spectral sequence
%for $\pi_*(M\otimes \mathrm{THH}(\mathrm{BP}\langle n\rangle)^{hS^1})$ by some class $t^m z'$, where
%$z' \in \pi_*(M\otimes \mathrm{THH}(\mathrm{BP}\langle n\rangle))$.

Let
	\[
	f: M\otimes \mathrm{THH}(\mathrm{BP}\langle n\rangle)
	\to
	M\otimes \mathrm{THH}(\mathrm{BP}\langle n\rangle/\mathrm{MU})
	\]
be the canonical map. Let us denote by $\{E_r'\}$ the homotopy fixed point
spectral sequence computing
$\pi_*\mathrm{THH}(\mathrm{BP}\langle n\rangle)^{hS^1}$
and by
$\{E_r^{''}\}$ the homotopy fixed point spectral sequence
computing $\pi_*\mathrm{THH}(\mathrm{BP}\langle n\rangle/\mathrm{MU})^{hS^1}$.
We will denote by $E_r(f): E_r' \to E_r^{''}$ the map induced by $f$.

By Theorem \ref{thm:body-detection}, we know that
$v_{n+1}$ is detected in $E_2^{''}$
by $t(\sigma^2v_{n+1})$. 
Let $v$ denote a central $v_{n+1}$-element in $\pi_*(M)$, projected to $\pi_*(M \otimes \mathrm{THH}(\mathrm{BP}\langle n \rangle)$.
By the
definition of a $v_{n+1}$-element, there is an $m>0$ such that
$f(v)=v^{m}_{n+1}$ modulo the ideal $(p, ..., v_n)$. 
Property (iii) in Proposition \ref{recall:smiths} guarantees that
$M\otimes
\mathrm{THH}(\mathrm{BP}\langle n\rangle/\mathrm{MU})$
has $(p, ..., v_n) = 0$, and hence that $v$ is detected by
$t^m(\sigma^2v_{n+1})^m$ in $E_2^{''}=E_{\infty}^{''}$.

It follows that $v$ cannot be detected in $E_2^{'}$ in filtration higher than $2m$.
By Lemma \ref{lemma:high-Adams}, $v$ must be detected in $E_2^{'}$ by a class in filtration at least $2m$.
Say that $v$ is detected by $t^m z'$, where $z' \in \pi_*(M \otimes \mathrm{THH}(\mathrm{BP}\langle n \rangle)$.

Then, since $E_2(f)(t^mz') = t^m (\sigma^2 v_{n+1})^m$, and $E_2''=E_{\infty}''$, we must have that $f(z')=(\sigma^2v_{n+1})^m$.

%The class $E_2(f)(t^mz')$ in $E_2^{''}$ must detect
%$v$ modulo terms
%of filtration higher than $m$. Since $v$ is actually detected
%in filtration $m$, and $E_2^{''}=E_{\infty}^{''}$,
%we must have $E_2(f)(t^mz')=t^m(\sigma^2v_{n+1})^m$,
%and hence $f(z')=(\sigma^2v_{n+1})^m$.

After replacing $z$ and $v$ by suitable powers, the result now follows
from Lemma \ref{lem:nice-power} applied to the elements $z$ and $z'$.
\end{proof}

\begin{remark} At height one and primes $p\ge 5$ a
version of Lemma \ref{lem:nice-z}(iii)
was obtained
by Ausoni-Rognes in \cite[Proposition 4.8]{ausoni-rognes-redshift}.
\end{remark}

\begin{proof}[Proof of Theorem \ref{thm:canvan-new}]
Fix $v$, $z$, and $m$ as in the previous lemma.
Let
	\[
	X= \{\tau_{\ge j}
(M\otimes
\mathrm{THH}(\mathrm{BP}\langle n\rangle))\}
	\]
denote the filtered spectrum corresponding to
taking connective covers of
$
M\otimes
\mathrm{THH}(\mathrm{BP}\langle n\rangle)$. We can
choose a lift $\tilde{v} \in \pi_*(X^{\ge 2mp^{n+1}})^{hS^1}$ of $v$
and form the cofibers in filtered spectra:
	\[
	Y:= X^{hC_{p^k}}/\tilde{v}, \,\,
	Z:= X^{tC_{p^k}}/\tilde{v}.
	\]
The filtered spectra $Y$ and $Z$ give spectral sequences
converging to
$M/v \otimes 
\mathrm{THH}(\mathrm{BP}\langle n\rangle)^{hC_{p^k}}$
and
$M/v \otimes 
\mathrm{THH}(\mathrm{BP}\langle n\rangle)^{tC_{p^k}}$
respectively, and the canonical map
$Y \to Z$ converges to the canonical map between
these two spectra. Observe that the $E_2$-page of
the spectral sequence
for $Y$
is concentrated
in nonnegative filtration (with our grading conventions).
It will therefore suffice to prove that the $E_2$-page of
the spectral for $Z$ is eventually concentrated
in negative filtration (uniformly in $k$). From the cofiber sequence
(with suspensions and grading shifts omitted)
	\[
	\mathrm{gr}(X)^{tC_{p^k}}
	\stackrel{t^mz}{\to}
	\mathrm{gr}(X)^{tC_{p^k}}
	\to \mathrm{gr}(Z),
	\]
it is enough to show that multiplication by $z$
on $\pi_*(\mathrm{gr}(X)^{tC_{p^k}})$ is eventually
an isomorphism in nonnegative filtration.
By \cite[Lemma IV.4.12]{nikolaus-scholze},
we have $\mathrm{gr}(X)^{tC_{p^k}}=
\mathrm{gr}(X)^{tS^1}/p^k$. So it suffices to prove that
multiplication by $z$ is eventually an isomorphism
in nonnegative filtration
for the group
	\[
	\pi_*(M\otimes \mathrm{THH}(\mathrm{BP}\langle n\rangle))
	[t^{\pm 1}].
	\]
But, more generally, if $L$ is any finitely generated
$\mathbb{Z}_{(p)}[z]$-module, then the analogous
claim is true for $L[t^{\pm 1}]$.
\end{proof}

\appendix

\section{Suspension Maps}\label{sec:suspension}

Suppose $R$ is an augmented (discrete) algebra over
a field $k$ with augmentation ideal $I$. Then there is a
homomorphism of abelian groups
	\[
	\sigma: I \to \mathrm{Tor}_1^R(k,k)
	\]
where $\sigma x$ is represented by the class $[x]$ in the bar complex.
At various points in the paper we use a generalization of
this construction to the spectrum level. 
Specifically, it is used in \S\ref{sec:mult} to
provide canonical lifts of elements in K\"unneth
spectral sequences and, more crucially, in \S\ref{sec:detection}
in order to prove the Detection
Theorem (Theorem \ref{thm:body-detection}). We make no claim
of originality for the material in this appendix, though we were not able
to find the Detection Lemma (Lemma \ref{lem:detection}) in the
literature. 

\begin{convention} Throughout this section $\mathcal{C}$ will denote
a stable, presentably symmetric monoidal category
with unit object $\mathbf{1}$. 
\end{convention}

\subsection{Construction of suspension maps}

For the purposes of functoriality, it is convenient to construct
our suspension maps in the setting of factorization homology.
Let $\mathsf{Mfld}^{\mathrm{fr}}_n$ denote the category
of framed $n$-manifolds as constructed in \cite{ayala-francis},
equipped with its symmetric monoidal structure under disjoint
unions. Let $\mathsf{Disk}^{\mathrm{fr}}_n$ be the full
subcategory spanned by $n$-manifolds equivalent
to disjoint unions of copies of $\mathbb{R}^n$. This category
is equivalent to the symmetric monoidal envelope of
the $\mathbb{E}_n$-operad. Factorization homology is then
given by a functor

	\[
	\int: \mathsf{Alg}_{\mathbb{E}_n}(\mathcal{C})
	\simeq \mathsf{Fun}^{\otimes}(\mathsf{Disk}^{\mathrm{fr}}_n,
	\mathcal{C}) \longrightarrow
	\mathsf{Fun}^{\otimes}(\mathsf{Mfld}^{\mathrm{fr}}_n,
	\mathcal{C})
	\]
which is left adjoint to restriction. Here $\mathsf{Fun}^{\otimes}(-,-)$
denotes the category of symmetric monoidal functors.

\begin{construction}[Unreduced suspension]
Since factorization homology is functorial on 
$\mathsf{Mfld}_n^{\mathrm{fr}}$ we always have an
(unpointed) map of spaces:
	\[
	\mathrm{Map}_{\mathsf{Mfld}_n^{\mathrm{fr}}}
	(N,M)
	\to
	\mathrm{Map}_{\mathcal{C}}\left(\int_N A, \int_MA\right)
	\]
If we set $N=\mathbb{R}^n$, then $\int_NA=A$ the above
is adjoint to a map
	\[
	s^M: \mathrm{Map}_{\mathsf{Mfld}_n^{\mathrm{fr}}}
	(\mathbb{R}^n,M)_+ \otimes A 
	= M_+\otimes A \to \int_MA
	\]
which is functorial in $M$ and $A$. Here we have used $X_+\otimes(-)$
to denote the tensoring of $\mathcal{C}$ over the category of
unpointed spaces. 

We observe that, when $A=\mathbf{1}$, this map is canonically
identified with the collapse
	\[
	M_+ \otimes \mathbf{1} \to \mathbf{1} \simeq \int_M \mathbf{1}.
	\]
\end{construction}

\begin{construction}[Suspension] Let $M$ be a framed
$n$-manifold equipped with a basepoint. From the previous
construction, we have a functorial diagram:
	\[
	\xymatrix{
	M_+ \otimes \mathbf{1} \ar[r]^-{s^M} \ar[d] & \mathbf{1} \ar[d]\\
	M_+ \otimes A \ar[r]_{s^M} & \int_MA
	}
	\]
The choice of basepoint provides a splitting
of the top map and hence a commutative square 
(functorial in $A$ and basepoint preserving maps
in $M$):
	\[
	\xymatrix{
	M \otimes \mathbf{1} \ar[r]\ar[d] & 0\ar[d]\\
	M\otimes A \ar[r] & \int_MA
	}
	\]
Thus we get a map from the pushout of the diagram
with the lower right vertex deleted:
	\[
	\sigma^M: M \otimes (A/\mathbf{1}) \to \int_MA
	\]
where $A/\mathbf{1}$ denotes the cofiber of the unit map
for $A$.
\end{construction}

\begin{remark} If one instead used the transposed diagram
	\[
	\xymatrix{
	M\otimes \mathbf{1} \ar[r] \ar[d] & M \otimes A\ar[d]\\
	0 \ar[r] & \int_MA
	}
	\]
this would alter the definition of $\sigma^M$ by $-1$.
Since neither choice seems canonical, and we will often
find ourselves rotating distinguishing triangles in the arguments
below, we will mostly make and prove statements about
$\sigma^M$ only up to a factor of $\pm 1$. 
For convenience and for the purposes of this paper,
we will, in \S\ref{ssec:detect}, fix a choice so as to make
a certain equation true on the nose rather than up to
a factor of $\pm 1$.
\end{remark}

\subsection{Examples of suspension maps}

\begin{example}[Dimension 0]
If $A$ is an $\mathbb{E}_0$-algebra, then we denote
$\sigma^{S^0}$ by $\sigma$, which is a map
	\[
	\sigma: (A/\mathbf{1}) \to \int_{S^0}A = A \otimes A.
	\]
From the construction, $\sigma$ comes
as the induced map from the square
	\[
	\xymatrix{
	\mathbf{1} \ar[rr]\ar[d] && 0 \ar[d]\\
	A \ar[rr]_-{\mathrm{id} \otimes 1 - 1 \otimes \mathrm{id}} &&
	A \otimes A
	}
	\]
Equivalently, we can describe this map (up to sign)
as arising from the large square in the
diagram
	\[
	\xymatrix{
	\Sigma^{-1}(A/\mathbf{1})\ar[dd]\ar[rr] \ar[dr]&& 0\ar[d]\\
	&\mathbf{1} \ar[r]\ar[d] & A \ar[d]^-{1\otimes \mathrm{id}}\\
	0 \ar[r] &A \ar[r]_-{\mathrm{id} \otimes 1} & A \otimes A
	}
	\]
\end{example}

\begin{variant}\label{var:suspension} 
Recall that the category $\mathsf{Alg}_{\mathbb{E}_1}(\mathcal{C})$
carries an action of $C_2$ given informally by sending an
$\mathbb{E}_1$-algebra $B$ to the algebra $B^{\mathrm{op}}$
equipped with the opposite multiplication. 

Let $R \in \mathsf{Alg}_{\mathbb{E}_1}(\mathcal{C})^{hC_2}$
be an object in the fixed points so that, in particular, $R$ comes
equipped with an equivalence $\tau: R \simeq R^{\mathrm{op}}$.
This induces an equivalence
\[(-)^{\tau}: \mathsf{LMod}_R(\mathcal{C})
 \to \mathsf{RMod}_R(\mathcal{C})\]
which is the identity on underlying objects.
Now let $k$ be a left $R$-module equipped with a map
$\mathbf{1} \to k$ in $\mathcal{C}$. We can
extend this to a left $R$-module map $1_k: R \to k$
and to a right $R$-module map $1^{\tau}_k: R \to k^{\tau}$.
Since $\tau$ is an equivalence, there is a canonical identification
between the fibers
	\[
	\mathrm{fib}(1_k) \simeq \mathrm{fib}(1_k^{\tau})
	\]
and we denote either by $I$.

This is enough to make sense of the following
diagram in $\mathcal{C}$:
	\[
	\xymatrix{
	I\ar[dd]\ar[rr] \ar[dr]&& 0\ar[d]\\
	&R \ar[r]^{1_k}\ar[d]_{1_k^{\tau}}
	& k \ar[d]^-{1_k^{\tau}\otimes \mathrm{id}}\\
	0 \ar[r] &k \ar[r]_-{\mathrm{id} \otimes 1_k} & k^{\tau} \otimes_R k
	}
	\]
Thus
we may extend the definition of $\sigma$ in this case to:
	\[
	\sigma: \Sigma I \to k^{\tau}\otimes_Rk.
	\]
\end{variant}

\begin{lemma}[Compatibility with K\"unneth spectral sequence]
\label{lem:compatible-with-kunneth}
Take $\mathcal{C} = \mathsf{Sp}$ and adopt
notation as in Variant \ref{var:suspension}. Let
$i: I \to R$ denote the fiber of $R\to k$. Suppose that the map
	\[
	1_k: \pi_*R \to \pi_*k
	\]
is surjective.
Then, for any $x \in \pi_*I$,
$\sigma(x) \in \pi_{*+1}(k^{\tau}\otimes_Rk)$ is detected in the
K\"unneth spectral sequence in filtration $1$ by the 
class 
	\[
	[1 \otimes i(x)\otimes 1] \in \mathrm{Tor}_1^{\pi_*R}(\pi_*k,
	\pi_*k),
	\]
up to sign.
\end{lemma}
\begin{proof} First we claim that the composite
	\[
	\Sigma I \stackrel{\sigma}{\longrightarrow}
	k^{\tau}\otimes_Rk \to k^{\tau} \otimes_R \Sigma I
	\]
is homotopic, up to sign, to the map
$1^{\tau}_k \otimes \mathrm{id}$. Indeed, consider
the following
diagram
	\[
	\xymatrix{
	&I \ar[r] \ar[ddl] \ar[d]& 0 \ar[d]\ar[ddr]&\\
	&R \ar[r] \ar[d]& k \ar[d]&\\
	0\ar[r]&k \ar[r]\ar[d]& k\amalg_Rk \ar[r] \ar[d]
	& k^{\tau}\otimes_Rk\ar[d]\\
	&0 \ar[r] & \Sigma I \ar[r] & k^{\tau} \otimes_R \Sigma I
	}
	\]
The vertices of the large trapezoid form the commutative square
	\[
	\xymatrix{
	I \ar[r]\ar[d] & 0\ar[d]\\
	0 \ar[r] & k^{\tau}\otimes_Rk
	}
	\]
used to define $\sigma$, and hence
the induced map on the pushout of
$(0 \leftarrow I \rightarrow 0)$ gives the desired factorization.

It follows from naturality of the K\"unneth spectral sequence
that the map $R\otimes_R I = 
\Sigma I \to k^{\tau}\otimes_R I$ induces, on $E^2$-terms,
the projection map
	\[
	\pi_*I \to \mathrm{Tor}_0^{\pi_*R}(\pi_*k, \pi_*I) = \pi_*I/(\pi_*I)^2.
	\]
Finally, recall the construction of the K\"unneth spectral
sequence for $M\otimes_RN$
proceeds by lifting a (graded) free $\pi_*R$-resolution
of $N$ to a filtration by left $R$-modules \cite{tilson},
and that this construction
is natural in the resolution. Since
	\[
	0 \to \pi_*I \to \pi_*R \to \pi_*k \to 0
	\]
is exact,
we may choose a resolution
$C_*$ of $\pi_*k$ that begins with $C_0 = R$ and with the
property that
	\[
	\pi_*I \leftarrow C_1 \leftarrow C_2 \leftarrow \cdots
	\]
is a resolution of $\pi_*I$. Considering $\pi_*R$ as a complex
concentrated in degree zero, the quotient map
$C_* \to (C_*/R)$ can then be lifted to a map of filtered objects
and then we may apply $k^{\tau}\otimes_R(-)$ to this map.
This gives a map of spectral sequences which, on the $E^2$-page,
gives the boundary map
$\partial: \mathrm{Tor}_1^{\pi_*R}(\pi_*k, \pi_*k) \to
\Sigma \mathrm{Tor}^{\pi_*R}_0(\pi_*k, \pi_*I)$, which is
an isomorphism. The result follows.
\end{proof}

\begin{example}[Dimension 1] \label{exm:double-suspension}
The circle acts on itself by framed maps, where we use the Lie group
framing, and hence the map to Hochschild homology
	\[
	s^{S^1}: S^1_+ \otimes A \to \int_{S^1}A = \mathrm{HH}(A)
	\]
is circle equivariant. Since the source of $s^{S^1}$ is induced,
the map must be induced from its restriction along the identity,
i.e. $s^{S^1}$ is adjoint to the nonequivariant map
$A \to \mathrm{HH}(A)$ corresponding to the identity element
in $S^1$. We abbreviate the reduced suspension map by
$\sigma^2: \Sigma (A/\mathbf{1}) \to \mathrm{HH}(A)$. 
\end{example}

\subsection{Relationship with the cotangent complex}

Let $A$ be an $\mathbb{E}_n$-algebra in $\mathcal{C}$.
We will abbreviate by $\mathbf{L}_A^{(n)}$ the
$\mathbb{E}_n$-algebra cotangent complex of $A$,
which is an $\mathbb{E}_n$-$A$-module. 
Recall \cite[Theorem 7.3.5.1]{ha} that we have a functorial
cofiber sequence of $\mathbb{E}_n$-$A$-modules:
	\[
	\mathcal{U}^{(n)}(A) \to A \to \Sigma^n\mathbf{L}^{(n)}_A.
	\]

\begin{example} When $n=0$, this is the cofiber sequence
	\[
	\mathbf{1} \to A \to (A/\mathbf{1})
	\]
\end{example}

If $k\le n$ we may apply $A\otimes_{\mathcal{U}^{(k-1)}}(-)$ to
the cofiber sequence computing the $\mathbb{E}_{k-1}$-cotangent
complex and get a cofiber sequence in $\mathcal{C}$:
	\[
	A \to A \otimes_{\mathcal{U}^{(k-1)}(A)} A
	\simeq \mathcal{U}^{(k)}(A) \to 
	A \otimes_{\mathcal{U}^{(k-1)}(A)} \Sigma^{k-1}\mathbf{L}_A^{(k-1)}
	\]
This gives a functorial splitting of objects in $\mathcal{C}$:
	\[
	A \oplus \Sigma^{k-1}\mathbf{L}^{(k)}_A \simeq_{\mathcal{C}} 
	\mathcal{U}^{(k)}(A)
	\]
and an identification
	\[
	\mathbf{L}^{(k)}_A \simeq_{\mathcal{C}} 
	A \otimes_{\mathcal{U}^{(k-1)}(A)}\mathbf{L}_A^{(k-1)}.
	\]
(Here we have places a subscript on the equivalence to emphasize
that this equivalence is not one of $\mathbb{E}_k$-$A$-modules).

\begin{lemma}\label{lem:suspn-and-cotangent} For $k\le n$,
the following diagram commutes (up to sign):
	\[
	\xymatrix{
	\Sigma^{k-1}A/\mathbf{1}\ar[d]_{\sigma^k} \ar@{=}[r] & \Sigma^{k-1}
	\mathbf{L}_A^{(0)} \ar[d]\\
	\mathcal{U}^{(k)}(A) \ar[r] & \Sigma^{k-1}\mathbf{L}^{(k)}_A
	}
	\]
\end{lemma}
\begin{proof} We prove this by induction on $k$, the base case
being trivial. For the inductive step, observe that, by functoriality of
of $\sigma^M$ in $M$, we have a diagram
	\[
	\xymatrix{
	0 \ar@{=}[r]\ar[d] & (\mathbb{R}^k) \otimes (A/\mathbf{1})\ar[d] &
	\Sigma^{k-1}(A/\mathbf{1}) \ar[r] \ar[l]\ar[d]_{\sigma^k}
	&  (\mathbb{R}^k) \otimes (A/\mathbf{1}) \ar@{=}[r] \ar[d]& 0\ar[d]\\
	A \ar@{=}[r]& \int_{\mathbb{R}^k}A & \int_{\mathbb{R}^k-\{0\}}A 
	\ar[l]\ar[r]&
	\int_{\mathbb{R}^k}A \ar@{=}[r]& A
	}
	\]
The induced map on the pushout is, on the one hand,
given by $\sigma^{k+1}$ and, on the other hand, by the
inductive hypothesis, given by
the composite
	\[
	\Sigma^{k}(A/\mathbf{1}) \stackrel{\sigma^k}{\to}
	\Sigma \mathcal{U}^{(k)}(A) \to \Sigma^{k}\mathbf{L}_A^{(k)}
	\stackrel{\sigma}{\to} \mathcal{U}^{k+1}(A)
	\]
where $\sigma$ is constructed
as in Variant \ref{var:suspension}\footnote{Notice that, in this inductive step, $k<n$, so
$\mathcal{U}^{(k)}(A)$ is at least an $\mathbb{E}_2$-algebra,
and hence the involution $\tau$ is trivial.}
On the other hand, as explained in the beginning of the proof of
Lemma \ref{lem:compatible-with-kunneth}, the composite
	\[
	\Sigma^k\mathbf{L}_A^{(k)} \stackrel{\sigma}{\to}
	\mathcal{U}^{(k+1)}(A) \to \Sigma^k\mathbf{L}_A^{(k+1)}
	\simeq A\otimes_{\mathcal{U}^{(k)}(A)}\Sigma^k\mathbf{L}_A^{(k)}
	\]
is given by the map $1\otimes \mathrm{id}$. This completes
the proof.
\end{proof}

\subsection{Undoing suspension in Hochschild homology}\label{ssec:detect}

If $X\in \mathsf{Fun}(\mathrm{B}S^1, \mathcal{C})$ is an object
equipped with an $S^1$-action, then we may
compute the limit over $\mathbb{C}P^1 \subseteq \mathrm{B}S^1$
using the fiber sequence
	\[
	\lim_{\mathbb{C}P^1} X \to X \to \Sigma^{-1}X,
	\]
where $X \to \Sigma^{-1}X$ is adjoint to the (reduced) action map
$\Sigma X \to S^1_+ \otimes X \to X$. We denote the connecting
homomorphism by
	\[
	t: \Sigma^{-2}X \to \lim_{\mathbb{C}P^1}X.
	\]
	
Our goal in this section is to prove the following lemma, which
allows us to use the circle action on Hochschild homology to
`undo' the suspension.

\begin{lemma}[Detection Lemma] \label{lem:detection}There is a functorial diagram\footnote{Recall that the definition of the suspension
map requires a choice, and that altering this choice
multiplies the map by $(-1)$. For the purposes
of this paper, we will fix this choice so 
that the diagram in this lemma commutes.}:
	\[
	\xymatrix{
	\Sigma^{-1}(A/\mathbf{1}) \ar[r]\ar[d]_-{\sigma^2} & \mathbf{1} \ar[d]\\
	\Sigma^{-2}\mathrm{HH}(A) \ar[r]_-{t} & 
	\lim_{\mathbb{C}P^1}\mathrm{HH}(A)
	}
	\]
Here, the map $\mathbf{1} \to \lim_{\mathbb{C}P^1}\mathrm{HH}(A)$
arises from the $S^1$-equivariant map $\mathbf{1} \to \mathrm{HH}(A)$
where $\mathbf{1}$ has the trivial action.
\end{lemma}

\begin{proof}
We have a diagram, functorial in $A$,
	\[
	\xymatrix{
	\mathbf{1} \ar[r] \ar[d]& 0\ar[dd]\\
	A \ar[d]& \\
	\mathrm{HH}(A) \ar[r] & \Sigma^{-1}\mathrm{HH}(A)
	}
	\] 
We can left Kan extend to a diagram:
\[\begin{tikzcd}
	{\mathbf{1}} && 0 & {} \\
	&&&& A && {\overline{A}} \\
	&& 0 & {} \\
	&&&& {\mathrm{HH}(A)} && {\Sigma^{-1}\mathrm{HH}(A)}
	\arrow[from=2-5, to=4-5]
	\arrow[from=2-5, to=2-7]
	\arrow[from=2-7, to=4-7]
	\arrow[from=4-5, to=4-7]
	\arrow[from=1-1, to=2-5]
	\arrow[from=1-1, to=1-3]
	\arrow[from=1-3, to=2-7]
	\arrow[from=1-3, to=3-3]
	\arrow[from=3-3, to=4-7]
\end{tikzcd}\]

The map $\overline{A} \to \Sigma^{-1}\mathrm{HH}(A)$ is
adjoint to $\sigma: \Sigma \overline{A} \to \mathrm{HH}(A)$, since
this latter map was constructed as the induced map on the pushout
of the diagram
	\[
	\xymatrix{
	\Sigma \mathbf{1} \ar[r] \ar[d]& 0\ar[d]\\
	\Sigma A \ar[r] & \mathrm{HH}(A)
	}
	\]
	
Now we may further right Kan extend to a diagram:

\[
\begin{tikzcd}
	{\mathbf{1}} && 0 \\
	&&&& A && {\overline{A}} \\
	{\lim_{\mathbb{C}P^1}\mathrm{HH}(A)} && 0 & {} \\
	&&&& {\mathrm{HH}(A)} && {\Sigma^{-1}\mathrm{HH}(A)}
	\arrow[from=1-1, to=2-5]
	\arrow[from=1-3, to=2-7]
	\arrow[from=1-1, to=3-1]
	\arrow[from=3-1, to=3-3]
	\arrow[from=1-1, to=1-3]
	\arrow[from=1-3, to=3-3]
	\arrow[from=3-1, to=4-5]
	\arrow[from=2-5, to=4-5]
	\arrow[from=2-5, to=2-7]
	\arrow[from=2-7, to=4-7]
	\arrow[from=4-5, to=4-7]
	\arrow[from=3-3, to=4-7]
\end{tikzcd}
\]

Here, the map $\mathbf{1} \to \lim_{\mathbb{C}P^1}\mathrm{HH}(A)$ arises from the canonical trivialization of the $S^1$-action
$S^1_+ \otimes \mathbf{1} \to 
\mathbf{1} \simeq \mathrm{HH}(\mathbf{1})$.

We may view the above cube as a map of fiber sequences
	\[
	\xymatrix{
	\mathbf{1} \ar[r]\ar[d] & A \ar[r] \ar[d]& \overline{A} \ar[d]\\
	\lim_{\mathbb{C}P^1}\mathrm{HH}(A) \ar[r] & 
	\mathrm{HH}(A) \ar[r] & \Sigma^{-1}\mathrm{HH}(A)
	}
	\]
which then yields the desired diagram:
	\[
	\xymatrix{
	\Sigma^{-1}\overline{A} \ar[d]_-{\sigma^2} 
	\ar[r]& \mathbf{1} \ar[d]\\
	\Sigma^{-2}\mathrm{HH}(A)\ar[r]_-{t} &
	\lim_{\mathbb{C}P^1}\mathrm{HH}(A)
	}
	\]
\end{proof}
\section{Recollections on Graded Objects}\label{sec:graded-koszul}

In this section we briefly review some facts about graded
rings used in the body of the paper.

\begin{definition} Let $k$ be an $\mathbb{E}_{\infty}$-ring.
The category of graded $k$-modules
	is defined by
		\[
		\mathsf{grMod}_k := \mathsf{Fun}(\mathbb{Z}^{\mathrm{ds}},
		\mathsf{Mod}_k),
		\]
	where $\mathbb{Z}^{\mathrm{ds}}$ denotes
	the integers viewed as a 0-category.
	The category $\mathsf{grMod}_k$ is a presentably
	symmetric monoidal category under Day convolution.
	If $M$ is a graded $k$-module, we will denote by $M_i$
	its values at $i$, and by $M(n)$ the precomposition
	with addition by $-n$ (so that $M(n)_i = M_{i-n}$).
	We will refer to the grading as the \textbf{weight} throughout.
\end{definition}

We will need various (co)connectivity conditions and finiteness conditions.
\begin{definition}
	\begin{itemize}
		\item Let $A$ be an augmented, graded
		 $\mathbb{E}_n$-$k$-algebra. We say that 
		 $A$ is \textbf{weight-connected}
		 (resp. \textbf{weight-coconnected}) if
		 the fiber of the augmentation $A \to k$ is
		 concentrated in positive grading (resp. negative
		 grading).
		 We denote the corresponding categories
		 with superscripts $\mathrm{wt-cn}$
		 and $\mathrm{wt-ccn}$, respectively.
		 \item We denote by
		 $\mathsf{grMod}^{\mathrm{wt}\ge n}_k$
		 (resp. $\mathsf{grMod}^{\mathrm{wt}\le n}_k$)
		 the full subcategory of graded $k$-modules concentrated
		 in weights at least $n$ (resp. at most $n$).
		 We will write $M\ge n$ (resp. $M\le n$) to indicate
		 that $M$ belongs to this subcategory.
	\end{itemize}
\end{definition}

\begin{remark}\label{rmk:regrading}
Observe that the map $-1: \mathbb{Z}^{\mathrm{ds}} \to 
\mathbb{Z}^{\mathrm{ds}}$ is a symmetric monoidal equivalence,
and hence induces a symmetric monoidal equivalence
on the category of graded $k$-modules, algebras, etc. 
It follows that any result about weight-connected algebras,
or modules of weight bounded below by $n$,
has a counterpart for weight-coconnected algebras or
modules of weight bounded above by $-n$. 
\end{remark}

\begin{lemma} Let $A$ be an augmented, graded $k$-algebra
and denote by $\overline{A}$ the fiber of the augmentation.
Let $M \in \mathsf{LMod}_A$ and $N \in \mathsf{RMod}_A$,
then there is a filtration on $M\otimes_AN$:
	\[
	M\otimes N = F_0 \to F_1 \to \cdots \to \colim F_i = 
	M\otimes_A N
	\]
such that
	\[
	\mathrm{gr}_i(M\otimes_A N) \simeq 
	\Sigma^{i}M \otimes 
	\overline{A}^{\otimes i} \otimes N.
	\]
\end{lemma}
\begin{proof} The relative tensor product is computed 
by the geometric realization of the standard simplicial
object with $n$th term $M \otimes A^{\otimes n} \otimes N$
(\cite[4.4.2.8]{ha})
and hence, by
the categorical Dold-Kan correspondence (\cite[1.2.4.1]{ha}),
is also computed as the colimit of a filtered object with associated
graded corresponding to the normalized complex
(which can be computed in the homotopy category), as indicated.
\end{proof}

\begin{corollary}\label{prop:misc-graded-facts}
Let $A$ be a weight connected algebra
$L, N \in \mathsf{LMod}_A$ and $M \in \mathsf{RMod}_A$.
If $M\ge \alpha$ and $N\ge \beta$ then
	$M\otimes_{A} N \ge \alpha+\beta$ and
	$(M\otimes_AN)_{\alpha +\beta} = M_{\alpha}\otimes N_{\beta}$.
\end{corollary}

We will now study a natural filtration on the category of
graded modules over a weight-connected algebra.
If $A$ is weight-connected, we denote by
	\[
	\mathsf{LMod}^{\mathrm{wt}\ge j}_A,
	\mathsf{LMod}^{\mathrm{wt}\le j}_A \subseteq
	\mathsf{LMod}_A
	\]
the full subcategories spanned by those modules
which are concentrated in weights at least $j$ and
at most $j$, respectively.

\begin{lemma}\label{lem:truncating} Let $A$ be a weight-connected
$\mathbb{E}_n$-$k$-algebra for some $1\le n\le \infty$.
	\begin{enumerate}[{\rm(i)}]
	\item The inclusion
	$\mathsf{LMod}_A^{\mathrm{wt}\ge j} \to
	\mathsf{LMod}_A$ admits a
	right adjoint, $(-)_{\ge j}$, computed as:
		\[
	(M_{\ge j})_i = \begin{cases} M_i & i\ge j\\
	0 & \text{else}.
	\end{cases}
		\]
	\item The inclusion $\mathsf{LMod}_A^{\mathrm{wt}\le j}
	\to \mathsf{LMod}_A$ admits a left adjoint, $(-)_{\le j}$,
	computed as $M \mapsto M/M_{\ge j+1}$. In particular,
		\[
		(M_{\le j})_i = \begin{cases}
		M_i & i\le j\\
		0 & \text{else}.
		\end{cases}
		\]
	\item The subcategory $\mathsf{LMod}_A^{\mathrm{wt}\ge0}$
	inherits an $\mathbb{E}_{n-1}$-monoidal structure.
	\item The localizations $(-)_{\le m}$ are compatible
	with the $\mathbb{E}_{n-1}$-monoidal structure on
	$\mathsf{LMod}_A^{\mathrm{wt}\ge 0}$.
	\item The tower
		\[
		A \to \cdots \to A_{\le m} \to A_{\le m-1} \to \cdots \to k
		\]
	of $\mathbb{E}_n$-$k$-algebras is a tower
	of square-zero extensions, i.e. we have pullback
	diagrams of $\mathbb{E}_n$-$k$-algebras:
		\[
		\xymatrix{
		A_{\le m}\ar[r]\ar[d] & k\ar[d]\\
		A_{\le m-1} \ar[r] & k\oplus
		\Sigma A_m(m)
		}
		\]
	\end{enumerate}
\end{lemma}
\begin{proof} The existence of these adjoints is immediate
since the inclusions preserve all limits and colimits.
To compute $M_{\ge j}$, observe that, for $i\ge j$,
the $A$-module $A(i)$ is in weights at least $j$, since
$A$ is weight-connected. The homogeneous component
$(M_{\ge j})_i$ is computed as the spectrum of maps of
$A$-modules from $A(i)$ to $M_{\ge j}$ which, by
the adjunction, is the same as the spectrum of maps
from $A(i)$ to $M$, which is $M_i$. This proves (i).

Claim (ii) follows formally from the observation that,
if $M \ge m+1$ and $N\le m$, then every map
$M \to N$ is zero. 

Claim (iii) follows, using \cite[2.2.1.1]{ha}, from the fact that 
$\mathsf{LMod}_A^{\mathrm{wt}\ge 0}$ contains the unit
and is closed under tensor products, by 
Corollary \ref{prop:misc-graded-facts}.

For claim (iv) we must show that, if
$M \to M'$ is an equivalence in weights at most $m$, then
so is $Z\otimes_A M \to Z\otimes_A M'$ and
$M\otimes_A Z \to M\otimes_A Z'$ for $Z\ge 0$. 
Let $F$ be the fiber of $M \to M'$ so that $F\ge m+1$, and
then the result follows from 
Corollary \ref{prop:misc-graded-facts} applied to
$Z\otimes_AF$ and $F\otimes_A Z$.

For the final claim, we need to produce a derivation, i.e. a map,
	\[
	\mathbf{L}^{(n)}_{A_{\le m-1}} \to \Sigma A_m(m)
	\]
refining the map $A_{\le m-1} \to \mathrm{cofib}(A_{\le m} \to A_{\le m-1})
= A_m(m)$. Here $\mathbf{L}^{(n)}$ denotes the 
$\mathbb{E}_n$-cotangent complex. We will produce this refinement
as a composite
	\[
	\mathbf{L}^{(n)}_{A_{\le m-1}} \to \mathbf{L}^{(n)}_{A_{\le m-1}/
	A_{\le m}} \to \Sigma A_m(m)
	\]
where the first map is the canonical one to the relative cotangent
complex and the second
is projection onto the first nonzero weight.
By \cite[7.5.3.1]{ha} applied to the $\mathbb{E}_n$-monoidal
category $\mathsf{Mod}^{\mathbb{E}_n}_{A_{\le m}}$ of
$\mathbb{E}_n$-$A_{\le m}$-modules, we can compute the relative
cotangent complex using the cofiber sequence:
	\[	
	 \mathcal{U}^{(n)}_{A_{\le m}}(A_{\le m-1})
	\to A_{\le m-1}\to
	\Sigma^{n}\mathbf{L}^{(n)}_{A_{\le m-1}/A_{\le m}}
	\]
of $\mathbb{E}_n$-$A_{\le m-1}$-modules. Using the recursive
construction of the enveloping algebra, we are reduced to
proving the following claim:
	\begin{itemize}
	\item[(*)] If $A \to B$ is a map of
	weight-connected $\mathbb{E}_1$-algebras
	with cofiber $C\ge j$, denote by $C'$
	the cofiber of $B\otimes_AB \to B$. Then $C'\ge j$
	and $C'_j = \Sigma C_j$. 
	\end{itemize}
To prove (*), observe that the multiplication map admits a section
so that $C' \simeq \Sigma B\otimes_A C$. The result now follows
from Corollary \ref{prop:misc-graded-facts}.
\end{proof}

\section{Spectral Sequences}\label{sec:sseq}

In the body of the paper, we use various spectral sequences
and maps of spectral sequences obtained
by applying certain functors and natural transformations to
towers. The purpose of this appendix is to check that these
maneuvers produce convergent spectral sequences
under certain conditions satisfied in the cases of interest.

\begin{convention} Throughout this section, $\mathcal{C}$
will denote a presentably symmetric monoidal stable
category with a $t$-structure. We will assume
that $\mathcal{C}$ satisfies the following properties
(all of which are satisfied, for example, by
modules over a connective $\mathbb{E}_{\infty}$-ring, equipped with
an action of a group):
	\begin{enumerate}[(i)]
	\item The $t$-structure is compatible with filtered colimits,
	i.e. $\mathcal{C}_{\le 0}$ is closed under filtered colimits.
	\item The $t$-structure is left and right complete, which in
	this case is equivalent to saying that 
		\[
		\colim_{n\to -\infty} \tau_{\le n}X = 0 =
		\lim_{n\to \infty} \tau_{\ge n}X.
		\]
	\item The $t$-structure is compatible with the
	symmetric monoidal structure; i.e. $\mathbf{1} \in 
	\mathcal{C}_{\ge 0}$ and 
	$X \otimes Y \in \mathcal{C}_{\ge n+m}$
	whenever $X \in \mathcal{C}_{\ge n}$ and 
	$Y \in \mathcal{C}_{\ge m}$.
	\end{enumerate}
\end{convention}

\subsection{Towers and convergence}\label{ssec:towers}

\begin{convention}\label{conv:grading-sseq}
Given a tower $\{X^{\ge s}\} \in \mathsf{Fun}(\mathbb{Z}^{\mathrm{op}}, \mathcal{C})$,
we index the associated spectral sequence so that
	\[
	E_2^{s,t} = \pi_{t-s}\mathrm{gr}^tX = \pi_{t-s}(\mathrm{cofib}(X^{\ge t+1}
	\to X^{\ge t})).
	\]
We write $X^{-\infty} := \colim X^{\ge s}$.
\end{convention}

\begin{warning} There is not a typo here: we mean $\pi_{t-s}\mathrm{gr}^t$
and not $\pi_{t-s}\mathrm{gr}^s$. The latter would have differentials as in
the $E_1$-term of a spectral sequence, whereas the former will behave
as an $E_2$-term.
\end{warning}

\begin{definition} Suppose $\{X^{\ge s}\}$ is a tower with associated
spectral sequence $\{E_r^{s,t}\}$. We say that $E_r$ 
\textbf{converges conditionally} to $\pi_*X^{-\infty}$ if 
$\lim X^{\ge s} = 0$. We say that $E_r$ \textbf{converges strongly}
if the associated filtration $F^s(\pi_{t-s}X^{-\infty}):=
\mathrm{im}(\pi_{t-s}X^{\ge 2t-s} \to \pi_{t-s}X^{-\infty})$ satisfies
	\[
	\holim_s F^s\pi_nX^{-\infty} = 0
	\]
\end{definition}

We will content ourselves below with
establishing general conditions under which
conditional convergence holds.
In the body of the paper, when we claim that some
spectral sequence actually 
converges strongly, it is because it also satisfies the conditions of
Boardman's theorem \cite[Theorem 7.1]{boardman}
for spectral sequences with
entering differentials:

\begin{theorem}[Boardman] Suppose that $E_r$
converges conditionally and that,
for each fixed $(s,t)$, there are only finitely many
nontrivial differentials entering with target in the $(s,t)$ spot.
Thus, we eventually have
$E_r^{s,t} \supseteq E_{r+1}^{s,t}$. Suppose further that
$\mathrm{lim}^1_rE_r^{s,t} = 0$ for each $(s,t)$. Then
$E_r$ converges strongly to $\pi_*X$.
\end{theorem}

\subsection{Descent towers}\label{ssec:descent-sseq}

Let $B$ be a connective, commutative algebra object in
$\mathcal{C}$. Then we may form the descent tower
functor (see, e.g., \cite[\S B-C]{burklund-hahn-senger}):
	\[
	\mathrm{desc}_B: \mathcal{C} \to \mathsf{Fun}(\mathbb{Z}^{\mathrm{op}},
	\mathcal{C}),
	\]
which is lax symmetric monoidal, and specified by
	\[
	\mathrm{desc}_B^{\ge j}(X):=
	\lim_{\Delta}( \tau_{\ge j}(X \otimes B^{\otimes \bullet+1})).
	\]
When $X$ is bounded below,
this yields a conditionally convergent spectral sequence
	\[
	E_2^{s,t} = 
	H^s(\pi_t(X) \to \pi_t(X\otimes B) \to \cdots) \Rightarrow
	\pi_{t-s}X^{\wedge}_B,
	\]
where $X^{\wedge}_B = \lim (X\otimes B^{\otimes \bullet})$.
When $\pi_*(B\otimes B)$ is flat over $\pi_*B$, we can further
identify the $E_2$-term with $\mathrm{Ext}$ in the category
of comodules over the Hopf algebroid 
$(\pi_*B, \pi_*(B\otimes B))$.

\begin{remark}\label{rmk:descent-along} When $\mathcal{C} = \mathsf{Mod}_A$
and $B$ is regarded as a commutative algebra object through a map
$A \to B$ of commutative algebras,
we will often refer to the above procedure as `descent along
the map $A \to B$'.
\end{remark}

\begin{remark}\label{rmk:descent-weaker}
It is possible to make this construction
with much weaker hypotheses (at the cost of losing some
multiplicative structure). For example, the cobar construction
and construction of the descent tower
makes sense when $\mathcal{C} = \mathsf{Mod}_A$ for
a connective $\mathbb{E}_2$-ring $A$ and $B$
a connective $\mathbb{E}_1$-$A$-algebra.
\end{remark}

\begin{remark} The tower $\mathrm{desc}_B^{\ge *}(X)$
is not the usual Adams tower, but rather its \emph{d\'ecalage}
(compare \cite[1.3.3, 1.3.4]{deligne}
and \cite[\S II]{hedenlund}),
which is why its associated graded has homotopy groups
corresponding to the $E_2$-page of the Adams spectral
sequence rather than the $E_1$-page.
\end{remark}

\begin{warning} The descent tower shears the
filtration in the Adams spectral sequence. If we fix $t-s=n$, then
contributions to Adams filtration $s$ come from
$\mathrm{desc}^{\ge s+n}$. So, for example, a \emph{horizontal}
vanishing line on, say, the $E_2$-term of the Adams spectral sequence
would correspond to behavior in the descent filtration that is more
like a vanishing line of slope 1. Of course, if one is only interested in
a finite range of values of $n$, there is no difference.
\end{warning}

This story is especially well-behaved when $\mathrm{fib}(\mathbf{1} \to B)$
is $1$-connective.

\begin{proposition}\label{prop:desc-convergence} Suppose that
$I = \mathrm{fib}(\mathbf{1} \to B)$ lies in $\tau_{\ge 1}\mathcal{C}$.
Then, for any $d$-connective object $X$,
the descent tower has the
following properties:
	\begin{enumerate}[{\rm (a)}]
	\item The natural map $X \to X^{\wedge}_B$ is an equivalence.
	\item $E_2^{s,t}$ vanishes when $2s-t\ge d$. 
	\item $\pi_n\mathrm{desc}_B^{\ge j}(X)=0$ whenever
	$j\ge d+2n$.
	\item For each $k$, there exists an $N$ such that, for
	$j\ge N$, $\mathrm{desc}^{\ge j}_B(X)$ is $k$-connective.
	\end{enumerate}
\end{proposition}
\begin{proof} Since $\lim_j \mathrm{desc}^{\ge j}_B(X) = 0$, 
we can study the vanishing of the homotopy groups 
of each $\mathrm{desc}^{\ge j}_B(X)$ by establishing a
vanishing range in the associated graded.
Thus $(b) \Rightarrow (c) \Rightarrow (d)$,
so we need only establish (a) and (b). But these claims can be proven
using the usual construction of the descent spectral sequence,
via the tower $\{\mathrm{Tot}^{\le s}(B^{\bullet +1} \otimes X)\}$,
where the result is clear. 
\end{proof}

\subsection{Classical Adams spectral sequence}

The classical Adams spectral sequence, given by descent along
$S^0 \to \mathbb{F}_p$, has slightly more involved convergence issues
since the fiber of the unit map $S^0 \to \mathbb{F}_p$ is not
$1$-connective. We review the classical approach to getting around this
issue and leverage this to understand the convergence behavior
of the Tate fixed point spectral sequence below.

Throughout this section $\mathrm{desc}(-) = \mathrm{desc}_{\mathbb{F}_p}(-)$.

\begin{construction} Since $\mathrm{desc}(-)$ is lax symmetric monoidal,
every descent tower is a module over $\mathrm{desc}(S^0)$.
Recall that the element $p \in \pi_0(S^0)$ is detected in Adams filtration
$1$, and hence lifts to an element
$v_0 \in \pi_0\mathrm{desc}^{\ge 1}(S^0)$. Thus, given any
spectrum $X$, we have a natural map
	\[
	v_0: \mathrm{desc}(X)(1) \to \mathrm{desc}(X),
	\]
where, for a filtered spectrum $Y$, $Y(j)$ refers
to the filtered spectrum with $Y(j)^{\ge s} = Y^{\ge s-j}$.
\end{construction}

\begin{remark} The composite of the shift operator with $v_0$
is multiplication by $p$. It follows that $v_0$ induces multiplication
by $p$ on both $\colim X$ and $\lim X$.
\end{remark}

\begin{remark} There is a canonical identification
$\mathrm{desc}(X)/v_0 \simeq \mathrm{desc}(X/p)$.
However, when $k\ge 2$, $\mathrm{desc}(X)/v_0^k$
and $\mathrm{desc}(X/p^k)$ differ. The former
tower has $E_2$-term computed by the homotopy groups
of an object of the derived category of $\mathcal{A}_*$-comodules
which does not lie in the heart.
\end{remark}

\begin{proposition}\label{prop:adams-tower-connectivity} 
Let $X$ be $d$-connective. Then
$\mathrm{desc}(X)/v_0^m$ has the property that,
for each $k$, there is an $N$ such that,
for all $j\ge N$, $\mathrm{desc}^{\ge j}(X)/v_0^m$
is $k$-connective. Moreover, each term
$\mathrm{desc}^{\ge j}(X)/v_0^m$ is $d$-connective.
\end{proposition}
\begin{proof} The conclusion about the tower is stable under
extensions, so we are reduced to the case when $m=1$
and $\mathrm{desc}(X)/v_0 = \mathrm{desc}(X/p)$. 
Since the tower is conditionally convergent,
it suffices to establish
a vanishing line on the $E_2$-page, and to show this is concentrated
in stems starting in dimension $d$.
The $E_2$-page is computed by
$\mathrm{Ext}_{\mathcal{A}_*}^{s,t}(
H_*(X) \otimes \Lambda(\tau_0))$, which
classically has the desired vanishing line
(see \cite[Theorem 2.1]{adams-periodicity} at the prime 2,
and \cite[Proposition 2]{liulevicius} at odd primes.)
\end{proof}

\subsection{Fixed point spectral sequences}

Given a tower $X$ in the category of spectra with an
action of a group $G$, we can take homotopy fixed points,
orbits, or Tate fixed points levelwise and produce a new tower.
In this section we establish some criteria for the 
conditional convergence
of this tower.

\begin{proposition}\label{prop:hfpss-converge} Suppose $X\in
\mathsf{Fun}(\mathbb{Z}^{\mathrm{op}}, \mathsf{Fun}(\mathrm{BG}, \mathcal{C}))$
is conditionally convergent (i.e. $\lim X = 0$). Then so
is $X^{hG}$.
\end{proposition}
\begin{proof} Limits commute with limits.
\end{proof}

The analogous result for Tate fixed points requires a proposition.
We are grateful to the referee for pointing out the following result
which simplifies our earlier treatment of convergence in this section and the
next.

\begin{proposition}\label{prop:bounded-tate}
Let $Y$ be a filtered $G$-spectrum which
is uniformly bounded below.
Then \[(\lim Y)^{tG} \to \lim (Y^{tG})\] is an equivalence.
\end{proposition}
\begin{proof} Without loss of generality we may assume that
$Y$ is uniformly $0$-connective, and hence that $\lim Y$ is $(-1)$-connective.
If $Z$ is a $G$-spectrum, then
$Z_{hG}$ is computed as $\colim_{\Delta^{op}} (G_+^{\otimes \bullet} \otimes Z)$.
Denote by $\mathrm{sk}_r(Z_{hG})$ the colimit over
$\Delta_{\le r}^{op}$. If $Z$ is $(-1)$-connective,
then the cofiber $(Z_{hG})/(\mathrm{sk}_rZ_{hG})$ is $(r-1)$-connective.
We learn that, in the diagram
	\[
	\xymatrix{
	\mathrm{sk}_r(\lim Y)_{hG} \ar[r] \ar[d]& \lim \mathrm{sk}_rY_{hG}\ar[d]\\
	(\lim Y)_{hG} \ar[r] & \lim (Y_{hG})
	}
	\]
the vertical arrows are an equivalence in a range increasing with $r$, and
the upper horizontal arrow is always an equivalence since 
$\lim (-)$ commutes with finite colimits. It follows that
	\[
	(\lim Y)_{hG} \to \lim (Y_{hG})
	\]
is an equivalence, and hence so is $(\lim Y)^{tG} \to \lim (Y^{tG})$. 
\end{proof}

\begin{corollary} If $Y$ is a uniformly bounded below, conditionally
convergent tower, then $Y^{tG}$ is also conditionally convergent.
\end{corollary}

\subsection{Hochschild homology of filtered rings}\label{ssec:hh-of-filtered}

If $A$ is a filtered $\mathbb{E}_1$-ring,
one can construct a corresponding filtration of
$\mathrm{THH}(A)$ and spectral sequence
(see \cite{angelini-knoll-salch-maysseq}). We will need to
understand how this spectral sequence interacts with
the Tate-valued Frobenius, and for this we need
a construction of $\mathrm{THH}(A)$ as a filtered
cyclotomic object. We refer the reader to \cite[\S A]{ammn-fiber}
for details, and review the relevant definitions here.

\begin{definition} Let $L_p: 
\mathsf{Fun}(\mathbb{Z}^{\mathrm{op}}, \mathsf{Sp})
\to \mathsf{Fun}(\mathbb{Z}^{\mathrm{op}}, \mathsf{Sp})$
denote left Kan extension along multiplication by
$p$.
\end{definition}

\begin{proposition} (\cite[\S A]{ammn-fiber}) Let $A$ be a filtered
or graded
$\mathbb{E}_1$-ring. Then $\mathrm{THH}(A)$ admits a natural
$L_p$-twisted diagonal, i.e. an $S^1$-equivariant map
	\[
	\varphi: L_p\mathrm{THH}(A) \to \mathrm{THH}(A)^{tC_p}.
	\]
In the filtered case, this map is compatible with passage
to the associated graded and, in both cases, the map
is compatible with forgetting to underlying objects.
\end{proposition}

\begin{remark} Since $L_p$ is adjoint to restriction along
multiplication by $p$, the Frobenius gives $S^1$-equivariant maps
	\[
	\varphi: \mathrm{THH}(A)^{\ge j} \to
	(\mathrm{THH}(A)^{tC_p})^{\ge jp}
	\]
for all $j$; and similarly for the graded case.
\end{remark}

In particular, this produces maps of spectral sequences
(which shear the gradings). We will be using these spectral
sequences in the case when we are filtering $A$ by
its descent tower for $S^0 \to \mathbb{F}_p$.
The following proposition guarantees convergence
(after $p$-completion) when $A$ is connective.

\begin{proposition}\label{prop:filtered-thh-converges} Let $A$ be a connective $\mathbb{E}_1$-ring.
Then, for each $1\le m\le \infty$, the tower
$\mathrm{THH}(\mathrm{desc}_{\mathbb{F}_p}(A))^{tC_{p^m}}$,
converges conditionally to $\mathrm{THH}(A)^{tC_{p^m}}$.
The tower $\mathrm{THH}(\mathrm{desc}_{\mathbb{F}_p}(A))^{\wedge}_{v_0}$ converges
conditionally to $\mathrm{THH}(A)^{\wedge}_p$. 
\end{proposition}
\begin{proof}
Since $\mathrm{THH}(\mathrm{desc}(A))$ is
uniformly bounded below, we have
	\[
	\lim \mathrm{THH}(\mathrm{desc}(A))^{tC_{p^m}}
	=
	(\lim \mathrm{THH}(\mathrm{desc}(A)))^{tC_{p^m}}
	\]
by virtue of Proposition \ref{prop:bounded-tate}.

On bounded below $C_{p^k}$-spectra $Z$
we have that $Z^{tC_{p^k}} = (Z^{\wedge}_p)^{tC_{p^k}}$
(see \cite[Lemma I.2.9, Lemma II.4.9]{nikolaus-scholze}). So to prove
conditional convergence it will
suffice to show that 
	\[
	(\lim \mathrm{THH}(\mathrm{desc}(A)))^{\wedge}_p = 0.
	\]
We have
	\[
	(\lim \mathrm{THH}(\mathrm{desc}(A)))/p
	= \lim (\mathrm{THH}(\mathrm{desc}(A))/v_0),
	\]
so it suffices to show that this vanishes. This, in turn,
will prove the second
claim that $\mathrm{THH}(\mathrm{desc}(A))^{\wedge}_{v_0}$
converges to $\mathrm{THH}(A)^{\wedge}_p$. 

We recall that $\mathrm{THH}(
\mathrm{desc}(A))$ is defined as
the geometric realization of a simplicial object with terms
$\mathrm{desc}(A)^{\otimes \bullet +1}$, and
therefore
$\mathrm{THH}(
\mathrm{desc}A)/v_0$ is computed as the
geometric realization of a simplicial object with terms
$\mathrm{desc}(A)/v_0 \otimes 
\mathrm{desc}(A)^{\otimes \bullet}$.
Observe that, if $Z_{\bullet}$ is any simplicial spectrum with each $Z_i$ connective,
then, by the Dold-Kan correspondence,
$\mathrm{sk}_r|Z_{\bullet}|/\mathrm{sk}_{r-1}|Z_{\bullet}|$ is a summand of
$\Sigma^rZ_r$ and hence must be $r$-connective. Thus, to check that
$\lim \mathrm{THH}(\mathrm{desc}(A))/v_0=0$ we need only check
that $\lim \mathrm{sk}_r\mathrm{THH}(\mathrm{desc}(A))/v_0 = 0$
for all $r$. Since this skeleton is a finite colimit, we are reduced to checking that
$\lim(\mathrm{desc}(A)/v_0 \otimes 
\mathrm{desc}(A)^{\otimes n})=0$ for all $n$. In fact, the terms in this
tower become increasingly connective by Proposition 
\ref{prop:adams-tower-connectivity}, so the result is proved.
\end{proof}

\bibliographystyle{amsalpha}
\bibliography{Bibliography}

\end{document}